\documentclass[11pt,           
final,                         
english                       
]{article}

\usepackage[english,german,strings]{babel}     
\usepackage{amsmath}           
\usepackage[toc,page]{appendix}
\usepackage[utf8]{inputenc}    
\usepackage[T1]{fontenc}       
\usepackage{longtable}         
\usepackage{exscale}           
\usepackage[final]{graphicx}   
\usepackage{array}             
\usepackage{fancyhdr}          
\usepackage[a4paper]{geometry} 
\usepackage{gitinfo2} 
\usepackage[multiuser]{fixme}  
\usepackage{xspace}            
\usepackage{tikz}              
\usepackage{ifdraft}           
\usepackage{nchairx} 
\usepackage[expansion=false    
           ]{microtype}        
\usepackage[nottoc]{tocbibind} 
\usepackage{csquotes}
\usepackage[
style=alphabetic,
]{biblatex}
\usepackage[
final=true,                    
plainpages=false,              
pdfpagelabels=true,            
pdfencoding=auto,              
unicode=true,                  
hypertexnames=true,            
naturalnames=true              
]{hyperref}                    
\usepackage{bookmark}

\graphicspath{{../tikz/pic/}{../tikz/plot/}}

\usetikzlibrary{arrows,
  intersections,
  matrix,
  calc,
  fadings,
  decorations.pathreplacing,
  positioning,patterns,
  decorations.markings,
  decorations.pathmorphing,
  cd,
  babel}

\ifdraft{\synctex=1}{}

\fxusetheme{color}

\ifdraft{
    \fancyfoot[L]{\footnotesize hkr: A New Look at the HKR Theorem}
    \fancyfoot[R]{\footnotesize \gitAuthorIsoDate{} Hash: \texttt{\gitAbbrevHash}}
}{}
\FXRegisterAuthor{md}{anmd}{marvin}
\FXRegisterAuthor{ce}{ance}{chiara}
\FXRegisterAuthor{js}{anjs}{jonas}
\FXRegisterAuthor{sw}{answ}{stefan}

\newcommand{\myemail}{\texttt{stefan.waldmann@mathematik.uni-wuerzburg.de}}
\newcommand{\myaddress}{
        \begin{minipage}{8cm}
            \centering\small
            Julius Maximilian University of Würzburg \\
            Institute of Mathematics \\
            Emil-Fischer-Straße 31 \\
            97074 Würzburg \\
            Germany
        \end{minipage}
        \\
        }

\newcommand{\AuthorOne}{\textbf{Marvin Dippell}}
\newcommand{\AuthorTwo}{\textbf{Chiara Esposito}}
\newcommand{\AuthorThree}{\textbf{Jonas Schnitzer}}
\newcommand{\AuthorFour}{\textbf{Stefan Waldmann}}

\newcommand{\AuthorAddressOne}{
        \begin{minipage}{8cm}
        \centering\small
        Department Mathematik/Informatik\\
        Universität zu Köln\\
        Weyertal 86-90\\
        50931 Köln\\
        Germany
    \end{minipage}
    \\}
\newcommand{\AuthorAddressTwo}{
        \begin{minipage}{8cm}
            \centering\small
            Dipartimento di Matematica\\
            Università degli Studi di Salerno\\
            via Giovanni Paolo II, 132\\
            84084 Fisciano (SA)\\
            Italy
        \end{minipage}
        \\
        }
\newcommand{\AuthorAddressThree}{
        \begin{minipage}{8cm}
            \centering\small
            Dipartimento di Matematica "Felice Casorati"\\
            Università degli Studi di Pavia\\
            Via Ferrata 5\\
            27100 Pavia\\
            Italy
        \end{minipage}
        \\
}
\newcommand{\AuthorAddressFour}{\myaddress}

\newcommand{\AuthorEmailOne}{\texttt{mdippell@unisa.it}}
\newcommand{\AuthorEmailTwo}{\texttt{chesposito@unisa.it}}
\newcommand{\AuthorEmailThree}{\texttt{jonaschristoph.schnitzer@unipv.it}}
\newcommand{\AuthorEmailFour}{\myemail}

\author{\AuthorOne\thanks{\AuthorEmailOne},\\[0.2cm]
        \AuthorAddressOne
        \\[-0.1cm]
        \AuthorTwo\thanks{\AuthorEmailTwo},\\[0.2cm]
        \AuthorAddressTwo
        \\[-0.1cm]
        \AuthorThree\thanks{\AuthorEmailThree},\\[0.2cm]
        \AuthorAddressThree
        \\[-0.1cm]
        \AuthorFour\thanks{\AuthorEmailFour}\\[0.2cm]
        \AuthorAddressFour
}

\makeatletter

\newcommand{\bibnote}[2]{\nocite{#1}\@namedef{#1chairxnote}{#2}}

\usepackage{etoolbox}
\makeatletter
\@ifpackageloaded{hyperref}{
  \patchcmd{\hyper@makecurrent}{%
      \ifx\Hy@param\Hy@chapterstring
          \let\Hy@param\Hy@chapapp
      \fi
  }{%
      \iftoggle{inappendix}{
          \@checkappendixparam{chapter}%
          \@checkappendixparam{section}%
          \@checkappendixparam{subsection}%
          \@checkappendixparam{subsubsection}%
          \@checkappendixparam{paragraph}%
          \@checkappendixparam{subparagraph}%
      }{}%
  }{}{\errmessage{failed to patch}}

  \newcommand*{\@checkappendixparam}[1]{%
      \def\@checkappendixparamtmp{#1}%
      \ifx\Hy@param\@checkappendixparamtmp
          \let\Hy@param\Hy@appendixstring
      \fi
  }
  \makeatletter

  \newtoggle{inappendix}
  \togglefalse{inappendix}

  \apptocmd{\appendix}{\toggletrue{inappendix}}{}{\errmessage{failed to patch appendix}}
  \@ifpackageloaded{appendix}{%
    \apptocmd{\subappendices}{\toggletrue{inappendix}}{}{\errmessage{failed to patch subappendices}}
  }{}
}{}

\renewcommand{\to}{%
  \relax\if@display
    \expandafter\longrightarrow
  \else
    \expandafter\rightarrow
  \fi
}
\makeatother

\newcommand{\vanishing}{\mathcal{J}}

\newcommand{\HC}{\operator{HC}}
\newcommand{\HCdiff}{\HC_{\script{diff}}}
\newcommand{\HH}{\operator{HH}}
\newcommand{\HHdiff}{\HH_{\script{diff}}}
\newcommand{\hkr}{\operator{hkr}}
\newcommand{\hkrInv}{\hkr^{-1}}
\newcommand{\hkrHomo}{\Theta^\nabla}

\makeatletter
\newcommand{\SymSec}{\ch@irxspacefont{S}}
\makeatother

\newcommand{\redSym}{\cc{\Sym}}
\newcommand{\Ca}{{\script{ca}}}
\newcommand{\CCa}{\operator{C}_\Ca}
\newcommand{\redCCa}{\cc{\operator{C}}_\Ca}
\newcommand{\HCa}{\operator{H}_\Ca}
\newcommand{\redHCa}{\cc{\operator{H}}_\Ca}
\newcommand{\HCainv}[1]{\operator{H}_{\Ca,#1}}

\newcommand{\CCE}{\operator{C}_\CE}

\newcommand{\CVanEst}{\operator{D}_{\script{vE}}}
\newcommand{\redCVanEst}{\cc{\operator{D}}_{\script{vE}}}

\newcommand{\HVanEst}{\operator{H}_{\script{vE}}}
\newcommand{\VanEstDiff}{\operator{VE}}
\newcommand{\VanEstInt}{\operator{VE}^{-1}}
\newcommand{\VanEstHomo}{\Theta}

\renewcommand{\Op}{\operatorname{Op}^\nabla}

\makeatletter
\newcommand{\ch@irxcoalgebrafont}{\mathcal}
\newcommand{\coalgebra}[1]{\ch@irxcoalgebrafont{#1}}
\makeatother

\newcommand{\coproduct}{\Delta}
\newcommand{\counit}{\varepsilon}
\newcommand{\Perm}{\operator{S}}
\newcommand{\Shuffle}{\operator{Sh}}

\newcommand{\shcoprod}{\coproduct_{\script{sh}}}
\newcommand{\redshcoprod}{\overline{\coproduct}_{\script{sh}}}

\newcommand{\Left}{\mathrm{L}}

\newcommand{\Double}{\mathrm{D}}

\newcommand{\imap}{\mathrm{i}}
\newcommand{\Imap}{\mathrm{I}}
\newcommand{\jmap}{\mathrm{j}}
\newcommand{\pmap}{\mathrm{p}}
\newcommand{\Pmap}{\mathrm{P}}
\newcommand{\qmap}{\mathrm{q}}
\newcommand{\Qmap}{\mathrm{Q}}
\newcommand{\hmap}{\mathrm{h}}
\newcommand{\Hmap}{\mathrm{H}}
\newcommand{\kmap}{\mathrm{k}}
\newcommand{\Kmap}{\mathrm{K}}
\newcommand{\bmap}{\mathrm{b}}
\newcommand{\Bmap}{\mathrm{B}}

\title{Global Homotopies for Differential Hochschild Cohomologies}

\date{
    \ifdraft{
        Current Version of hkr: \gitAuthorIsoDate\\[0.2cm]
        {\small
            Last changes by \gitAuthorName{} on \gitAuthorDate \\
            Git revision of hkr: \texttt{\gitAbbrevHash{}} \gitReferences%
        }%
    }{}
}

\begin{document}

\selectlanguage{english}

\maketitle

\begin{abstract}
    We construct explicit global homotopies for differential Hochschild cochains in differential geometry, thereby upgrading the classical Hochschild–Kostant–Rosenberg map to a deformation retract. Our approach combines two key techniques: a symbol calculus from differential geometry and a coalgebraic version of the van Est theorem. To demonstrate its effectiveness, we develop deformation retracts in several related settings, including principal bundles and invariant contexts. As a byproduct, we recover the classical Hochschild–Kostant–Rosenberg theorem and compute previously inaccessible Hochschild cohomologies.
\end{abstract}

\newpage

\tableofcontents
\thispagestyle{plain}
\newpage

\section{Introduction}
\label{sec:Introduction}

Hochschild cohomology encodes important information about associative
algebras, in particular it is used to understand their structure,
deformations, and representations.  To name just a few important
applications, Hochschild cohomology appears in algebraic geometry,
algebraic topology, category theory as well as in functional analysis.
In particular, in algebraic deformation theory in the sense of
Gerstenhaber \cite{gerstenhaber:1963a,gerstenhaber:1964a,
  gerstenhaber:1966a, gerstenhaber:1968a, gerstenhaber:1974a}
Hochschild cohomology controls the obstructions to existence and
equivalence of formal associative deformations of an algebra.  Equally
important is its appearance in noncommutative geometry
\cite{connes:1994a} with its relations to cyclic cohomology. Here it
is the arena to formulate various types of index
theorems \cite{nest.tsygan:1995a}.  It is thus of utmost interest to
have a good understanding of Hochschild cohomology of the algebras
under consideration. However, explicitly computing cohomologies can be
difficult and requires more insight into the specific problem.  The
Hochschild-Kostant-Rosenberg theorem (HKR theorem for short)
determines Hochschild cohomology in the case of an algebra of
functions as the multivector fields.  In its original version
\cite{hochschild.kostant.rosenberg:1962a} the focus was on regular
affine algebras.  But ever since, this statement has been extended in
various directions.

In many applications not only the cohomology but the complex itself is
of crucial importance. The Hochschild complex carries the structure of
a differential graded Lie algebra with the Gerstenhaber bracket as Lie
bracket.  In the celebrated formality theorem of Kontsevich
\cite{kontsevich:1997:pre, kontsevich:2003a} the Lie bracket in
(differential) Hochschild cohomology is related to the Gerstenhaber
bracket of the complex by means of an $L_\infty$-quasi isomorphism, a
fundamental result to deformation theory.

If the algebra under consideration carries additional structures this
is usually reflected on the corresponding Hochschild complex.  In the
case of a smooth manifold $M$ the algebra $\Cinfty(M)$ of smooth
functions carries a Fréchet space structure. Thus it is meaningful to
require Hochschild cochains to be continuous. For this scenario,
Connes computed the continuous Hochschild cohomology under the
assumption that $M$ admits a nowhere vanishing vector field in
\cite{connes:1986a}, see also \cite[Sect.~3.2$\alpha$]{connes:1994a},
using a suitable topological bar and Koszul resolution. This has been
extended to general smooth manifolds by Pflaum \cite{pflaum:1998a},
see also \cite{gutt:1997a} and \cite{nadaud:1999a}.

Deformation quantization is one of the major applications of algebraic
deformation theory \cite{bayen.et.al:1978a}.  In this realm,
continuous cochains are often considered to be still too general.
Instead, one is interested in local or multidifferential cochains. By
Peetre's theorem \cite{peetre:1960a, peetre:1959a}, local cochains are
at least locally multidifferential \cite{cahen.gutt.dewilde:1980a}.
Therefore the focus is on multidifferential cochains and their
differential Hochschild cohomology.  Here Vey gave a first proof of
the HKR theorem in \cite{vey:1975a}, Gutt gave an explicit (recursive)
local formula for the primitive of an exact cochain in her thesis
\cite{gutt:1980b}, later extended by Cahen, DeWilde, and Gutt also to
local cochains in \cite{cahen.gutt.dewilde:1980a}.  While in these
results no homotopies have been provided, in
\cite{dewilde.lecomte:1995a} DeWilde and Lecomte gave a proof of the
HKR theorem by constructing local homotopies.  The overall result is
that the differential Hochschild cohomology $\HHdiff(M)$ of
$\Cinfty(M)$ is canonically given by the multivector fields
$\Anti^\bullet\Secinfty(TM)$ on $M$ with the Gerstenhaber bracket
being the Schouten bracket and the cup product being the usual wedge
product. The canonical inclusion of multivector fields as
multidifferential operators, called the HKR map, becomes the
isomorphism in cohomology. A purely algebraic version based on the
relation to the Koszul resolution was given by Kassel in
\cite[Thm.~XVIII.7.1]{kassel:1995a}.

\medskip

While this settles the computation of the Hochschild cohomology of
$\Cinfty(M)$, in many applications it is desirable to have a canonical
choice of primitives for classes in Hochschild cohomology given by a
(global) homotopy.  Moreover, there are further applications in need
of Hochschild cohomologies with various $\Cinfty(M)$-bimodules as
coefficients. One application is the deformation of modules where the
Hochschild cohomology with values in the endomorphisms of the module
controls the deformation theory. Such cohomologies have successfully
been computed in various contexts like
\cite{bordemann.neumaier.waldmann.weiss:2010a, bordemann:2005a,
  bordemann.et.al:2005a:pre, hurle:2018a:pre}, see also
\cite{dippell.esposito.waldmann:2022a} for a recent set-up in the
context of coisotropic reduction.

Finally, in various deformation problems the original algebra carries
some group action by automorphisms relevant for the problem. Thus it
is natural to require the deformations to respect this symmetry as
well. This leads to the notion of invariant Hochschild
cohomology. First results for the explicit computation of invariant
Hochschild cohomologies have been obtained for proper Lie group
actions on a manifold in \cite{miaskiwskyi:2021a} by averaging
procedures.

\medskip

\emph{The aim of this paper is to present a conceptual and adaptable construction of a deformation retract for differential Hochschild cochains (with coefficients). This approach not only refines existing results but also provides a framework for computing Hochschild cohomology in a wide range of important cases.}

\medskip

The strategy is surprisingly simple concerning the
techniques needed to arrive at global and explicit
homotopies. Essentially, we need to proceed in two steps:
\begin{enumerate}
\item \label{item:PartII} In a first step, we identify the
    differential Hochschild complex $\HCdiff^\bullet(M)$ with certain tensor
    fields by means of a global symbol calculus $\Op$. This requires
    to use a torsion-free covariant derivative $\nabla$ on the
    manifold $M$ and perhaps additional covariant derivatives
    depending on the coefficients one is interested in. The usage of
    covariant derivatives to obtain a global symbol calculus is a
    classic technique one can trace back at least to \cite[Chap.~IV,
    §9]{palais:1965a}. The version we use can be found e.g. in
    \cite[App.~A]{waldmann:2007a}.

    Once we replaced differential Hochschild cochains by their global
    symbols the Hochschild differential becomes the differential of
    the coalgebra complex $\CCa^\bullet(\Secinfty(TM))$ of the symmetric
    coalgebra $\Sym\Secinfty(TM)$ and we obtain an isomorphism
    \begin{equation}
        \label{eq:intro_SymbolCalculus}
        \begin{tikzcd}[column sep = huge]
            \CCa^\bullet(\Secinfty(TM))
            \arrow[r,"\Op","\simeq"{swap}]
            & \HCdiff^\bullet(M)
        \end{tikzcd}
    \end{equation}
    of differential graded algebras.  This way we translated the
    problem of computing Hochschild cohomology to the purely algebraic
    task of finding the associated coalgebra cohomology.

\item \label{item:PartI} To compute coalgebra cohomology
    $\CCa^\bullet(V)$ for a module $V$ over a commutative ring
    $\ring{R}$ we will use techniques from Lie group and Lie algebra
    cohomology.  Here the van Est double complex is a very powerful
    and far-reaching method, see \cite{vanest:1953a, vanest:1953b} for
    the original formulation and \cite{meinrenken.salazar:2020a} for a
    more recent version which was the main source of inspiration for
    us.  In our situation we can interpret $\CCa^\bullet(V)$ as the
    complex of polynomial group cohomology on the (hypothetical)
    predual of $V$ considered as an \emph{abelian} Lie group.  Then
    the classical van Est theorem should identify this cohomology with
    the Chevalley-Eilenberg cohomology of the \emph{abelian} Lie
    algebra on this predual.

    This was already observed for \emph{local} symbols based on the
    flat covariant derivative in a local chart in
    \cite{cahen.gutt.dewilde:1980a} and used to compute Hochschild
    cohomologies locally with a subsequent gluing using partitions of
    unity.

    However, we have to go beyond the finite-dimensional case and
    hence preduals do not necessarily exist.  Therefore we decide to
    work in a dual picture from the beginning.  The van Est double
    complex $\CVanEst^{\bullet,\bullet}(V)$ simplifies drastically in
    this formulation and it is connected to $\CCa^\bullet(V)$ and the
    Chevalley-Eilenberg complex $\CCE^\bullet(V) = \Anti^\bullet V$ of
    the abelian Lie coalgebra $V$ by means of deformation retracts:
    \begin{equation*}
        \begin{tikzcd}[column sep = huge]
            \CCE^\bullet(V)
            \arrow[r,shift left = 2pt]
            & \CVanEst^{\bullet}(V)
            \arrow[l,shift left = 2pt]
            \arrow[r,shift left = 2pt]
            & \CCa^\bullet(V)
            \arrow[l,shift left = 2pt]
        \end{tikzcd}
    \end{equation*}
    The main technique is to use the homological perturbation lemma as
    it can be found in e.g. \cite{crainic:2004a:pre,
      meinrenken.salazar:2020a}, see also \cite{huebschmann:2011a} for
    a historical overview.  Combining these deformation retracts we
    obtain the van Est deformation retract
    \begin{equation}
        \label{eq:intro_VanEst}
        \begin{tikzcd}[column sep = huge]
            \CCE^\bullet(V)
            \arrow[r,"\VanEstInt", shift left = 3pt]
            & \CCa^{\bullet}(V)
            \arrow[l,"\VanEstDiff", shift left = 3pt]
            \arrow[loop,
            out = -30,
            in = 30,
            distance = 30pt,
            start anchor = {[yshift = -7pt]east},
            end anchor = {[yshift = 7pt]east},
            "\VanEstHomo"{swap}
            ]
        \end{tikzcd}.
    \end{equation}
    This concludes the computation of the coalgebra cohomology as the
    differential on $\CCE^\bullet(V)$ is trivial and therefore its
    cohomology is given by $\Anti^\bullet V$.

    All structures we develop in this part are completely canonical
    thus enjoying good functorial properties and can be adapted easily
    to a large class of coefficient modules as well. In fact, again
    the perturbation lemma proves to be extremely efficient to include
    modules into the picture.
    Moreover, we believe that this result alone has quite some relevance,
    since it can be directly applied to  compute certain cohomologies
    related to non-trivial Lie (co)algebras, see
    \autoref{rem:nontrivialLiecoalg}.
\end{enumerate}

By combining the symbol calculus \eqref{eq:intro_SymbolCalculus} with
the van Est deformation retract \eqref{eq:intro_VanEst} we arrive at a
two-fold result: an HKR theorem for differential Hochschild cohomology
together with a global homotopy such that
\begin{equation*}
    \begin{tikzcd}[column sep = large]
        \Anti^\bullet \Secinfty(TM)
        \arrow[r,"\hkr", shift left = 3pt]
        &\bigl( \HCdiff^{\bullet}(M),\delta \bigr)
        \arrow[l,"\hkrInv_\nabla", shift left = 3pt]
        \arrow[loop,
        out = -30,
        in = 30,
        distance = 30pt,
        start anchor = {[yshift = -7pt]east},
        end anchor = {[yshift = 7pt]east},
        "\hkrHomo"{swap}
        ]
    \end{tikzcd}
\end{equation*}
is a deformation retract.

The above procedure can easily be adapted to incorporate coefficient
modules.  This opens the way to obtain global homotopies for
Hochschild cohomologies in many related scenarios.  To illustrate how
our results can be used we apply the techniques in various different
situations which have been considered in the literature before,
computing some formerly unknown cohomologies. This results in the
following list of theorems, more of which can be obtained similarly:
\begin{cptitem}
\item For vector bundles $E$ and $F$ over $M$ we compute
    $\HCdiff^{\bullet}(M, E)$ and
    $\HCdiff^{\bullet}(M, \Diffop(E; F))$ in
    \autoref{thm:HKR_VectorValued} and
    \autoref{thm:HKR_DiffopEFValues}, respectively.
\item Given a submanifold $C \subseteq M$ we obtain the tangential
    Hochschild cohomology $\HCdiff^\bullet(M)_C$ as well as the
    Hochschild cohomology $\HCdiff^\bullet\bigl(M,\Diffop(C)\bigr)$ with
    values in $\Diffop(C)$, see \autoref{thm:HKR_Submanifold} and
    \autoref{thm:HKR_ValuesDiffopC}.
\item For a surjective submersion $\pr \colon P \to M$ the projectable
    Hochschild cohomology $\HCdiff^{\bullet}(P)^F$ as well as the
    Hochschild cohomology $\HCdiff^{\bullet}(M,\Diffop(P))$ with
    values in $\Diffop(P)$ is computed in
    \autoref{thm:HKR_Projectable} and
    \autoref{thm:HKR_surjSubmersions}.
\item Given a foliation $D \subseteq TM$ we compute its related
    Hochschild cohomology $\HC_D^{\bullet}(M)$ in
    \autoref{thm:HKR_Foliation}.
\item Considering invariant covariant derivatives and using the good
    functorial properties of the homotopies and the deformation
    retracts opens the door to invariant versions of all the previous
    HKR theorems.  For instance, we prove HKR theorems for arbitrary
    Lie group and Lie algebras actions preserving a covariant
    derivative in \autoref{thm:HKR_inv} as well as for principal
    bundles in \autoref{thm:HKR_PrincipalBundle}.
\end{cptitem}

Finally, we also remark that all our constructions are localizable and
hence of sheaf-theoretic nature. In fact, the global symbol calculus
is local in the sense that the global symbol of a multidifferential
operator at a point only depends on the restriction of it to an
arbitrary open neighbourhood of the point.  The subsequent algebraic
part is even pointwise. Thus our constructions can, in principle, also
be used in more algebraic frameworks, even though we refrain from
doing this in the present work.

\medskip

The paper is organized as follows: in
\autoref{part:CoalgebraCohomology} we compute coalgebra cohomology in
a purely algebraic setting.  To do this we collect in
\autoref{sec:CoalgebrasComodules} the required basics of the coalgebra
cohomology of symmetric algebras and Chevalley-Eilenberg cohomology of
abelian Lie coalgebras.  This sets the stage for a coalgebraic
formulation of the van Est double complex together with the van Est
differentiation and integration maps in
\autoref{sec:VanEstDoubleComplex}.  Functorial aspects are shown
before we conclude the first part by including coefficients into the
picture in \autoref{sec:VanEstComplexCoeff}.  In
\autoref{part:HochschildCohomology} we start by recalling the
differential Hochschild complex and establish the global symbol
calculus for multidifferential operators in
\autoref{sec:HochschildComplex}.  Finally, we combine the previous
results to obtain various HKR theorems based on explicit homotopies
and deformation retracts in \autoref{sec:HKRDifferentialGeometry}.  We
conclude with \autoref{sec:AdaptedCovariantDerivatives} and
\autoref{sec:HomologicalAlgebra}, the first on adapted covariant
derivatives, the second on the homological perturbation lemma.

\medskip
\noindent
\textbf{Acknowledgements:} It is a pleasure to thank Martin Bordemann
and Thomas Weber for valuable discussions.  Moreover, we want to thank
Luca Vitagliano and Olivier Brahic for helpful remarks. Furthermore,
we are very grateful to Ezra Getzler for explaining the relations to
his recent work and giving us detailed comments and remarks on our
construction.  M. D. and C. E. were supported by the National Group
for Algebraic and Geometric Structures, and their Applications (GNSAGA
– INdAM).

\part{Coalgebra Cohomology of the Symmetric Algebra}
\label{part:CoalgebraCohomology}

In this first part we study the relation of coalgebra cohomology with
Chevalley-Eilenberg cohomology of a symmetric coalgebra using the
so-called van Est double complex.

For this we fix an associative and commutative unital ring $\ring{R}$
with $\mathbb{Q} \subseteq \ring{R}$.  All modules are supposed to be
$\ring{R}$-modules and all maps will be $\ring{R}$-linear.  Moreover,
all unadorned tensor products will be tensor products over the ring
$\ring{R}$.  In the second part we will mainly be interested in the
case $\ring{R} = \Cinfty(M)$ of smooth real-valued functions on a
smooth manifold.

\section{(Lie) Coalgebras, their Comodules and Cohomology}
\label{sec:CoalgebrasComodules}

The aim of this section is now to give a quick overview on the
algebraic structures on the symmetric algebra $\Sym V$ of some
$\ring{R}$-module $V$.  We then define the coalgebra cohomology and
Chevalley-Eilenberg cohomology of $V$.

\subsection{The Bialgebra Structure of $\Sym V$}

This section is mainly meant to fix notation: we recall the classical
(graded) bialgebra structures on $\Sym V$ and $\Anti V$.  All this can
be found e.g. in \cite{kassel:1995a}.  To keep the notation as short
as possible we use \emph{Sweedler's notation} for the coproduct
$\Delta$ of a coalgebra $\algebra{C}$ and write for $c\in \algebra{C}$
\begin{equation}
    \Delta(c)
    =
    c_\sweedler{1}\tensor c_\sweedler{2}.
\end{equation}

In the following list, we collect the explicit formulas for the
coalgebra and bialgebra structures on the symmetric and anti-symmetric
algebras generated by an $\ring{R}$-module $V$.
\begin{enumerate}
\item \label{item:Coalgebra_SymV} Consider
    $\Sym V = \bigoplus_{r = 0}^\infty \Sym^r V$ with $\Sym^r V$ being
    the $r$-fold symmetric tensor power of $V$.  We can define a
    coproduct $\shcoprod \colon \Sym V \to \Sym V \tensor \Sym V$ by
    specifying it on generators by
    \begin{align}
        \label{eq:shcoproductOnGenerators}
        \shcoprod(\Unit) \coloneqq \Unit \tensor \Unit,\qquad
        \shcoprod(x) \coloneqq \Unit \tensor x + x \tensor \Unit.
    \end{align}
    Together with the canonical projection
    $\counit \colon \Sym V \to \Sym^0 V = \ring{R}$ as counit this
    defines the structure of an (ungraded) cocommutative coalgebra on
    $\Sym V$ when we extend $\shcoprod$ as algebra morphism with
    respect to the symmetric tensor product $\vee$.  For general
    $x_1 \vee \dots \vee x_n \in \Sym^r V$ we obtain the explicit
    formula
    \begin{equation}
        \shcoprod (x_1 \vee \cdots \vee x_r)
        =
        \sum_{i=0}^{r} \sum_{\sigma \in \Shuffle(i,r-i)}
        x_{\sigma(1)}
        \vee \cdots \vee
        x_{\sigma(i)}
        \tensor
        x_{\sigma(i+1)}
        \vee \cdots \vee
        x_{\sigma(r)}.
    \end{equation}
    Here $\Shuffle(i,r-i)\subset \Perm_r$ denotes the set of
    $(i,r-i)$-\emph{shuffles}, i.e. permutations $\sigma$ such that
    $\sigma(1) < \cdots < \sigma(i)$ and
    $\sigma(i+1) < \cdots < \sigma(r)$, where we set
    $\Shuffle(0,r) = \Shuffle(r,0) = \{\id\}$.  The map $\shcoprod$ is
    usually referred to as the \emph{shuffle coproduct}. Moreover,
    $\Sym V$ together with the shuffle coproduct $\shcoprod$ and
    counit $\counit$ , as well as the symmetric tensor product $\vee$
    and unit $\iota(1_\ring{R}) \coloneqq \Unit \in \Sym^0 V$ is a
    commutative cocommutative bialgebra.
\item \label{item:Coalgebra_SymVred} Starting from the shuffle
    coproduct $\shcoprod$ on $\Sym V$ we can define the \emph{reduced
      shuffle coproduct}
    $\redshcoprod \colon \Sym V \to \Sym V \tensor \Sym V $ by
    \begin{equation}
        \redshcoprod(\phi)
        \coloneqq
        \shcoprod(\phi) - \Unit \tensor \phi - \phi \tensor \Unit,
    \end{equation}
    for all $\phi \in \Sym V$.  This yields another cocommutative
    coalgebra structure on $\Sym V$.  More explicitly, we get for
    $\phi = x_1 \vee \dots \vee x_r \in \Sym^r V$ the formula
    \begin{equation}
        \label{eq:redcoproduct}
        \redshcoprod(x_1 \vee \cdots \vee x_r)
        =
        \sum_{i=1}^{r-1} \sum_{\sigma \in Sh(i,r-i)}
        x_{\sigma(1)}
        \vee \cdots \vee
        x_{\sigma(i)}
        \tensor
        x_{\sigma(i+1)}
        \vee \cdots \vee
        x_{\sigma(r)},
    \end{equation}
    for $r \geq 1$ and $\redshcoprod(\Unit) = -\Unit \tensor \Unit$.
    To avoid confusion we will not use Sweedler's notation for the
    reduced shuffle coproduct.
\item \label{item:Coalgebra_AntiV} Similar to the symmetric case we
    can also define a coproduct $\shcoprod$ on
    $\Anti^\bullet V = \bigoplus_{\ell=0}^\infty \Anti^\ell V$ by
    setting
    \begin{align}
        \label{eq:AntiShCoproductOnGenerators}
        \shcoprod(1) \coloneqq 1 \tensor 1,
        \qquad
        \shcoprod(v) \coloneqq 1 \tensor v + v \tensor 1
    \end{align}
    on generators of $\Anti^\bullet V$ and extend this again to be an
    algebra morphism, now with respect to $\wedge$.  Together with the
    canonical projection
    $\counit \colon \Anti^\bullet V \to \Anti^0 V = \ring{R}$ as
    counit this defines the structure of a graded cocommutative
    coalgebra on $\Anti^\bullet V$.  For general
    $\xi_1 \vee \dots \vee \xi_\ell \in \Anti^\ell V$ we obtain the
    explicit formula
    \begin{equation}
        \shcoprod (\xi_1 \wedge \cdots \wedge \xi_\ell)
        =
        \sum_{k=0}^{\ell} \sum_{\sigma \in \Shuffle(k,\ell-k)}
        \sign(\sigma) \cdot
        \xi_{\sigma(1)}
        \wedge \cdots \wedge
        \xi_{\sigma(k)}
        \tensor
        \xi_{\sigma(k+1)}
        \wedge \cdots \wedge
        \xi_{\sigma(\ell)}.
    \end{equation}
    The anti-symmetric tensor power $\Anti^\bullet V$ together with
    the shuffle coproduct $\shcoprod$ and counit $\counit$ as well as
    the anti-symmetric tensor product $\wedge$ and unit
    $\iota(1_\ring{R}) \coloneqq 1 \in \Anti^0 V$ is a graded
    commutative cocommutative bialgebra.
\end{enumerate}

We will denote by $\pr_+ \colon \Sym V \to \Sym V$ the projection onto
positive symmetric degrees, i.e. $\pr_+(\Unit) = 0$ and
$\pr_+(\phi) = \phi$ for all $\phi \in \Sym^n V$ with $n \geq 1$.
Note that it holds
\begin{equation}
    \label{eq:ProjectionOnShcorprod}
    (\pr_+ \tensor \pr_+) \circ \redshcoprod(\phi)
    =
    (\pr_+ \tensor \pr_+) \circ \shcoprod(\phi)
    =
    \pr_+(\phi_\sweedler{1}) \tensor \pr_+(\phi_\sweedler{2})
\end{equation}
for all $\phi \in \Sym V$.
The following observation will be useful in many situations.
\begin{lemma}
    \label{lem:exponentialproj}%
    Let $V$ be an $\ring{R}$-module.
    \begin{lemmalist}
    \item For $\phi \in \Sym V$ one has
        \begin{align}
            \pr_{\Sym^r V}(\phi)
            = \frac{1}{r!}
            \pr_V(\phi_\sweedler{1})
            \vee \cdots \vee
            \pr_V(\phi_\sweedler{r}).
        \end{align}
    \item For $\xi \in \Anti^\bullet V$ one has
        \begin{align}
            \pr_{\Anti^\ell V}(\xi)
            = \frac{1}{\ell!}
            \pr_V(\xi_\sweedler{1})
            \wedge \cdots \wedge
            \pr_V(\xi_\sweedler{\ell}).
        \end{align}
    \end{lemmalist}
\end{lemma}
\begin{proof}
    We prove the statement only for $\Sym V$, the computations for
    $\Anti^\bullet V$ is very similar.  Let
    $\phi = x_1 \vee \cdots \vee x_r \in \Sym^r V$ be arbitrary, then
    \begin{align*}
        \phi_\sweedler{1} \tensor \cdots \tensor \phi_\sweedler{n}
        &=
        \Delta^n(x_1 \vee \cdots \vee x_r)
        \\
        &=
        \sum_{i_1 + \cdots + i_n = r}
        \sum_{\sigma\in \Shuffle(i_1, \ldots, i_n)}
        x_{\sigma(1)}
        \vee \cdots \vee
        x_{\sigma(i_1)}
        \tensor \cdots \tensor
        x_{\sigma(i_1 + \cdots + i_{n-1}+1)}
        \vee \cdots \vee
        x_{\sigma(r)}.
    \end{align*}
    Here $\Shuffle(i_1, \dots, i_n)$ denotes the set of
    multi-shuffles, i.e. permutations $\sigma$ such that
    $\sigma(i_1 + \cdots + i_{j-1} + 1) < \cdots < \sigma(i_1 + \cdots
    + i_j)$ for all $j = 1, \ldots, n$.  This means in particular that
    $\pr_V(\phi_\sweedler{1}) \tensor \cdots \tensor
    \pr_V(\phi_\sweedler{n})$ is only non-zero, if $n=r$ and we have
    \begin{align*}
        \pr_V(\phi_\sweedler{1})
        \tensor \cdots \tensor
        \pr_V(\phi_\sweedler{r})
        =
        \sum_{\sigma \in \Shuffle(1, \ldots, 1) = \Perm_r}
        x_{\sigma(1)}
        \tensor \cdots \tensor
        x_{\sigma(r)}
    \end{align*}
    and therefore by applying $\vee$ on both sides we obtain
    \begin{align*}
        \pr_V(\phi_\sweedler{1})
        \vee \dots \vee
        \pr_V(\phi_\sweedler{r})
        =
        r! \cdot
        x_1
        \vee \cdots \vee
        x_r
    \end{align*}
    and the claim is proven.
\end{proof}

There is another structure available both on $\Sym V$ and on
$\Anti V$, namely the antipode $s$ defined by
\begin{align}
    \label{eq:AntipodeSV}
    s(x_1 \vee \dots \vee x_r)
    &\coloneqq
    (-1)^r \cdot x_1 \vee \dots \vee x_r
    \\
    \shortintertext{and}
    s(x_1 \wedge \dots \wedge x_k)
    &\coloneqq
    (-1)^k \cdot x_1 \wedge \dots \wedge x_k,
\end{align}
respectively.  The antipode turns $\Sym V$ and $\Anti V$ into Hopf
algebras.

\subsection{Coalgebra Cohomology}
\label{sec:CoalgebraCohomology}

In this subsection we introduce the complex which computes the
cohomology of a coalgebra.  This is sometimes called (co-)Hochschild
cohomology of a coalgebra but we will call it simply coalgebra
cohomology, since we want to avoid a clash of nomenclature with the
later parts of this note.  The construction we present here works for
every coaugmented coalgebra, see
e.g. \cite[Chap. XVIII]{kassel:1995a}, nevertheless, we will only
discuss the case of $\Sym V$.
\begin{definition}[Coalgebra complex]
    \label{def:CoalgebraComplex}%
    Let $V$ be an $\ring{R}$-module.  On the tensor algebra
    $\Tensor^\bullet \Sym V$ over the symmetric algebra of $V$, which
    we will denote by $\CCa^\bullet(V)$, one defines the differential
    $\delta_\Ca \colon \Tensor^k \Sym V \to \Tensor^{k+1} \Sym V$ by
    extending $-\redshcoprod$ as a graded derivation with respect to
    the tensor product $\tensor$, i.e. we have
    \begin{equation}
        \label{eq:dDefTheRealThing}
        \delta_\Ca(X)
        =
        \sum_{i=1}^k (-1)^i
        \big(\id^{\tensor i-1}
        \tensor
        \redshcoprod
        \tensor
        \id^{\tensor k - i}\big)(X),
    \end{equation}
    for $X \in \Tensor^k\Sym V$.  We call the complex
    $(\CCa^\bullet(V),\delta_\Ca)$ the \emph{coalgebra complex of
      $V$}.
\end{definition}

Note that for simplicity we call $\CCa^\bullet(V)$ the coalgebra
complex of $V$ instead of $\Sym V$, even though $V$ itself is not even
a coalgebra. Note also that $\CCa^\bullet(V)$ will always be
considered as associative algebra with respect to the tensor
product. This will be important when we introduce modules of
$\CCa^\bullet(V)$ to admit also coefficients later on.
\begin{proposition}
    \label{prop:dIsGradedDifferential}%
    Let $V$ be an $\ring{R}$-module.
    \begin{propositionlist}
    \item \label{item:dSquareZero} The map $\delta_\Ca$ is a
        differential, i.e. $\delta_\Ca^2 = 0$.
    \item \label{item:dIsDerivation} The coalgebra complex
        $(\CCa^\bullet(V), \delta_\Ca)$ together with the tensor
        product $\tensor$ is a differential graded algebra.  In
        particular we have
        \begin{equation}
            \label{eq:dFG}
            \delta_\Ca(X \tensor Y)
            =
            \delta_\Ca(X) \tensor Y
            +
            (-1)^k X \tensor \delta_\Ca (Y),
        \end{equation}
        for all $X \in \Tensor^k \Sym V$ and
        $Y \in \Tensor^\bullet\Sym V$ with $k \in \mathbb{N}_0$.
    \item \label{item:dSweedler}
        In Sweedler's notation $\delta_\Ca$ is given by
        \begin{equation}
            \delta_\Ca(X)
            =
            \Unit \tensor X
            -
            \sum_{i=1}^{k} (-1)^i
            X_1 \tensor \cdots \tensor
            (X_i)_\sweedler{1} \tensor (X_i)_\sweedler{2}
            \tensor \cdots \tensor X_k
            +
            (-1)^{k +1} X \tensor \Unit,
        \end{equation}
        for $X = X_1 \tensor \cdots \tensor X_k \in \Tensor^k\Sym V$.
    \end{propositionlist}
\end{proposition}
\begin{proof}
    The first part follows by checking it on generators $X \in \Sym V$
    of the tensor algebra.  There it amounts to the coassociativity of
    $\redshcoprod$.  The second part is clear since $\delta_\Ca$ was
    defined as a graded derivation.  The last part follows from the
    explicit formula \eqref{eq:dDefTheRealThing} for $\delta_\Ca$ and
    using Sweedler's notation for $\redshcoprod$.
\end{proof}

Thus we indeed obtain a complex whose cohomology becomes a graded
algebra since the differential is a graded derivation of the tensor
product.  Before actually computing this cohomology we need to extend
the complex in order to incorporate also coefficients in a comodule.
\begin{definition}[$\Sym V$-comodule]
    \label{def:SVComodule}
    Let $V$ and $\module{M}$ be $\ring{R}$-modules.  A
    \emph{$\Sym V$-comodule} structure on $\module{M}$ is given by a
    linear map $\Left \colon \module{M} \to \Sym V \tensor \module{M}$
    such that
    \begin{equation}
        \label{eq:LeftComoduleStructure}
        (\id \tensor \Left) \circ \Left
        =
        (\shcoprod \tensor \id) \circ \Left
    \end{equation}
    and
    \begin{equation}
        \label{eq:LeftComoduleCounital}
        (\counit \tensor \id) \circ \Left
        =
        \id.
    \end{equation}
\end{definition}
We will use Sweedler's notation and write
\begin{equation}
    \Left (m)
    =
    m_\sweedler{-1} \tensor m_\sweedler{0}
\end{equation}
for $m \in \module{M}$.  Note that we use negative Sweedler indices to
denote the left comodule structure.  The comodule property
\eqref{eq:LeftComoduleStructure} then means that we can write
$m_\sweedler{-2} \tensor m_\sweedler{-1} \tensor m_\sweedler{0}$
without having to bother if we used twice the coaction or the
comultiplication to obtain $m_\sweedler{-2}$.  Moreover,
\eqref{eq:LeftComoduleCounital} becomes
$\counit(m_\sweedler{-1}) m_\sweedler{0} = m$ in Sweedler's notation.

The following examples will turn out to be of interest even from a
geometric point of view.
\begin{example}[$\Sym V$-comodules]
    \label{ex:Comodules}%
    Let $V$ be an $\ring{R}$-module.
    \begin{examplelist}
    \item \label{item:TrivialComoduleStructure} Since $\Sym V$ has a
        distinguished group-like element $\Unit$, we have for every
        $\ring{R}$-module $\module{M}$ the trivial comodule structure
        given by
        \begin{equation}
            \label{eq:TrivialLeft}
            \Left_{\script{triv}}(m)
            =
            \Unit \tensor m,
        \end{equation}
        where $m \in \module{M}$.
    \item \label{item:TrivialComodule} Consider the scalars
        $\module{M} = \ring{R}$ which we can turn into a comodule by
        \begin{equation}
            \label{eq:ScalarsAreComodule}
            \Left_\ring{R} (1_\ring{R})
            =
            \Unit \tensor 1_\ring{R}.
        \end{equation}
        Note that this is a particular case of the trivial comodule
        structure \eqref{eq:TrivialLeft}.
    \item \label{item:SVIsComodule} The symmetric algebra $\Sym V$
        becomes a left comodule itself by setting
        \begin{equation}
            \label{eq:LeftSymVComodule}
            \Left_{\Sym V}(\phi)
            =
            \shcoprod(\phi)
        \end{equation}
        for $\phi \in \Sym V$.  Note that the two tensor factors
        $\Sym V \tensor \Sym V$ on the target side play a different
        role: one is the comodule, the other the coalgebra.
    \item \label{item:SymVTensorTrivialComoduleM} A combination of
        these two constructions can be obtained as follows.  For an
        $\ring{R}$-module $\module{M}$ we endow
        $\Sym V \tensor \module{M}$ with the tensor product of the
        canonical comodule structure on $\Sym V$ and the trivial one
        on $\module{M}$.  This yields
        \begin{equation}
            \label{eq:LeftM}
            \Left_{\Sym V \tensor \module{M}}(\phi \tensor m)
            =
            \Left_{\Sym V} (\phi) \tensor m
            =
            \shcoprod(\phi) \tensor m
        \end{equation}
        on factorizing elements
        $\phi \tensor m \in \Sym V \tensor \module{M}$.
    \item \label{item:OtherSymWisComodule} A particular case of
        \eqref{eq:LeftM} comes from the following situation. Assume
        $V \subseteq W$ is a submodule of an ambient $\ring{R}$-module
        such that we have a complementary $\ring{R}$-module $V^\perp$
        with $W = V \oplus V^\perp$.  Then
        \begin{equation}
            \label{eq:SymWisSymVTensorSymKerp}
            \Sym W
            \simeq
            \Sym V \tensor \Sym V^\perp
        \end{equation}
        yields the left $V$-coaction
        \begin{equation}
            \label{eq:ComoduleSW}
            \Left_{\Sym W}
            =
            \Left_{\Sym V} \tensor \id_{\Sym V^\perp}
            =
            (\pr_{V}^\vee \tensor \id) \circ \shcoprod^W,
        \end{equation}
        where $\shcoprod^W$ is the shuffle coproduct of the symmetric
        algebra $\Sym W$ over $W$, and
        $\pr_V^\vee = \pr_V \vee \dots \vee \pr_V \colon \Sym W \to
        \Sym V$ is the extension of the projection
        $\pr_V \colon W \to V$ as an algebra morphism.  Here we use
        the canonical identification
        \eqref{eq:SymWisSymVTensorSymKerp}.
    \item \label{item:SymUComoduleV} Similarly, we can exchange the
        role of the ambient module with the submodule. Thus consider a
        submodule $U \subseteq V$.  Then $\Sym U$ becomes a left
        comodule over $\Sym V$ by
        \begin{equation}
            \label{eq:LeftSymU}
            \Left_{\Sym U}(u)
            =
            \shcoprod(u)
            \in \Sym U \tensor \Sym U
            \subseteq
            \Sym V \tensor \Sym U
        \end{equation}
        for $u \in \Sym U$. Strictly speaking, the last step might not
        be an injective inclusion but just a possibly non-injective
        map due to torsion effects.
    \end{examplelist}
\end{example}
Having a comodule $\module{M}$ we can generalize the coalgebra complex
from \autoref{def:CoalgebraComplex} to an $\module{M}$-valued version
as follows.  We consider the differential
\begin{equation}
    \label{eq:DifferentialBiComoduleVersion}
    \delta_\module{M}\colon
    \Tensor^\bullet\Sym V \tensor \module{M}
    \to
    \Tensor^{\bullet+1}\Sym V \tensor \module{M},
\end{equation}
defined by
\begin{equation}
    \label{eq:DifferentialDefOnMValued}
    \delta_\module{M}(X \tensor m)
    \coloneqq
    \delta_\Ca(X) \tensor m
    + (-1)^{k+1}(X \tensor \pr_+(m_\sweedler{-1})) \tensor m_\sweedler{0},
\end{equation}
for $X \tensor m \in \Tensor^k\Sym V \tensor \module{M}$.  Here
$\pr_+ \colon \Sym V \to \oplus_{k=1}^\infty \Sym^k V$ denotes the
projection to positive symmetric degrees, as in
\eqref{eq:ProjectionOnShcorprod}.
\begin{proposition}
    \label{prop:dgModuleStructureCCaVM}%
    Let $\module{M}$ be an $\Sym V$-comodule.
    \begin{propositionlist}
    \item Then $\delta_\module{M}$ defined by
        \eqref{eq:DifferentialDefOnMValued} satisfies
        $(\delta_\module{M})^2 = 0$.
    \item The complex
        $(\Tensor^\bullet \Sym V \tensor \module{M},
        \delta_\module{M})$ is a differential graded left
        $\CCa^\bullet(V)$-module with left action given by
        \begin{equation}
            \label{eq:ModuleStructureCCaVM}
            Y \acts (X \tensor m)
            \coloneqq
            (Y \tensor X) \tensor m,
        \end{equation}
        for
        $X \tensor m \in \Tensor^\bullet \Sym V \tensor \module{M}$
        and $Y \in \Tensor^\bullet\Sym V$.  In particular, we have
        \begin{equation}
            \label{eq:dgModuleStructureCCaVM}
            \delta_\module{M}
            \bigl(
                Y \acts(X \tensor m)
            \bigr)
            = \delta_\Ca(Y) \acts (X \tensor m)
            + (-1)^{k} Y \acts \delta_\module{M}(X \tensor m),
        \end{equation}
        for
        $X \tensor m \in \Tensor^\bullet \Sym V \tensor \module{M}$
        and $Y \in \Tensor^k\Sym V$.
    \end{propositionlist}
\end{proposition}
\begin{proof}
    For the first part let
    $X \tensor m \in \Tensor^k\Sym V \tensor \module{M}$.  Then
    \begin{align*}
        (\delta_\module{M})^2(X \tensor m)
        &=
        \delta_\module{M}\big(
        \delta_\Ca(X) \tensor m
        + (-1)^{k+1} (X \tensor \pr_+(m_\sweedler{-1})) \tensor m_\sweedler{0}
        \big)
        \\
        &=
        \underbrace{(\delta_\Ca)^2}_{=0}(X) \tensor m
        \\
        &\quad+ (-1)^{k+1}
        \delta_\Ca\bigl(X \tensor \pr_+(m_\sweedler{-1})\bigr)
        \tensor m_\sweedler{0}
        \\
        &\quad+ (-1)^{k}
        \bigl(\delta_\Ca(X) \tensor \pr_+(m_\sweedler{-1}) \bigr)
        \tensor m_\sweedler{0}
        \\
        &\quad-
        \bigl( X \tensor \pr_+(m_\sweedler{-2}) \tensor \pr_+(m_\sweedler{-1}) \bigr)
        \tensor m_\sweedler{0}
        \\
        &= -
        \bigl(X \tensor \delta_\Ca(\pr_+(m_\sweedler{-1}))\bigr)
        \tensor m_{\sweedler{0}}
        \\
        &\quad-
        \bigl(X \tensor \pr_+(m_\sweedler{-2}) \tensor \pr_+(m_\sweedler{-1}) \bigr)
        \tensor m_\sweedler{0}
        \\
        &=
        0.
    \end{align*}
    In the last step we used that
    $\delta_\Ca(\pr_+(X)) = - \pr_+(X_\sweedler{1}) \tensor
    \pr_+(X_\sweedler{2})$ holds for all $X \in \Sym V$ by
    \eqref{eq:ProjectionOnShcorprod}.  For the second part it is clear
    that \eqref{eq:ModuleStructureCCaVM} defines a graded left module
    structure.  Finally, \eqref{eq:dgModuleStructureCCaVM} is a
    straightforward computation.
\end{proof}
\begin{definition}[Coalgebra complex with coefficients]
    \label{definition:CgrComplex}%
    Let $\module{M}$ be a $\Sym V$-comodule.  We denote the complex
    $\Tensor^\bullet\Sym V \tensor \module{M}$ with the differential
    $\delta_\module{M}$ defined in \eqref{eq:DifferentialDefOnMValued}
    by $\CCa^\bullet(V,\module{M})$.  The corresponding cohomology
    will be denoted by
    \begin{equation}
        \label{eq:CohomSVM}
        \HCa^\bullet(V, \module{M})
        =
        \bigoplus_{k=0}^\infty
        \HCa^k(V, \module{M})
        \quad
        \textrm{with}
        \quad
        \HCa^k(V, \module{M})
        =
        \frac{\ker \delta_\Ca\at{\Tensor^k \Sym V \tensor \module{M}}}
        {\image \delta_\Ca\at{\Tensor^{k-1}\Sym V \tensor \module{M}}}.
    \end{equation}
\end{definition}

Thus each comodule from \autoref{ex:Comodules} yields an associated
coalgebra cohomology, which will turn out to be of great interest from
a geometric point of view.  Nevertheless, at the moment we are only
able to compute the coalgebra cohomology in one of these situations:
\begin{example}[Coalgebra cohomology]
    \label{ex:CoalgebraCohomology}%
    Let $V$ be an $\ring{R}$-module.
    \begin{examplelist}
    \item \label{item:CoalgebraCohomology_Scalar} The scalar case from
        \autoref{prop:dIsGradedDifferential} can be recovered by
        considering the trivial comodule $\module{M} = \ring{R}$ as in
        \autoref{ex:Comodules}~\ref{item:TrivialComodule}.
    \item \label{item:CoalgebraCohomology_SymmetricAlgebra} Consider
        the coefficient module $\module{M} = \Sym V$ with coaction
        $\Left_{\Sym V} (\phi) = \shcoprod(\phi) = \phi_\sweedler{-1}
        \tensor \phi_\sweedler{0}$ as in
        \autoref{ex:Comodules}~\ref{item:SVIsComodule}.  Then the
        differential is given by
        \begin{equation}
            \delta_{\Sym V}(X \tensor \phi)
            =
            \delta_\Ca(X) \tensor \phi
            + (-1)^{k+1}\bigl(X \tensor \pr_+(\phi_\sweedler{-1})\bigr)
            \tensor
            \phi_\sweedler{0},
        \end{equation}
        for $X \in \Tensor^k\Sym V$.  Consider now the linear map
        \begin{equation}
            \label{eq:deltaInverse}
            \delta^{-1}\colon
            \CCa^\bullet (V, \Sym V)
            \to
            \CCa^{\bullet - 1}(V, \Sym V)
        \end{equation}
        defined by
        \begin{equation}
            \delta^{-1}\big(
            (X_1 \tensor \cdots \tensor X_k)
            \tensor
            \phi
            \big)
            = \begin{cases}
                (-1)^{k}\counit(\phi) \cdot
                (X_1 \tensor \cdots \tensor X_{k-1})
                \tensor
                X_k
                & \textrm{if } k \geq 1 \\
                0
                & \textrm{if } k = 0.
            \end{cases}
        \end{equation}
        Then a straightforward computation shows that
        \begin{equation}
            \delta_{\Sym V} \delta^{-1}
            + \delta^{-1} \delta_{\Sym V} = \id - \iota \pi
        \end{equation}
        holds, where $\iota(1_\ring{R}) = 1 \tensor \Unit$ is the
        canonical inclusion and
        $\pi \colon \Tensor^\bullet \Sym V \tensor \Sym V \to \ring{R}
        = \Tensor^0 \Sym V \tensor \Sym^0 V$ denotes the projection
        onto degree $0$.  In other words, we obtain a homotopy retract
        \begin{equation}
            \begin{tikzcd}
                \ring{R}
                \arrow[r,"\iota", shift left = 3pt]
                &\bigl( \CCa^{\bullet}(V,\Sym V), \delta_{\Sym V} \bigr)
                \arrow[l,"\pi", shift left = 3pt]
                \arrow[loop,
                out = -30,
                in = 30,
                distance = 30pt,
                start anchor = {[yshift = -7pt]east},
                end anchor = {[yshift = 7pt]east},
                "\delta^{-1}"{swap}
                ]
            \end{tikzcd}
            .
        \end{equation}
        See \autoref{sec:HomotopyRetracts} for the basics on homotopy
        retracts.  Here we consider $\ring{R}$ as a complex
        concentrated in degree $0$.  Hence we computed the cohomology
        to be
        \begin{equation}
            \HCa^\bullet(V,\Sym V)
            \simeq \ring{R}.
        \end{equation}
    \item \label{item:CoalgebraCohomology_SVtensorTrivCoeff} Consider
        an $\ring{R}$-module $\module{M}$ with trivial
        $\Sym V$-coaction, as in
        \autoref{ex:Comodules}~\ref{item:SymVTensorTrivialComoduleM}.
        Then the differential on
        $\CCa^\bullet(V,\Sym V \tensor \module{M})$ is simply given by
        $\delta_{\Sym V \tensor \module{M}} = \delta_{\Sym V} \tensor
        \id_\module{M}$.  Thus, when viewing $\module{M}$ as a complex
        concentrated in degree $0$, we have
        $\CCa^\bullet(V, \Sym V \tensor \module{M}) \simeq
        \CCa^\bullet(V, \Sym V) \tensor \module{M}$ as cochain
        complexes. Hence, we clearly obtain a homotopy retract
        \begin{equation}
            \begin{tikzcd}[column sep=large]
                \ring{R} \tensor \module{M}
                \arrow[r,"\iota \tensor \id_\module{M}", shift left = 3pt]
                &\bigl (\CCa^{\bullet}(V,\Sym V \tensor \module{M}),\delta_{\Sym V \tensor \module{M}} \bigr)
                \arrow[l,"\pi \tensor \id_\module{M}", shift left = 3pt]
                \arrow[loop,
                out = -30,
                in = 30,
                distance = 30pt,
                start anchor = {[yshift = -7pt]east},
                end anchor = {[yshift = 7pt]east},
                "\delta^{-1} \tensor \id_\module{M}"{swap}
                ]
            \end{tikzcd}
            .
        \end{equation}
        Thus one has for a trivial comodule $\module{M}$ the
        cohomology
        \begin{equation}
            \label{eq:HCEValuesInSVTensorM}
            \HCa^\bullet(V, \Sym V \tensor \module{M})
            \simeq
            \ring{R} \tensor \module{M}
            \simeq
            \module{M},
        \end{equation}
        concentrated in degree $0$.
    \item \label{item:CoalgebraCohomology_Submodule} A particular
        scenario for the previous part
        \ref{item:CoalgebraCohomology_SVtensorTrivCoeff} can be
        obtained already for $W$ being another $\ring{R}$-module such
        that $V \subseteq W$ is a submodule with complement and write
        $W = V^\perp \oplus V$, see
        \autoref{ex:Comodules}~\ref{item:OtherSymWisComodule}.  Then
        $\Sym W = \Sym V \tensor \Sym V^\perp$ as in
        \eqref{eq:SymWisSymVTensorSymKerp}.  We take the trivial
        coaction on $\Sym V^\perp$ and the canonical one on $\Sym V$.
        Then \ref{item:CoalgebraCohomology_SVtensorTrivCoeff} yields
        \begin{equation}
            \label{eq:HCaSymW}
            \HCa^\bullet(V, \Sym W)
            \simeq
            \Sym V^\perp
        \end{equation}
        concentrated in degree $0$.
    \end{examplelist}
\end{example}

To compute the coalgebra cohomologies in the case of trivial
coefficients, as well as for $U \subseteq V$ as in
\autoref{ex:Comodules}~\ref{item:SymUComoduleV}, we need to develop
more technology.

\subsection{Chevalley-Eilenberg Cohomology}
\label{sec:ChevalleyEilenbergCohomologyDualPointView}

There is another complex associated to an $\ring{R}$-module $V$,
namely the Chevalley-Eilenberg complex of $V$ considered as an abelian
Lie coalgebra with trivial cobracket
\begin{equation}
    \label{eq:TrivialCobracket}
    c\colon V \ni v
    \; \mapsto \;
    0 \in \Anti^2 V.
\end{equation}
Since we want to introduce this complex immediately with coefficients,
we first need to consider Lie coactions of $V$.
\begin{definition}[$V$-Lie coaction]
    \label{definition:LieCoaction}%
    Let $V$ and $\module{M}$ be $\ring{R}$-modules.  A \emph{Lie
      coaction} of $V$ on $\module{M}$ is given by a linear map
    $\rho \colon \module{M} \to \module{M} \tensor V$ such that
    \begin{equation}
        \label{eq:LieCoaction}
        \image (\rho \tensor \mathord{\id}_V) \circ \rho
        \subseteq
        \module{M} \tensor \Sym^2 V.
    \end{equation}
\end{definition}
\begin{remark}
    \label{remark:LeftRichtLieStuff}%
    Note that instead of using left coactions as in
    \autoref{def:SVComodule} we defined a Lie coaction as a right Lie
    coaction.  This turns out to be convenient when we introduce the
    van Est double complex. Of course, in our situation $V$ is
    abelian, i.e. the cobracket vanishes identically. Hence left and
    right coactions coincide. Nevertheless, it will be more convenient
    to think of right coactions in the following.
\end{remark}

We will use Sweedler's notation for the Lie coaction and write
$\rho(m) = m_\sweedler{0} \tensor m_\sweedler{1}$.  With this notation,
\eqref{eq:LieCoaction} is equivalent to
\begin{equation}
    m_\sweedler{0} \tensor m_\sweedler{1} \tensor m_\sweedler{2}
    =
    m_\sweedler{0} \tensor m_\sweedler{2} \tensor m_\sweedler{1},
\end{equation}
for $m \in \module{M}$.

Most examples of $V$-Lie coaction will be derived from coactions of
the coalgebra $\Sym V$ as follows.
\begin{lemma}
    \label{lemma:CoactionInfinitesimal}%
    Let $\module{M}$ be an $\Sym V$-comodule.  Then
    $\rho \colon \module{M} \to \module{M} \tensor V$ defined by
    \begin{equation}
        \label{eq:InfinitesimalLieCoaction}
        \rho(m)
        \coloneqq
        - m_\sweedler{0} \tensor \pr_V(m_\sweedler{-1})
    \end{equation}
    is a $V$-Lie coaction on $\module{M}$.
\end{lemma}
The minus sign is purely conventional and reflects the above idea to
have a right coaction.  We call $\rho$ with
\eqref{eq:InfinitesimalLieCoaction} the corresponding
\emph{infinitesimal Lie coaction} for the comodule $\module{M}$.  This
allows us to construct examples of Lie coactions out of the
$\Sym V$-coactions considered in \autoref{ex:Comodules}:
\begin{example}[Infinitesimal Lie coactions]
    \label{ex:LieCoactions}%
    Let $V$ and $\module{M}$ be $\ring{R}$-modules.
    \begin{examplelist}
    \item \label{item:TrivialLieCoaction} Starting with the trivial
        comodule structure $\Left_{\script{triv}}$ as in
        \eqref{eq:TrivialLeft} yields the \emph{trivial $V$-Lie
          coaction} $\rho_{\script{triv}}$
        \begin{equation}
            \label{eq:TrivialLieCoaction}
            \rho_{\script{triv}} = 0.
        \end{equation}
    \item \label{item:CanonicalLieCoaction} Consider the canonical
        comodule structure $\Left_{\Sym V} = \shcoprod$ of $\Sym V$
        over itself as in \eqref{eq:LeftSymVComodule}. Then the
        corresponding infinitesimal Lie coaction $\rho_{\Sym V}$ is
        given by the linear extension of
        \begin{equation}
            \label{eq:lambdaOnSV}
            \rho_{\Sym V}(x_1 \vee \cdots \vee x_r)
            = -
            \sum_{i=1}^r
            x_1 \vee \cdots \stackrel{i}{\wedge} \cdots \vee x_r \tensor x_i,
        \end{equation}
        where $r \in \mathbb{N}_0$ and $x_1, \ldots, x_\ell \in V$.
    \item \label{item:LieCoactionForSubmoduleU} Consider a submodule
        $U \subseteq V$ with the comodule structure $\Left_{\Sym U}$
        for $\Sym U$ from \eqref{eq:LeftSymU}.  Then the infinitesimal
        Lie coaction is
        \begin{equation}
            \label{eq:VcoactsOnSU}
            \rho_{\Sym U} (\phi)
            =
            - \phi_\sweedler{0} \tensor \pr_V(\phi_\sweedler{-1}),
        \end{equation}
        where we view $\phi_\sweedler{-1}$ as an element of $\Sym V$.
    \end{examplelist}
\end{example}

The $\Sym V$-coaction can be reconstructed from its infinitesimal Lie
coaction as follows: Since for every $r \in \mathbb{N}_0$ and
$m \in \module{M}$ we have
\begin{align}
    \rho^r(m)
    =
    (-1)^r m_\sweedler{0} \tensor \pr_V(m_\sweedler{-1})
    \vee \dots \vee
    \pr_V(m_\sweedler{-r})
    =
    (-1)^r m_\sweedler{0} \tensor r! \cdot \pr_{\Sym^r V}(m_\sweedler{-1})
\end{align}
by \autoref{lem:exponentialproj}, we know that there exists some
$n \in \mathbb{N}_0$ such that $\rho^r(m) = 0$ for all $r > n$.  Thus
we get
\begin{equation}
    \label{eq:SummingUpInfinitesimalStuff}
    \sum_{r=0}^{\infty} (-1)^r\frac{\rho^r(m)}{r!}
    =
    m_\sweedler{0} \tensor
    \sum_{r=0}^{\infty} (-1)^r \pr_{\Sym^r V}(m_\sweedler{-1})
    =
    m_\sweedler{0} \tensor s(m_\sweedler{-1}).
\end{equation}
Hence, up to a tensor flip and a sign we obtain the $\Sym V$-coaction
on $\module{M}$.  Recall, that $s$ denotes the antipode on $\Sym V$,
see \eqref{eq:AntipodeSV}.

Let us now recall the definition of the complex for the
Chevalley-Eilenberg cohomology of the abelian Lie coalgebra $V$.  The
underlying $\ring{R}$-module is
\begin{equation} \label{eq:ChevalleyEilenbergComplex}
    \CCE^\bullet(V, \module{M})
    \coloneqq
    \bigoplus_{\ell=0}^\infty
    \CCE^\ell(V, \module{M})
    \quad
    \textrm{with}
    \quad
    \CCE^\ell(V, \module{M})
    \coloneqq
    \module{M} \tensor \Anti^\ell V ,
\end{equation}
where $\module{M}$ carries a Lie coaction $\rho$ in the following.
The differential of
$\partial_\module{M} \colon \module{M} \tensor \Anti^\bullet V \to
\module{M} \tensor \Anti^{\bullet+1}V$ is then defined by
\begin{equation}
    \label{eq:CEDifferentialDef}
    \partial_\module{M}(m \tensor \xi)
    \coloneqq
    - m_\sweedler{0} \tensor (m_\sweedler{1} \wedge \xi),
\end{equation}
for $m \in \module{M}$ and $\xi \in \Anti^\bullet V$.  In the case of
$\module{M} = \ring{R}$ with trivial Lie coaction we will write
$\CCE^\bullet(V) \coloneqq \CCE^\bullet(V,\ring{R})$.
If $\rho$ is the infinitesimal Lie coaction for a $\Sym V$-comodule $\module{M}$
as in \autoref{lemma:CoactionInfinitesimal}, then the differential is given by
\begin{equation}
    \label{eq:CEDifferentialInfinitesimal}
    \partial_\module{M}(m \tensor \xi)
    \coloneqq
    m_\sweedler{0} \tensor (\pr_V(m_\sweedler{-1}) \wedge \xi),
\end{equation}
for $m \tensor \xi \in \module{M} \tensor \Anti^\bullet V$.
\begin{proposition}
    \label{prop:dgModuleStructureCCEVM}
    Let $\module{M}$ be a $V$-Lie comodule.
    \begin{propositionlist}
    \item Then $\partial_\module{M}$ defined by
        \eqref{eq:CEDifferentialDef} satisfies
        $(\partial_\module{M})^2 = 0$.
    \item For $\module{M} = \ring{R}$ with trivial Lie coaction
        $\CCE^\bullet(V) = \Anti^\bullet V$ is a differential graded
        algebra with differential $\partial_\ring{R} = 0$ and graded
        commutative multiplication $\wedge$.
    \item The complex
        $(\CCE^\bullet(V,\module{M}),\partial_\module{M})$ is a
        differential graded right $\Anti^\bullet V$-module with right
        action given by
        \begin{equation}
            \label{eq:ModuleStructureCCEVM}
            (m \tensor \xi) \racts \eta
            \coloneqq m \tensor (\xi \wedge \eta),
        \end{equation}
        for $m \tensor \xi \in \module{M} \tensor \Anti^\bullet V$ and
        $\eta \in \Anti^\bullet V$.  In particular, we have
        \begin{equation}
            \label{eq:dgModuleStructureCCEVM}
            \partial_\module{M}\bigl((m \tensor \xi) \racts \eta\bigr)
            = \partial_\module{M}(m \tensor \xi) \racts \eta
        \end{equation}
        for $m \tensor \xi \in \module{M} \tensor \Anti^\bullet V$ and
        $\eta \in \Anti^\bullet V$.
    \end{propositionlist}
\end{proposition}
\begin{proof}
    By \eqref{eq:LieCoaction} we know that applying $\rho$ twice
    yields symmetric tensors in $\Sym^2V$.  A subsequent
    $\wedge$-product then results in $\partial_\module{M}^2 = 0$,
    proving the first part.  The second part is clear from the
    definition of $\delta_\ring{R}$.  For the last part note that
    \eqref{eq:ModuleStructureCCEVM} clearly defines a graded module
    structure.  The last part follows directly from the definition of
    $\partial_\module{M}$ and \eqref{eq:ModuleStructureCCEVM}.
\end{proof}
\begin{definition}[Chevalley-Eilenberg complex with coefficients]
    \label{definition:CEWithCoefficients}%
    Let $\module{M}$ be a $V$-Lie comodule.  We denote the complex
    $\module{M} \tensor \Anti^\bullet V$ with differential
    $\partial_\module{M}$ defined in \eqref{eq:CEDifferentialDef} by
    $\CCE^\bullet(V,\module{M})$.  The corresponding cohomology will
    be denoted by
    \begin{equation} \label{eq:HCEDef}
        \HCE^\bullet(V, \module{M})
        =
        \bigoplus_{k=0}^\infty \HCE^\ell(V, \module{M})
        \quad
        \textrm{with}
        \quad
        \HCE^\ell(V, \module{M})
        =
        \frac{\ker \partial_\module{M}\at{\CCE^\ell(V, \module{M})}}
        {\image \partial_\module{M}\at{\CCE^{\ell-1}(V, \module{M})}}.
    \end{equation}
    In the case $\module{M} = \ring{R}$ with trivial Lie coaction we
    write $\HCE^\bullet(V) \coloneqq \HCE^\bullet(V,\ring{R})$.
\end{definition}

For later use we collect a few examples where the Chevalley-Eilenberg
cohomology can be computed explicitly and easily:
\begin{example}[Chevalley-Eilenberg cohomology]
    \label{ex:CECohomology}%
    Let $V$ and $\module{M}$ be $\ring{R}$-modules.
    \begin{examplelist}
    \item \label{item:TrivialCE} For the trivial Lie coaction on
        $\module{M}$ as in
        \autoref{ex:LieCoactions}~\ref{item:TrivialLieCoaction} the
        Chevalley-Eilenberg complex is given by
        $\CCE^\bullet(V, \module{M}) = \module{M} \tensor
        \Anti^\bullet V$ with differential $\partial_\module{M} = 0$.
        Thus for the trivial Lie coaction on $\module{M}$ we have a
        trivial homotopy retract
        \begin{equation}
            \begin{tikzcd}
                \module{M} \tensor \Anti^\bullet V
                \arrow[r,"\id", shift left = 3pt]
                &\bigl( \CCE^{\bullet}(V,\module{M}), 0 \bigr)
                \arrow[l,"\id", shift left = 3pt]
                \arrow[loop,
                out = -30,
                in = 30,
                distance = 30pt,
                start anchor = {[yshift = -7pt]east},
                end anchor = {[yshift = 7pt]east},
                "0"{swap}]
            \end{tikzcd}
        \end{equation}
        computing the cohomology:
        \begin{equation}
            \label{eq:HCETrivialCoaction}
            \HCE^\bullet(V, \module{M})
            =
            \module{M} \tensor \Anti^\bullet V.
        \end{equation}
        In the case $\module{M} = \ring{R}$ we have
        $\HCE^\bullet(V) = \Anti^\bullet V$.
    \item \label{item:SVwithVLieCoaction} For $\module{M} = \Sym V$
        with canonical Lie coaction
        $\rho_{\Sym V}(\phi) = - \phi_\sweedler{0} \tensor
        \pr_V(\phi_\sweedler{-1})$ as in
        \autoref{ex:LieCoactions}~\ref{item:CanonicalLieCoaction} the
        corresponding Chevalley-Eilenberg differential
        $\partial_{\Sym V}$ is given by
        \begin{equation}
            \label{eq:deltaCEforSV}
            \partial_{\Sym V} (x_1 \vee \cdots \vee x_r \tensor \xi)
            =
            \sum_{i=1}^r
            x_1 \vee \cdots \stackrel{i}{\wedge} \cdots \vee x_r
            \tensor
            (x_i \wedge \xi),
        \end{equation}
        where $\xi \in \Anti^\ell V$ and $x_1, \ldots, x_r \in V$ with
        $r \in \mathbb{N}_0$, or in Sweedler's notation
        \begin{equation}
            \partial_{\Sym V}(\phi \tensor \xi)
            =
            \phi_\sweedler{0}
            \tensor
            \pr_V(\phi_\sweedler{-1}) \wedge \xi.
        \end{equation}
        Here for $r = 0$ we have $\partial_{\Sym V}(\xi) = 0$.  Now
        $\partial_{\Sym V}$ can also be characterized as follows: the
        Chevalley-Eilenberg complex $\Sym V \tensor \Anti^\bullet V$
        carries the structure of a graded commutative algebra with
        respect to the antisymmetric degree and freely generated as
        such by two copies of $V$, one for the Grassmann part, one for
        the symmetric algebra part.  Then $\partial_{\Sym V}$ is the
        unique graded derivation of degree $+1$ of this algebra with
        \begin{equation}
            \label{eq:deltaOnGenerators}
            \partial_{\Sym V}(x \tensor 1) = 1 \tensor x
            \quad
            \textrm{and}
            \quad
            \partial_{\Sym V}(1 \tensor x) = 0
        \end{equation}
        for $x \in V$. The computation of the cohomology of
        $\partial_{\Sym V}$ is now a standard homotopy argument: We
        define a graded derivation $\partial^*$ of
        $\Sym V \tensor \Anti^\bullet V$ with antisymmetric degree
        $-1$ by specifying it on generators $x \in V$ by
        \begin{equation}
            \label{eq:deltaStarDef}
            \partial^*(x \tensor 1) = 0
            \quad
            \textrm{and}
            \quad
            \partial^*(1 \tensor x) = x \tensor 1.
        \end{equation}
        Explicitly, we get
        \begin{equation}
            \partial^*(
            \phi \tensor
            \xi_1 \wedge \dots \wedge \xi_\ell)
            = \sum_{i=1}^{\ell} (-1)^{i+1}
            \phi \vee \xi_i
            \tensor
            \xi_1 \wedge \cdots
            \overset{i}{\wedge} \cdots \wedge
            \xi_\ell,
        \end{equation}
        or in Sweedler's notation
        \begin{equation}
            \partial^*(\phi \tensor \xi)
            =
            \phi \vee \pr_V(\xi_\sweedler{-1})
            \tensor
            \xi_\sweedler{0},
        \end{equation}
        for $\phi \tensor \xi \in \Sym V \tensor \Anti^\bullet V$.
        Then the graded commutator $[\partial_{\Sym V}, \partial^*]$
        is again a graded derivation, now of symmetric and
        antisymmetric degree $0$.  We obtain
        \begin{equation}
            \label{eq:deltaHomotopy}
            \partial_{\Sym V} \partial^* + \partial^* \partial_{\Sym V}
            =
            \degs + \dega
        \end{equation}
        with the symmetric and antisymmetric degree derivations being
        defined by their values on generators as
        $\dega(x \tensor 1) = 0 = \degs(1 \tensor x)$ and
        $\dega(1 \tensor x) = 1 \tensor x$ as well as
        $\degs(x \tensor 1) = x \tensor 1$. Note that
        \eqref{eq:deltaHomotopy} is embarrassingly trivial to show as
        it suffices to check this on generators. We introduce the map
        \begin{equation}
            \label{eq:deltaInv}
            \partial^{-1}_{\Sym V}(\phi \tensor \xi)
            =
            \begin{cases}
                \frac{1}{r + \ell} \partial^*(\phi \tensor \xi)
                & \textrm{if } r + \ell > 0 \\
                0
                & \textrm{if } r + \ell = 0,
            \end{cases}
        \end{equation}
        where $\xi \in \Anti^\ell V$ and $\phi \in \Sym^r V$.
        Together with the projection
        $\pi \colon \Sym V \tensor \Anti^\bullet V \to \ring{R}$ onto
        the part with antisymmetric degree and symmetric degree zero,
        and the inclusion
        $\iota \colon \ring{R} \to \Sym V \tensor \Anti^\bullet V$ we
        end up with a homotopy retract
        \begin{equation}
            \label{eq:PoincareLemma}
            \begin{tikzcd}
                \ring{R}
                \arrow[r,"\iota", shift left = 3pt]
                &\bigl( \CCE^{\bullet}(V,\Sym V),\partial_{\Sym V} \bigr)
                \arrow[l,"\pi", shift left = 3pt]
                \arrow[loop,
                out = -30,
                in = 30,
                distance = 30pt,
                start anchor = {[yshift = -7pt]east},
                end anchor = {[yshift = 7pt]east},
                "\partial^{-1}_{\Sym V}"{swap}]
            \end{tikzcd}
            .
        \end{equation}
        Hence we have computed the Chevalley-Eilenberg cohomology to
        be
        \begin{equation}
            \label{eq:HCEValuesInSV}
            \HCE^\bullet(V, \Sym V)
            \simeq
            \ring{R}
        \end{equation}
        concentrated in degree $0$ with the explicit isomorphism given
        by the projection onto symmetric and antisymmetric degree
        zero.

        Note that this computation has a nice geometric
        interpretation: for $V = \field{R}^n$ this reduces to the
        computation of the de Rham cohomology of differential forms
        with \emph{polynomial coefficients} by means of the standard
        homotopy from the Poincaré Lemma, see also
        e.g. \cite[Exercise~2.11]{waldmann:2007a} for this point of
        view. The notation for the homotopy is inspired by Fedosov's
        construction of star products on symplectic manifolds
        \cite[Lemma~5.1.2]{fedosov:1996a}.
    \item \label{item:UintoVSymUCE} Another case is obtained for a
        submodule $U \subseteq V$ and the canonical $V$-Lie coaction
        $\rho_{\Sym U}$ as in
        \autoref{ex:LieCoactions}~\ref{item:LieCoactionForSubmoduleU}.
        The Chevalley-Eilenberg differential
        \begin{equation}
            \label{eq:CEDiffForUinVCoactOnSU}
            \partial_{\Sym U}\colon
            \Sym U \tensor \Anti^\bullet V
            \to
            \Sym U \tensor \Anti^{\bullet+1} V
        \end{equation}
        is now determined by
        \begin{equation}
            \label{eq:CEDiffForUinVCoactOnSUExplicit}
            \partial_{\Sym U}(u_1 \vee \cdots \vee u_r \tensor \xi)
            =
            \sum_{i=1}^r
            u_1 \vee \cdots \stackrel{i}{\wedge} \cdots \vee u_r
            \tensor
            u_i \wedge \xi
        \end{equation}
        for $\xi \in \Anti^\bullet V$ and $u_1, \ldots, u_r \in U$.
        We assume now that $U \subseteq V$ is complemented, i.e. we
        have $V = U \oplus U^\perp$.  This splitting induces an
        isomorphism
        $\Anti^\bullet V \simeq \Anti^\bullet U \tensor \Anti^\bullet
        U^\perp$ which in turn induces an isomorphism
        \begin{equation}
            \CCE^\bullet(V,\Sym U)
            = \Sym U \tensor \Anti^\bullet V
            \simeq
            \Sym U
            \tensor
            \Anti^\bullet U
            \tensor
            \Anti^\bullet U^\perp
            \simeq
            \CCE^\bullet(U,\Sym U)
            \tensor
            \CCE^\bullet(U^\perp)
        \end{equation}
        of complexes.  For $\CCE^\bullet(U^\perp)$ and
        $\CCE^\bullet(U,\Sym U)$ we have homotopy retracts by
        \ref{item:TrivialCE} and \ref{item:SVwithVLieCoaction},
        respectively.  Their tensor product then yields a homotopy
        retract
        \begin{equation}
            \label{diag:CCESUhomotopRetract}
            \begin{tikzcd}[column sep = large]
                \ring{R} \tensor \Anti^\bullet U^\perp
                \arrow[r,"\iota \tensor \id", shift left = 3pt]
                &\CCE^\bullet(V,\Sym U)
                \arrow[l,"\pi \tensor \id", shift left = 3pt]
                \arrow[loop,
                out = -30,
                in = 30,
                distance = 30pt,
                start anchor = {[yshift = -7pt]east},
                end anchor = {[yshift = 7pt]east},
                "\partial^{-1}_{\Sym U} \tensor \id"{swap}
                ]
            \end{tikzcd}
            ,
        \end{equation}
        see \autoref{prop:SumTensorHomotopies}, thus inducing an
        explicit isomorphism
        $\HCE^\bullet(V,\Sym U) \simeq \ring{R} \tensor \Anti^\bullet
        U^\perp \simeq \Anti^\bullet U^\perp$.
    \end{examplelist}
\end{example}
\begin{remark}
    \label{remark:HCEIsEasy}%
    As a conclusion from these examples we note that the computation
    of the Chevalley-Eilenberg cohomology is, in many cases, fairly
    easy and \emph{explicit}.  We even have very explicit homotopies
    to find explicit primitives of exact cocycles as well as explicit
    decompositions of a cocycle into an exact part and a
    representative of its cohomology class.  This will turn out to be
    advantageous for the computation of the Hochschild cohomologies
    later on.
\end{remark}

\section{The Van Est Double Complex}
\label{sec:VanEstDoubleComplex}

We now introduce and study a double complex connecting the coalgebra
cohomology of a given $\ring{R}$-module $V$ with the
Chevalley-Eilenberg cohomology of $V$ understood as a abelian Lie
coalgebra.  To motivate the introduction of this double complex,
consider again the case $V = \Secinfty(TM)$ of vector fields on some
manifold $M$.  In this case the coalgebra cohomology
$\CCa^\bullet(V) = \Tensor^\bullet\Sym \Secinfty(TM)$ can be
interpreted as the complex of polynomial functions on the nerve of
$T^*M$ interpreted as a Lie groupoid with multiplication $+$.  Thus we
can view $\CCa^\bullet(V)$ as a certain polynomial group(oid)
cohomology.  For group cohomology there is the classical van Est
Theorem \cite{vanest:1953a,vanest:1953b} relating it to the
Chevalley-Eilenberg cohomology of the corresponding Lie algebra, and
this has been generalized to groupoids in \cite{crainic:2003a}.

Our strategy is now to reformulate the geometric double complex used
in the original proof of the van Est Theorem, see
\cite{vanest:1953b,vanest:1955b}, in purely (co-)algebraic terms in
order to relate $\CCa^\bullet(V)$ with $\CCE^\bullet(V)$ for a general
$V$.  We follow closely \cite{meinrenken.salazar:2020a} in our
notation. Note, however, that our situation is more general in so far
as we deal with a purely algebraic situation and, on the other hand,
much more special since we only consider the (co-)abelian situation.

Consider the following $\mathbb{N}_0 \times \mathbb{N}_0$-graded
module:
\begin{equation}
    \label{eq:CVanEstDef}
    \CVanEst^{\bullet, \bullet}(V)
    \coloneqq
    \Tensor^\bullet \Sym V
    \tensor
    \Sym V
    \tensor
    \Anti^\bullet V
\end{equation}
We can interpret it in two different ways: Either as
$\CVanEst^{\bullet,\bullet}(V) = \CCa^\bullet(V, \Sym V) \tensor
\Anti^\bullet V$, or as
$\CVanEst^{\bullet,\bullet}(V) = \Tensor^\bullet \Sym V \tensor
\CCE^\bullet(V, \Sym V)$.  Both interpretations can now be used to
equip $\CVanEst^{\bullet,\bullet}(V)$ with a differential: Using the
canonical left coaction of $\Sym V$ on itself we obtain from
\autoref{sec:CoalgebraCohomology} a differential $\delta_{\Sym V}$ on
$\CCa^\bullet(V,\Sym V)$, see also
\autoref{ex:CoalgebraCohomology}~\ref{item:CoalgebraCohomology_SymmetricAlgebra}.
We define
\begin{equation}
    \delta \coloneqq \delta_{\Sym V} \tensor \id_{\Anti V}
    \colon
    \CVanEst^{k,\ell}(V)
    \to
    \CVanEst^{k+1,\ell}(V).
\end{equation}
On elements this differential is given by
\begin{equation}
    \label{eq:VerticalVEDifferential}
    \delta(X \tensor \phi \tensor \xi)
    = \delta_\Ca (X)
    \tensor
    \phi
    \tensor
    \xi
    + (-1)^{k+1}
    (X \tensor \pr_+(\phi_\sweedler{-1}))
    \tensor
    \phi_\sweedler{0}
    \tensor
    \xi,
\end{equation}
for $X \tensor \phi \tensor \xi \in \CVanEst^{k,\ell}(V)$.

On the other hand, we can use the infinitesimal Lie coaction
$\rho(\phi) = - \phi_\sweedler{0} \tensor \pr_V(\phi_\sweedler{1})$ of
$V$ on $\Sym V$ from
\autoref{ex:CECohomology}~\ref{item:SVwithVLieCoaction} to define
another differential
\begin{equation}
    \partial
    \coloneqq
    (-1)^{k+1} \id_{\Tensor\Sym V} \tensor \partial_{\Sym V}
    \colon
    \CVanEst^{k,\ell}(V)
    \to
    \CVanEst^{k,\ell+1}(V).
\end{equation}
On elements this differential is given by
\begin{equation}
    \label{eq:HorizontalVEDifferential}
    \partial(X \tensor \phi \tensor \xi)
    = (-1)^{k}
    X
    \tensor
    \phi_\sweedler{0}
    \tensor
    \bigl(\pr_V(\phi_\sweedler{1}) \wedge \xi\bigr),
\end{equation}
for $X \tensor \phi \tensor \xi \in \CVanEst^{k,\ell}(V)$.  These
differentials $\delta$ and $\partial$ equip
$\CVanEst^{\bullet,\bullet}(V)$ with the structure of a double
complex: For this we define the \emph{total differential} $\Double$ by
\begin{equation}
    \Double
    \coloneqq \delta + \partial.
\end{equation}
\begin{proposition}
    \label{proposition:VanEstDoubleComplex}%
    On $\CVanEst^{\bullet,\bullet}(V)$ it holds $\Double^2 = 0$.
\end{proposition}
\begin{proof}
    We have
    \begin{equation*}
        \Double^2 =
        \delta^2 + \delta \partial + \partial\delta + \partial^2
        = \delta \partial + \partial\delta.
    \end{equation*}
    Let
    $X \tensor \phi \tensor \xi \in \CVanEst^{k,\ell}(V)$.
    Then it holds
    \begin{align*}
        \partial\delta(X \tensor \phi \tensor \xi)
        &=
        \partial\bigl(\delta_{\Sym V} (X \tensor \phi) \tensor \xi\bigr) \\
        &=
        \partial\bigl(
        \delta_\Ca(X) \tensor \phi \tensor \xi\bigr)
        \\
        &\quad+
        (-1)^{k+1}
        \partial\bigl((X \tensor \pr_+(\phi_\sweedler{-1}))
        \tensor
        \phi_\sweedler{0}
        \tensor
        \xi
        \bigr)
        \\
        &=
        (-1)^{k+1}
        \delta_\Ca(X)
        \tensor
        \phi_\sweedler{0}
        \tensor
        \bigl(\pr_V(\phi_\sweedler{1}) \wedge \xi \bigr)
        \\
        &\quad+
        \bigl(X \tensor \pr_+(\phi_\sweedler{-1}) \bigr)
        \tensor
        \phi_\sweedler{0}
        \tensor
        \bigl(\pr_V(\phi_\sweedler{1}) \wedge \xi \bigr)
        \\
        &=
        (-1)^{k+1} \delta\bigl(
        X
        \tensor
        \phi_\sweedler{0}
        \tensor
        (\pr_V(\phi_\sweedler{1}) \wedge \xi )\bigr)
        \\
        &=
        -\delta \partial(X \tensor \phi \tensor \xi).
    \end{align*}
\end{proof}
\begin{definition}[Van Est double complex]
    \label{def:VanEstDoubleComplex}%
    Let $V$ be an $\ring{R}$-module.  We call
    $\CVanEst^{\bullet,\bullet}(V)$ with differentials $\delta$ and
    $\partial$ the \emph{van Est double complex} of $V$.  Its total
    complex will be denoted by
    \begin{equation}
        \CVanEst^\bullet(V)
            \coloneqq
        \bigoplus_{i=0}^\infty \bigoplus_{k + \ell = i}
        \CVanEst^{k,\ell}(V),
    \end{equation}
    with the total differential
    $\Double \colon \CVanEst^{\bullet}(V) \to
    \CVanEst^{\bullet+1}(V)$, and we will denote its total cohomology
    by $\HVanEst^\bullet(V)$.
\end{definition}
\begin{proposition}[Bimodule structure on $\CVanEst^\bullet(V)$]
    \label{prop:BimoduleCVanEst}
    The total complex $(\CVanEst^\bullet(V),\Double)$ equipped with
    the left action
    \begin{equation}
        Y \acts (X \tensor \phi \tensor \xi)
        \coloneqq
        (Y \tensor X) \tensor \phi \tensor \xi
    \end{equation}
    and the right action
    \begin{equation}
        (X \tensor \phi \tensor \xi) \racts \eta
        \coloneqq
        X \tensor \phi \tensor (\xi \wedge \eta)
    \end{equation}
    is a differential graded
    $(\CCa^\bullet(V),\CCE^\bullet(V))$-bimodule.
\end{proposition}
\begin{proof}
    Both actions do obviously commute.  The fact that both actions are
    differential graded follows directly from
    \autoref{prop:dgModuleStructureCCaVM} and
    \autoref{prop:dgModuleStructureCCEVM}.
\end{proof}

\subsection{Relation to Chevalley-Eilenberg Complex}
\label{subsec:RelationChevalleyEilenbergComplex}

We can relate $\CCE^\bullet(V)$ with $\CVanEst^{\bullet,\bullet}(V)$
using the linear maps
$\imap \colon \CCE^\ell(V) \to \CVanEst^{\bullet,\ell}(V)$, defined by
\begin{equation}
    \label{eq:VanEstiDef}
    \imap(\xi)
    =
    1 \tensor \Unit \tensor \xi,
\end{equation}
and $\pmap \colon \CVanEst^{\bullet,\ell}(V) \to \CCE^\ell(V)$,
defined by
\begin{equation}
    \label{eq:pDef}
    \pmap(X \tensor \phi \tensor \xi)
    =
    \counit(X) \counit(\phi) \cdot \xi.
\end{equation}
Note that we distinguish the unit $1$ in the tensor algebra from the
unit $\Unit$ in the symmetric algebra.  The next lemma shows that for
every fixed $\ell \in \mathbb{N}_0$ these form morphisms of complexes
between $(\CCE^\ell(V),0)$ and $(\CVanEst^{\bullet,\ell},\delta)$,
with $\CCE^\ell(V)$ being concentrated in degree zero.  It can be
shown by an easy computation.
\begin{lemma}
    \label{lemma:deltaiipartial}%
    \
    \begin{lemmalist}
    \item \label{item:deltaiNull}%
        We have $\delta \imap = 0$.
    \item \label{item:pdeltaNull}%
        We have $\pmap \delta = 0$.
    \end{lemmalist}
\end{lemma}

Now consider for each $\ell \in \mathbb{N}_0$ the map
$\hmap \colon \CVanEst^{\bullet,\ell}(V) \to
\CVanEst^{\bullet-1,\ell}(V)$, defined by
\begin{equation}
    \label{eq:hDef}
    \hmap\big(
    (X_1 \tensor \cdots \tensor X_k)
    \tensor
    \phi
    \tensor
    \xi
    \big)
    \coloneqq
    \begin{cases}
        (-1)^{k} \counit(\phi) \cdot
        (X_1 \tensor \cdots \tensor X_{k-1})
        \tensor
        X_k
        \tensor
        \xi
        & \text{if } k \geq 1, \\
        0 & \text{if } k = 0.
    \end{cases}
\end{equation}
The following lemma shows that for each $\ell$ we have a deformation
retract of the form
\begin{equation}
    \label{diag:CEdeformationRetract}
    \begin{tikzcd}
        \CCE^\ell(V)
        \arrow[r,"\imap", shift left = 3pt]
        &\bigl( \CVanEst^{\bullet,\ell}(V),\delta \bigr)
        \arrow[l,"\pmap", shift left = 3pt]
        \arrow[loop,
        out = -30,
        in = 30,
        distance = 30pt,
        start anchor = {[yshift = -7pt]east},
        end anchor = {[yshift = 7pt]east},
        "\hmap"{swap}
        ]
    \end{tikzcd},
\end{equation}
see also \autoref{fig:ColumnAugmentation}.
\newpage
\begin{figure}
    \centering
    \includegraphics{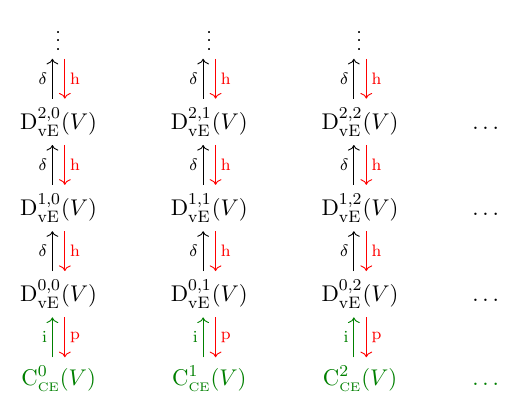}
    \caption{The augmentation of the columns of the van Est double complex}
    \label{fig:ColumnAugmentation}
\end{figure}
\begin{lemma}
    \label{lemma:pideltah}%
    \
    \begin{lemmalist}
    \item \label{item:piNull}%
        We have $\pmap \imap = \id$.
    \item \label{item:deltahhdeltaId}%
        We have $\delta \hmap + \hmap \delta = \id - \imap \pmap$.
    \end{lemmalist}
\end{lemma}
\begin{proof}
    We compute
    $\pmap \imap(\xi) = \pmap(1 \tensor \Unit \tensor \xi) = \xi$ for
    the first part.  For the second part note that $\imap \pmap$ is
    simply the projection onto tensor and symmetric degree $0$.  Then
    \autoref{ex:CoalgebraCohomology}~\ref{item:CoalgebraCohomology_SymmetricAlgebra}
    yields the second part.
\end{proof}

Using the homological perturbation lemma we want to view $\partial$ as
perturbation of $\delta$, see \autoref{sec:HomologicalPerturbation}.
For this we show that $\partial$ is locally nilpotent, and for later
use we also collect some formulas.
\begin{lemma}
    \label{lem:partialh}%
    \
    \begin{lemmalist}
    \item \label{item:partialHPowerNZero} For
        $X \tensor \phi \tensor \xi \in \CVanEst^{k,\ell}(V)$ and
        $n > k$ it holds
        $(\partial \hmap)^n(X \tensor \phi \tensor \xi) = 0$.
    \item \label{item:partialHPowerN} For
        $X = X_1 \tensor \cdots \tensor X_k \in \CCa^k(V)$ and
        $1 \leq n \leq k$ we have
        \begin{equation}
            \label{eq:partialhpowern}
            \begin{split}
                (\partial \hmap)^n (X \tensor \Unit \tensor 1)
                &= (-1)^n
                (
                X_1
                \tensor
                \cdots
                \tensor
                X_{k-n}
                )
                \tensor
                (X_{k-n+1})_\sweedler{0}
                \\
                &\qquad
                \tensor
                \pr_V((X_{k-n+1})_\sweedler{1})
                \wedge
                \pr_V(X_{k-n+2})
                \wedge \cdots \wedge
                \pr_V(X_k).
            \end{split}
        \end{equation}
    \end{lemmalist}
\end{lemma}
\begin{proof}
    The first part can be shown by counting degrees: For this note
    that
    $\partial \hmap \colon \CVanEst^{k,\ell}(V) \to
    \CVanEst^{k-1,\ell+1}(V)$.  Thus
    $(\partial \hmap )^n(X \tensor \phi \tensor \xi) \in
    \CVanEst^{k-n,\ell+n}(V)$, and hence
    $(\partial \hmap)^n(X \tensor \phi \tensor \xi) = 0$ for all
    $n > k$.  We prove the second part by induction.  Thus let
    $X = X_1 \tensor \cdots \tensor X_k \in \CCa^k(V)$ with $1 \leq k$.
    Then
    \begin{align*}
        \partial \hmap (X \tensor \Unit \tensor 1)
        &=
        (-1)^{k} \cdot
        \partial \bigl(
        (X_1 \tensor \cdots \tensor X_{k-1})
        \tensor
        X_k
        \tensor
        1
        \bigr)
        \\
        &=
        -
        (X_1 \tensor \cdots \tensor X_{k-1})
        \tensor
        (X_k)_\sweedler{0}
        \tensor
        \pr_V\bigl((X_k)_\sweedler{1}\bigr).
    \end{align*}
    Now assume that \eqref{eq:partialhpowern} holds for some
    $1 \leq n \leq k - 1$.  Then
    \begin{align*}
        (\partial \hmap)^{n+1}(X \tensor \Unit \tensor 1)
        &=
        (-1)^n (\partial \hmap)\Bigl(
        (
        X_1
        \tensor
        \cdots
        \tensor
        X_{k-n}
        )
        \tensor
        (X_{k-n+1})_\sweedler{0}
        \\
        &\qquad
        \tensor
        \pr_V((X_{k-n+1})_\sweedler{1})
        \wedge
        \pr_V(X_{k-n+2})
        \wedge \cdots \wedge
        \pr_V(X_k)
        \Bigr)
        \\
        &=
        (-1)^{k} \cdot
        \partial \bigl(
        (
        X_{1}
        \tensor
        \cdots
        \tensor
        X_{k-n-1}
        )
        \tensor
        X_{k-n}
        \\
        &\qquad
        \tensor
        \pr_V(X_{k-n+1})
        \wedge
        \cdots
        \wedge
        \pr_V(X_k)
        \bigr)
        \\
        &=
        (-1)^{n+1}
        (
        X_1
        \tensor
        \cdots
        \tensor
        X_{k-(n+1)}
        )
        \tensor
        (X_{k-(n+1)+1})_\sweedler{0}
        \\
        &\qquad
        \tensor
        \pr_V((X_{k-(n+1)+2})_\sweedler{1})
        \wedge
        \pr_V(X_{k-(n+1)+3})
        \wedge
        \cdots
        \wedge
        \pr_V(X_k).
    \end{align*}
\end{proof}

\begin{proposition}[Chevalley-Eilenberg deformation retract]
    \label{prop:CEDeformationRetract}%
    The following is a deformation retract
    \begin{equation}
        \label{diag:CEperturbedDeformationRetract}
        \begin{tikzcd}
            \bigl( \CCE^\bullet(V), 0 \bigr)
            \arrow[r,"\imap", shift left = 3pt]
            &\bigl( \CVanEst^{\bullet}(V),\Double \bigr)
            \arrow[l,"\Pmap", shift left = 3pt]
            \arrow[loop,
            out = -30,
            in = 30,
            distance = 30pt,
            start anchor = {[yshift = -7pt]east},
            end anchor = {[yshift = 7pt]east},
            "\Hmap"{swap}
            ]
        \end{tikzcd}
    \end{equation}
    with
    \begin{equation}
        \Pmap
        =
        \pmap \sum_{n=0}^{\infty}(-1)^n(\partial \hmap)^n,
            \qquad\text{and}\qquad
        \Hmap
        =
        \hmap \sum_{n=0}^{\infty}(-1)^n(\partial \hmap)^n.
    \end{equation}
    In particular, we have the following statements:
    \begin{propositionlist}
    \item It holds $\Pmap \imap = \id$.
    \item It holds
        $\Double \Hmap + \Hmap \Double = \id - \imap \Pmap$.
    \end{propositionlist}
\end{proposition}
\begin{proof}
    The first part follows directly from $\hmap \imap = 0$ and
    $\pmap \imap = \id$.  Since by
    \autoref{lem:partialh}~\ref{item:partialHPowerNZero} the
    perturbation is locally nilpotent we know that
    \begin{equation*}
        (\id + \partial \hmap)^{-1}
        =
        \sum_{n=0}^{\infty}(-1)^n(\partial \hmap)^n,
    \end{equation*}
    see \autoref{sec:HomologicalPerturbation}.  Applying now the
    perturbation lemma in the sense of
    \autoref{cor:PerturbationLemmaII} to
    \eqref{diag:CEdeformationRetract} we obtain the above deformation
    retract.  The only thing left to show is that the perturbed
    differential on $\CCE^\bullet(V)$ is indeed the zero map.  For
    this note that
    $\partial \imap(\xi) = \partial(1 \tensor \Unit \tensor \xi) = 0$.
    By \autoref{prop:HomologicalPerturbation} the perturbed
    differential is given by
    \begin{align*}
        \pmap \sum_{n=0}^{\infty}(-1)^n
        (\partial \hmap)^n \partial \imap = 0.
    \end{align*}
\end{proof}
\begin{corollary}
    \label{cor:CEandVEcohomology}%
    Let $V$ be a $\ring{R}$-module.  Then
    $\imap \colon \Anti^\bullet V \to \HVanEst^\bullet(V)$ is an
    isomorphism with inverse $\Pmap$.
\end{corollary}

Thus we have shown that $\CCE^\bullet(V)$ and $\CVanEst^\bullet(V)$
are quasi-isomorphic.  This quasi-isomorphism is even compatible with
the $\CCE^\bullet(V)$-module structure:
\begin{lemma}
    \label{lem:ImapIsModuleHom}%
    The map $\imap \colon \CCE^\bullet(V) \to \CVanEst^\bullet(V)$ is
    a morphism of differential graded right $\CCE^\bullet(V)$-modules.
\end{lemma}
\begin{proof}
    By \autoref{lemma:deltaiipartial} we know that $\imap$ is
    compatible with the differentials.  Moreover, for
    $\xi \in \CCE^\bullet(V)$ and $\eta \in \CCE^\ell(V)$ we have
    \begin{equation*}
        \imap(\xi \wedge \eta)
        =
        1 \tensor \Unit \tensor (\xi \wedge \eta)
        =
        \imap(\xi) \racts \eta.
    \end{equation*}
\end{proof}

\subsection{Relation to the Coalgebra Complex}
\label{sec:RelationCoalgebraComplex}

Consider for each fixed $k \in \mathbb{N}_0$ the maps
$\jmap \colon \CCa^k(V) \to \CVanEst^{k,\bullet}(V)$, defined by
\begin{equation}
    \label{eq:jNullDef}
    \jmap(X)
    \coloneqq
    X \tensor \Unit \tensor 1,
\end{equation}
and $\qmap \colon \CVanEst^{k,\bullet}(V) \to \CCa^k(V)$, defined by
\begin{equation}
    \label{eq:qDef}
    \qmap(X \tensor \phi \tensor \xi)
    \coloneqq
    \counit(\phi) \counit(\xi) \cdot X.
\end{equation}
The next lemma shows that for every fixed $k \in \mathbb{N}_0$ these
form morphisms of complexes between $\CCa^k(V)$ and
$(\CVanEst^{k,\bullet},\partial)$, with $\CCa^k(V)$ being concentrated
in degree zero.  This is again a simple computation.
\begin{lemma}
    \label{lem:partialjjdelta}%
    \
    \begin{lemmalist}
    \item \label{item:partialjNull}%
        We have $\partial \jmap = 0$.
    \item \label{item:qpartialNull}%
        We have $\qmap \partial = 0$.
    \end{lemmalist}
\end{lemma}

Now consider for each $k \in \mathbb{N}_0$ the map
$\kmap \colon \CVanEst^{k,\bullet}(V) \to \CVanEst^{k,\bullet-1}(V)$,
defined by
\begin{equation}
    \label{eq:kNullDef}
    \begin{split}
        \kmap(
        X
        \tensor
        \phi
        \tensor
        \xi
        )
        &\coloneqq
        \frac{(-1)^{k}}{r + \ell} \cdot
        X
        \tensor
        (\phi \vee \pr_V(\xi_\sweedler{-1}) )
        \tensor
        \xi_\sweedler{0},
    \end{split}
\end{equation}
for
$X \tensor \phi \tensor \xi \in \Tensor^k\Sym V \tensor \Sym^r V
\tensor \Anti^\ell V$.

\begin{figure}
    \centering
    \includegraphics{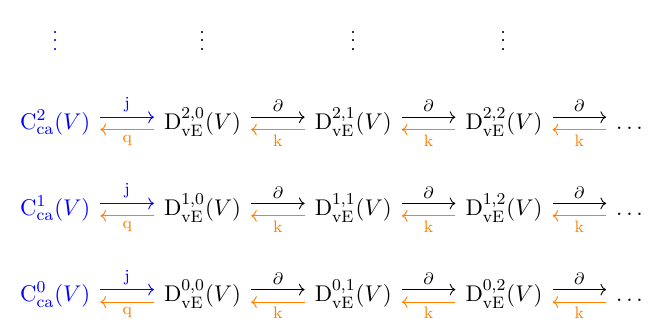}
    \caption{The augmentation of the rows of the van Est double complex.}
    \label{fig:RowAugmentation}
\end{figure}

\begin{lemma}
    \label{lem:CoalgebraDeformationRetract}%
    For each $k \in \mathbb{N}_0$ we have a deformation retract
    \begin{equation}
        \begin{tikzcd}
            \CCa^k(V)
            \arrow[r,"\jmap", shift left = 3pt]
            &\bigl( \CVanEst^{k,\bullet}(V),\partial \bigr)
            \arrow[l,"\qmap", shift left = 3pt]
            \arrow[loop,
            out = -30,
            in = 30,
            distance = 30pt,
            start anchor = {[yshift = -7pt]east},
            end anchor = {[yshift = 7pt]east},
            "\kmap"{swap}
            ]
        \end{tikzcd}
        ,
    \end{equation}
    see also \autoref{fig:RowAugmentation}.  In particular, we have
    the following:
    \begin{lemmalist}
    \item \label{item:qjNull}%
        We have $\qmap \jmap = \id$.
    \item \label{item:partialkIdMinusjq}%
        We have $\partial \kmap + \kmap \partial = \id - \jmap \qmap$.
    \end{lemmalist}
\end{lemma}
\begin{proof}
    For the first part compute
    $\qmap \jmap (X) = \qmap (X \tensor \Unit \tensor 1) = X$.  The
    second part is just the polynomial Poincaré Lemma from
    \eqref{eq:PoincareLemma} in positive degrees.
\end{proof}

In fact, this is a special deformation retract, see
\autoref{def:SpecialDeformationRetract}:
\begin{lemma}
    \label{lem:groupSpecialDeformationRetract}%
    \
    \begin{lemmalist}
    \item \label{item:kk}%
        We have $\kmap^2 = 0$.
    \item \label{item:qk}%
        We have $\qmap \kmap = 0$.
    \item \label{item:kj}%
        We have $\kmap \jmap = 0$.
    \end{lemmalist}
\end{lemma}
\begin{proof}
    The first part follows since
    $\xi_\sweedler{-2} \tensor \xi_\sweedler{-1} = -\xi_\sweedler{-1}
    \tensor \xi_\sweedler{-2}$.  Since $\kmap$ raises the symmetric
    degree by $1$ and $\qmap$ projects to symmetric degree $0$ we get
    the second part.  The last part follows directly from
    $\pr_V(1) = 0$.
\end{proof}

We want to view $\delta$ as perturbation of $\partial$.  For later use
we also collect some formulas.
\begin{lemma}
    \label{lem:deltaKPowerN}%
    \
    \begin{lemmalist}
    \item \label{item:deltaKPowerNZero} For
        $X \tensor \phi \tensor \xi \in \CVanEst^{k,\ell}(V)$ and
        $n > \ell$ we have
        $(\delta \kmap)^{n}(X \tensor \phi \tensor \xi) = 0$.
    \item \label{item:FormularDeltaKPowerN} Let
        $\xi \in \Anti^\ell V$.  Then
        \begin{equation}
            (\delta \kmap)^n \imap(\xi)
            = (-1)^n \frac{(\ell-n)!}{\ell!} \cdot
            \bigl(\pr_V(\xi_\sweedler{-n})
            \tensor
            \cdots
            \tensor
            \pr_V(\xi_\sweedler{-1})
            \bigr)
            \tensor
            \Unit
            \tensor
            \xi_\sweedler{0}
        \end{equation}
        for all $n \in \mathbb{N}_0$.
    \end{lemmalist}
\end{lemma}
\begin{proof}
    The first part can be shown by counting degrees: For this note
    that
    $\delta \kmap \colon \CVanEst^{k,\ell}(V) \to
    \CVanEst^{k+1,\ell-1}(V)$.  Thus
    $(\delta \kmap )^n(X \tensor \phi \tensor \xi) \in
    \CVanEst^{k+n,\ell-n}(V)$, and hence
    $(\delta \kmap)^n(X \tensor \phi \tensor \xi) = 0$ for all
    $n > \ell$.  The second part is clear for $n=0$.  Suppose the
    equation holds for some $n \in \mathbb{N}_0$, then
    \begin{align*}
        (\delta \kmap)^{n+1}\imap(\xi)
        &=
        (-1)^n
        \frac{(\ell-n)!}{\ell!} \delta \kmap \Bigl(
        \bigl(\pr_V(\xi_\sweedler{-n})
        \tensor
        \cdots
        \tensor
        \pr_V(\xi_\sweedler{-1})
        \bigr)
        \tensor
        \Unit
        \tensor
        \xi_\sweedler{0}
        \Bigr)
        \\
        &=
        \frac{(\ell-(n+1))!}{\ell!} \delta \Bigl(
        \bigl(\pr_V(\xi_\sweedler{-(n+1)})
        \tensor
        \cdots
        \tensor
        \pr_V(\xi_\sweedler{-2})
        \bigr)
        \tensor
        \pr_V(\xi_\sweedler{-1})
        \tensor
        \xi_\sweedler{0}
        \Bigr)
        \\
        &=
        (-1)^{n+1}\frac{(\ell-(n+1))!}{\ell!}
        \bigl(
        \pr_V(\xi_\sweedler{-(n+1)})
        \tensor
        \cdots
        \tensor
        \pr_V(\xi_\sweedler{-1})
        \bigr)
        \tensor
        \Unit
        \tensor
        \xi_\sweedler{0}.
    \end{align*}
    For the third line we used that
    $\delta_\Ca(\pr_V(\xi_\sweedler{-n}) \tensor \cdots \tensor
    \pr_V(\xi_\sweedler{-1})) = 0$.
\end{proof}

\begin{proposition}[Coalgebra complex deformation retract]
    \label{prop:coalgebraDeformationRetract}
    The following is a deformation retract
    \begin{equation}
        \label{diag:CgrNormPerturbedDeformationRetract}
        \begin{tikzcd}
            \bigl( \CCa^\bullet(V), \delta_\Ca \bigr)
            \arrow[r,"\jmap", shift left = 3pt]
            &\bigl( \CVanEst^{\bullet}(V),\Double \bigr)
            \arrow[l,"\Qmap", shift left = 3pt]
            \arrow[loop,
            out = -30,
            in = 30,
            distance = 30pt,
            start anchor = {[yshift = -7pt]east},
            end anchor = {[yshift = 7pt]east},
            "\Kmap"{swap}
            ]
        \end{tikzcd}
    \end{equation}
    with
    \begin{equation}
        \Qmap
        \coloneqq
        \qmap\sum_{n=0}^{\infty}(-1)^n(\delta \kmap)^n,
            \qquad\text{and}\qquad
        \Kmap
        \coloneqq
        \kmap\sum_{n=0}^{\infty}(-1)^n(\delta \kmap)^n.
    \end{equation}
    In particular, we have the following results:
    \begin{propositionlist}
    \item We have $\Qmap \jmap = \id$.
    \item We have $\Double \Kmap + \Kmap \Double = \id - \jmap \Qmap$.
    \end{propositionlist}
\end{proposition}
\begin{proof}
    Since by \autoref{lem:deltaKPowerN}~\ref{item:deltaKPowerNZero}
    the perturbation is locally nilpotent we know
    \begin{equation*}
        (\id + \delta \kmap)^{-1} = \sum_{n=0}^{\infty}(-1)^n(\delta \kmap)^n.
    \end{equation*}
    Applying now the perturbation lemma in the sense of
    \autoref{cor:PerturbationLemmaII} to
    \eqref{diag:CEdeformationRetract} we obtain the above deformation
    retract.  The only thing left to show is that the perturbed
    differential on $\CCa^\bullet(V)$ is indeed the coalgebra
    differential.  For this note that $\kmap \delta \jmap = 0$.  By
    \autoref{prop:HomologicalPerturbation} this differential is given
    by
    \begin{equation*}
        \qmap \sum_{n=0}^{\infty}(-1)^n(\delta \kmap)^n \delta \jmap(X)
        = \qmap \delta \jmap(X)
        = \qmap (\delta_\Ca(X) \tensor \Unit \tensor 1)
        = \delta_\Ca(X).
    \end{equation*}
\end{proof}

In fact, we obtain a special deformation retract:
\begin{proposition}
    \label{proposition:KQSpecialDefRetract}%
    \
    \begin{propositionlist}
    \item \label{item:KKNull}%
        We have $\Kmap^2 = 0$.
    \item \label{item:QKNull}%
        We have $\Qmap \Kmap = 0$.
    \item \label{item:KjNull}%
        We have $\Kmap \jmap = 0$.
    \end{propositionlist}
\end{proposition}
\begin{proof}
    By \autoref{lem:groupSpecialDeformationRetract} the homotopy
    retract we perturb in \autoref{prop:coalgebraDeformationRetract}
    is in fact a special deformation retract.  Thus by
    \autoref{cor:SpecialDeformationRetract} the perturbation is also a
    special deformation retract.
\end{proof}
\begin{corollary}
    \label{cor:CgrandVEcohomology}%
    Let $V$ be a $\ring{R}$-module.  Then
    $\jmap \colon \HCa^\bullet(V) \to \HVanEst^\bullet(V)$ is an
    isomorphism with inverse $\Qmap$.
\end{corollary}

Thus we obtain a quasi-isomorphism between $\CCa^\bullet(V)$ and
$\CVanEst^\bullet(V)$.  Moreover, this quasi-isomorphism is a morphism
of $\CCa^\bullet(V)$-modules:
\begin{lemma}
    \label{lem:JmapIsModuleHom}%
    The map $\jmap \colon \CCa^\bullet(V) \to \CVanEst^\bullet(V)$ is
    a morphism of differential graded left $\CCa^\bullet(V)$-modules.
\end{lemma}
\begin{proof}
    By \autoref{lem:partialjjdelta} we know that $\jmap$ is compatible
    with the differentials.  Moreover, for $X \in \CCa^{k_1}(V)$ and
    $Y \in \CCa^{k_2}(V)$ we have
    \begin{equation*}
        \jmap(Y \tensor X)
        =
        (Y \tensor X) \tensor \Unit \tensor 1
        =
        Y \acts \jmap(X).
    \end{equation*}
\end{proof}

\subsection{The Van Est Maps}
\label{sec:VanEstMaps}

In the last two sections we have constructed homotopy retracts
\begin{equation}
    \begin{tikzcd}
        \CCa^\bullet(V)
        \arrow[r,"\jmap",shift left = 3pt]
        &\CVanEst^{\bullet}(V)
        \arrow[l,"\Qmap", shift left = 3pt]
        \arrow[r,"\Pmap", shift left = 3pt]
        &\CCE^\bullet(V)
        \arrow[l,"\imap",shift left = 3pt]
    \end{tikzcd}
\end{equation}
with $\imap$ and $\jmap$ being quasi-isomorphisms with quasi-inverses
$\Pmap$ and $\Qmap$, respectively.  Thus we immediately see that
$\Pmap \circ \jmap$ is a quasi-isomorphism.  But in fact we can
achieve more: We can even construct a homotopy retract between
$\CCa^\bullet(V)$ and $\CCE^\bullet(V)$ directly.  For this we need to
collect some properties.
\begin{lemma}
    \label{lem:hkpk}%
    \
    \begin{lemmalist}
    \item \label{item:hk}%
        We have $\hmap \kmap = 0$.
    \item \label{item:HK}%
        We have $\Hmap \Kmap = 0$.
    \item \label{item:pk}%
        We have $\pmap \kmap = 0$.
    \item \label{item:PK}%
        We have $\Pmap \Kmap = 0$.
    \end{lemmalist}
\end{lemma}
\begin{proof}
    Since $\hmap$ raises the symmetric degree by $1$ and $\hmap$
    projects onto symmetric degree $0$ we immediately obtain
    $\hmap \kmap = 0$.  Then it follows directly
    \begin{equation*}
        \Hmap \Kmap
        =
        (\id + \hmap \partial)^{-1} \hmap
        \kmap (\id + \delta \kmap)^{-1}
        =
        0.
    \end{equation*}
    Similarly, $\pmap$ projects onto symmetric degree $0$ and thus
    $\pmap \kmap = 0$ holds.  Then we obtain
    \begin{equation*}
        \Pmap \Kmap
        =
        \pmap \sum_{i=0}^{\infty} (-1)^i (\partial \kmap)^i
        \kmap (\id + \delta \kmap)^{-1}
        =
        \pmap \kmap (\id + \delta \kmap)^{-1}
        =
        0,
    \end{equation*}
    since $\kmap^2 = 0$.
\end{proof}

\begin{theorem}[Van Est deformation retract]
    \label{thm:VanEstDeformationRetract}%
    Let $V$ be an $\ring{R}$-module.  Then
    \begin{equation}
        \label{diag:VanEstMapsDeformationRetract}
        \begin{tikzcd}[column sep=large]
            \bigl( \CCE^\bullet(V), 0 \bigr)
            \arrow[r,"\VanEstInt", shift left = 3pt]
            &\bigl( \CCa^{\bullet}(V),\delta_\Ca \bigr)
            \arrow[l,"\VanEstDiff", shift left = 3pt]
            \arrow[loop,
            out = -30,
            in = 30,
            distance = 30pt,
            start anchor = {[yshift = -7pt]east},
            end anchor = {[yshift = 7pt]east},
            "\VanEstHomo"{swap}
            ]
        \end{tikzcd}
    \end{equation}
    is a deformation retract, where
    \begin{align}
        \VanEstDiff
        \coloneqq
        \Pmap \circ \jmap,
            \qquad
        \VanEstInt
        \coloneqq
        \Qmap \circ \imap
            \qquad\text{and}\qquad
        \VanEstHomo
        \coloneqq
        \Qmap \Hmap \jmap.
    \end{align}
    In particular, the following holds:
    \begin{theoremlist}
    \item We have $\VanEstDiff \circ \VanEstInt = \id$.
    \item We have
        $\delta_\Ca \VanEstHomo + \VanEstHomo \delta_\Ca = \id -
        \VanEstInt \circ \VanEstDiff$.
    \end{theoremlist}
\end{theorem}
\begin{proof}
    We obtain both statements by composition of homotopy retracts, see
    \autoref{sec:HomotopyRetracts}.  Recall that $(\jmap,\Qmap,\Kmap)$
    is a deformation retract by \autoref{prop:coalgebraDeformationRetract},
    and thus $(\Qmap,\jmap,0)$ is a homotopy retract.
    Similarly, we know that $(\imap,\Pmap,\Hmap)$ is a
    deformation retract by \autoref{prop:CEDeformationRetract},
    and thus $(\Pmap,\imap,0)$ is a homotopy retract.
    Composing $(\jmap,\Qmap,\Kmap)$ with $(\Pmap,\imap,0)$
    we obtain the homotopy retract
    \begin{equation*}
        (\Pmap,\imap,0) \circ (\jmap,\Qmap,\Kmap)
        =
        (\Pmap \jmap, \Qmap \imap, \Pmap \Kmap \imap)
        =
        (\VanEstDiff, \VanEstInt, 0).
    \end{equation*}
    Here we used \autoref{lem:hkpk}~\ref{item:PK}.  In other words, we
    have
    $\id - \VanEstDiff \circ \VanEstInt = \partial 0 + 0 \partial =
    0$.  This shows the first part.  For the second part we compose
    $(\imap,\Pmap;\Hmap)$ and $(\Qmap,\jmap,0)$ to obtain the
    deformation retract
    \begin{equation*}
        (\Qmap,\jmap,0)
        \circ
        (\imap,\Pmap, \Hmap)
        =
        (\Qmap \imap, \Pmap \jmap, \Qmap \Hmap \jmap)
        =
        (\VanEstInt, \VanEstDiff, \VanEstHomo).
    \end{equation*}
    Hence we have
    $\delta_\Ca \VanEstHomo + \VanEstHomo \delta_\Ca = \id -
    \VanEstInt \circ \VanEstDiff$.
\end{proof}

\begin{remark}
    \label{rem:nontrivialLiecoalg}%
    \begin{remarklist}
    \item
    Note that both of the complexes $\CCE^\bullet(V)$ and
    $\CCa^\bullet(V)$ have interpretations even if the Lie cobracket
    is not trivial: $\CCE^\bullet(V)$ is the usual Chevalley-Eilenberg
    complex of a Lie coalgebra, and $\CCa^\bullet(V)$ is the coalgebra
    complex of the universal enveloping coalgebra of the Lie
    coalgebra. Since our Lie coalgebra $V$ is abelian, one can show
    that its universal enveloping coalgebra is indeed $\Sym V$, see
    \cite{michaelis:1980a}. We are very certain that one can use the
    deformation retract \eqref{diag:VanEstMapsDeformationRetract} and
    the dual Poincaré-Birkhoff-Witt theorem (see
    \cite{michaelis:1985a}) to perturb $\delta_\Ca$ in the direction
    of a non-trivial Lie cobracket in order to obtain a deformation
    retract connecting the Chevalley-Eilenberg complex of a locally
    finite Lie coalgebra with the coalgebra complex of its universal
    enveloping coalgebra. Since this is not directly connected to the
    overall aim of this paper, this attempt will be part of a future
    work.
    \item Moreover, the deformation retract
          \autoref{diag:VanEstMapsDeformationRetract}, or better said a
          dual version of it, can be used to build an explicit deformation
          retract for a (dg) Lie algebra $\liealg{g}$ and its
           universal enveloping algebra $\mathcal{U}(\liealg{g})$
             \begin{equation}
        \begin{tikzcd}[column sep=large]
            \mathcal{U}(\liealg{g})
            \arrow[r, shift left = 3pt]
            &\bigl(\Omega\Anti\liealg{g},\partial)
            \arrow[l, shift left = 3pt]
            \arrow[loop,
            out = -30,
            in = 30,
            distance = 30pt,
            start anchor = {[yshift = -7pt]east},
            end anchor = {[yshift = 7pt]east},
            ]
        \end{tikzcd},
    \end{equation}
    where $\Anti\liealg{g}$ is the Chevalley-Eilenberg complex with
    its canonical dg coalgebra structure and $\Omega$ is the cobar
    resolution of this coalgebra.  This implies immediately that the
    universal enveloping algebra of a (dg) Lie algebra is Koszul. This
    line of thought together with an extended version of the
    Poincaré-Birkoff-Witt theorem for $L_\infty$-algebras is carried
    out in \cite{getzler:2024a}, but using the homotopy constructed in
    \cite{dewilde.lecomte:1995a}. In the appendix of
    \cite{getzler:2024a} a detailed comparison between the two
    homotopies is given.
\end{remarklist}
 \end{remark}

We call $\VanEstDiff$ the \emph{van Est differentiation} map, while we
call $\VanEstInt$ the \emph{van Est integration}.  It follows directly
that $\HCa^\bullet(V) \simeq \Anti^\bullet V$ via the van Est maps.
In fact, this is actually an isomorphism of algebras as the next
result shows.
\begin{corollary}
    \label{cor:CaCohomologyComputed}%
    Let $V$ be an $\ring{R}$-module.  Then
    $\VanEstDiff \colon \HCa^\bullet(V) \to \Anti^\bullet V$ is an
    isomorphism of graded algebras with inverse $\VanEstInt$.
\end{corollary}
\begin{proof}
    It remains to show that $\VanEstDiff$ is an algebra homomorphism.
    For this recall that in cohomology $\Pmap$ is the inverse of
    $\imap$ and $\Pmap$ is right $\HCE^\bullet(V)$-linear.  Hence for
    $[A] \in \HVanEst^\bullet(V)$ we have
    \begin{align*}
        [1 \tensor \Unit \tensor 1] \racts \Pmap[A]
        &=
        \imap\Pmap[1 \tensor \Unit \tensor 1] \racts \Pmap[A]
        \\
        &=
        \imap(\Pmap[1 \tensor \Unit \tensor 1] \wedge \Pmap[A])
        \\
        &=
        \imap \Pmap [A]
        \\
        &=
        [A].
    \end{align*}
    Similarly, we get
    $[A] = \Qmap[A] \acts [1 \tensor \Unit \tensor 1]$.  With this we
    get for two classes $[A], [B] \in \HVanEst^\bullet(V)$
    \begin{align*}
        \VanEstDiff([A] \tensor[] [B])
        &=
        \Pmap \jmap ([A] \tensor[] [B])
        \\
        &=
        \Pmap([A] \acts \jmap[B])
        \\
        &=
        \Pmap([A] \acts [1 \tensor \Unit \tensor 1] \racts \Pmap\jmap[B])
        \\
        &=
        \Pmap([A] \acts [1 \tensor \Unit \tensor 1]) \wedge \Pmap\jmap([B])
        \\
        &=
        \Pmap(\Qmap\jmap[A] \acts [1 \tensor \Unit \tensor 1]) \wedge \Pmap\jmap[B]
        \\
        &=
        \Pmap \jmap [A] \wedge \Pmap \jmap [B]
        \\
        &=
        \VanEstDiff[A] \wedge \VanEstDiff[B].
    \end{align*}
\end{proof}
\begin{figure}
    \centering
    \includegraphics{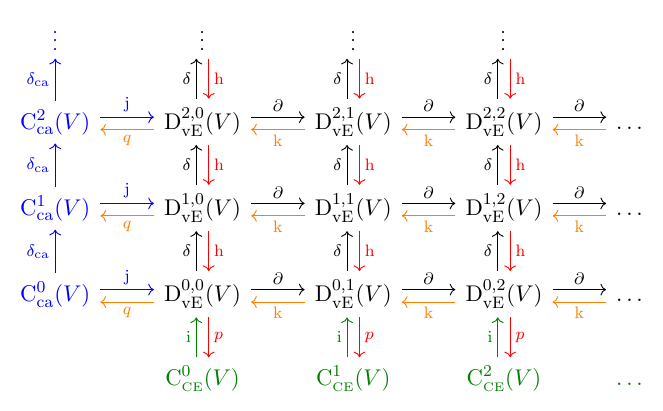}
    \caption{The double complex with the augmentation maps and the
      homotopies.}
    \label{fig:doublecomplex}
\end{figure}

Even though $\Pmap$ and $\Qmap$ were obtained by a perturbation
formula we can determine $\VanEstDiff$ and $\VanEstInt$ very
explicitly.
\begin{proposition}[Van Est maps]
    \label{prop:VanEstMaps}%
    \
    \begin{propositionlist}
    \item It holds
        \begin{equation}
            \label{eq:FormulaVanEstDiff}
            \VanEstDiff(X_1 \tensor \cdots \tensor X_k)
            =
            \pr_V(X_1) \wedge \cdots \wedge \pr_V(X_k)
        \end{equation}
        for $X_1, \ldots, X_k \in \Sym V$.
        In particular, $\VanEstDiff \colon \CCa^\bullet(V) \to \CCE^\bullet(V)$
        is a morphism of dg algebras.
    \item It holds
        \begin{equation}
            \label{eq:FormulaVanEstInt}
            \VanEstInt(\xi_1 \wedge \cdots \wedge \xi_\ell)
            =
            \frac{1}{\ell!}
            \sum_{\sigma \in \Perm_\ell}
            \sign(\sigma)
            \xi_{\sigma(1)} \tensor \cdots \tensor \xi_{\sigma(\ell)}
        \end{equation}
        for $\xi_1, \ldots, \xi_\ell \in V$.
    \end{propositionlist}
\end{proposition}
\begin{proof}
    For the first part consider
    $X = X_1 \tensor \cdots \tensor X_k \in \CCa^k(V)$.  Then
    \begin{align*}
        \VanEstDiff(X)
        =
        \pmap \sum_{n=0}^{\infty} (-1)^n (\partial \hmap)^n \jmap(X)
        =
        \pmap (-1)^k (\partial \hmap)^k \jmap (X)
        =
        \pr_V(X_1) \wedge \cdots \wedge \pr_V(X_k),
    \end{align*}
    where we used \autoref{lem:partialh}~\ref{item:partialHPowerN}.
    For the second part let
    $\xi = \xi_1 \wedge \cdots \wedge \xi_\ell \in \Anti^\ell V$ be
    given.  Then
    \begin{align*}
        \VanEstInt(\xi)
        &=
        \qmap \sum_{n=0}^{\infty}(-1)^n(\delta \kmap)^n
        \imap(\xi)
        \\
        &=
        (-1)^\ell \qmap (\delta \kmap)^\ell
        \imap(\xi)
        \\
        &=
        \frac{1}{\ell!}
        \pr_V(\xi_\sweedler{-\ell})
        \tensor
        \cdots
        \tensor
        \pr_V(\xi_\sweedler{-1})
    \end{align*}
    by \autoref{lem:deltaKPowerN}~\ref{item:FormularDeltaKPowerN} and
    \autoref{lem:exponentialproj}. The result follows since $\xi$ is
    antisymmetric.
\end{proof}

\subsection{Functoriality}
\label{sec:Functoriality}

We consider now two different rings of scalars, $\ring{R}$ and
$\ring{S}$ containing $\field{Q}$, together with a unital ring
morphism $\vartheta\colon \ring{R} \to \ring{S}$. In some
applications, we have $\ring{S} = \ring{R}$ and $\vartheta$ is a ring
automorphism, but for the time being this will not be assumed.

Next, let $V$ be a module over $\ring{R}$ and $W$ a module over
$\ring{S}$. We are then interested in a module morphism
$\Phi\colon V \to W$ along $\vartheta$,
i.e. we have $\Phi(\lambda v) = \vartheta(\lambda) \Phi(v)$ for
$\lambda \in \ring{R}$ and $v \in V$.

This is enough to make the map $\Phi$ well-defined on tensor powers of
$V$ mapping them into the corresponding tensor powers of $W$ by
applying $\Phi$ to each tensor factor separately. For
$\ring{R} = V^{\tensor 0}$ we use $\vartheta$ to map this into
$\ring{S} = W^{\tensor 0}$. This way, we arrive at algebra morphisms
$\Phi \colon \Sym V \to \Sym W$ and $\Phi \colon \Anti V \to \Anti W$,
both over $\vartheta$.  Similarly, $\Phi$ extends to
$\Phi \colon \Tensor \Sym V \to \Tensor \Sym W$, again over
$\vartheta$.  To lighten notation a little we will denote all these
maps with the same symbol $\Phi$.

Our goal is now to show that $\Phi$ induces morphisms between the
various homotopy retracts introduced in
\autoref{sec:VanEstDoubleComplex}.  By the definition of $\delta$ and
$\partial$ it is easy to see that $\Phi$ commutes with these
differentials.  Similarly, one can check quickly that $\Phi$ also
commutes with the homotopies $\hmap$ and $\kmap$.  Hence for every
$\ell \in \mathbb{N}_0$ it induces a morphism
\begin{equation}
    \label{eq:PhiMorphAlongVarthetaCCE}
    \begin{tikzcd}[row sep = large]
        \CCE^\ell(V)
        \arrow[r,"\imap^V", shift left = 3pt]
        \arrow[d,"\Phi"{swap}]
        &\bigl( \CVanEst^{\bullet,\ell}(V),\delta^V \bigr)
        \arrow[l,"\pmap^V", shift left = 3pt]
        \arrow[loop,
        out = -30,
        in = 30,
        distance = 30pt,
        start anchor = {[yshift = -7pt]east},
        end anchor = {[yshift = 7pt]east},
        "\hmap^V"{swap}
        ]
        \arrow[d,"\Phi"]
        \\
        \CCE^\ell(W)
        \arrow[r,"\imap^W", shift left = 3pt]
        &\bigl( \CVanEst^{\bullet,\ell}(W),\delta^W \bigr)
        \arrow[l,"\pmap^W", shift left = 3pt]
        \arrow[loop,
        out = -30,
        in = 30,
        distance = 30pt,
        start anchor = {[yshift = -7pt]east},
        end anchor = {[yshift = 7pt]east},
        "\hmap^W"{swap}
        ]
    \end{tikzcd}
\end{equation}
of homotopy retracts, and for every $k \in \mathbb{N}_0$ it induces a
morphism
\begin{equation}
    \label{eq:PhiMorphAlongVarthetaCca}
    \begin{tikzcd}[row sep = large]
        \CCa^k(V)
        \arrow[r,"\jmap^V", shift left = 3pt]
        \arrow[d,"\Phi"{swap}]
        &\bigl( \CVanEst^{k,\bullet}(V),\partial^V \bigr)
        \arrow[l,"\qmap^V", shift left = 3pt]
        \arrow[loop,
        out = -30,
        in = 30,
        distance = 30pt,
        start anchor = {[yshift = -7pt]east},
        end anchor = {[yshift = 7pt]east},
        "\kmap^V"{swap}
        ]
        \arrow[d,"\Phi"]
        \\
        \CCa^k(W)
        \arrow[r,"\jmap^W", shift left = 3pt]
        &\bigl( \CVanEst^{k,\bullet}(W),\partial^W \bigr)
        \arrow[l,"\qmap^W", shift left = 3pt]
        \arrow[loop,
        out = -30,
        in = 30,
        distance = 30pt,
        start anchor = {[yshift = -7pt]east},
        end anchor = {[yshift = 7pt]east},
        "\kmap^W"{swap}
        ]
    \end{tikzcd}
\end{equation}
of homotopy retracts, always over $\vartheta$.  See
\autoref{sec:HomotopyRetracts} for more on morphisms of homotopy
retracts.  Note that all involved squares commute, i.e. we have
\begin{align}
    \label{eq:PhiCommutesWithStuff}
    \Phi \circ \imap^V
    &=
    \imap^W \circ \Phi,
    &
    \Phi \circ \pmap^V
    &=
    \pmap^W \circ \Phi,
    \\
    \label{eq:PhiCommutesWithMoreStuff}
    \Phi \circ \jmap^V
    &=
    \jmap^W \circ \Phi,
    &
    \Phi \circ \qmap^V
    &= \qmap^W \circ \Phi.
\end{align}
\begin{proposition}
    \label{prop:FunctorialityCEandGrRetracts}%
    Let $\Phi\colon V \to W$ be a module morphism over a unital ring
    morphism $\vartheta\colon \ring{R} \to \ring{S}$ for an
    $\ring{R}$-module $V$ and a $\ring{S}$-module $W$.
    \begin{propositionlist}
    \item The morphism $\Phi \colon V \to W$ induces a morphism
        \begin{equation}
            \begin{tikzcd}[row sep = large, column sep = large]
                \bigl( \CCE^\bullet(V), 0 \bigr)
                \arrow[r,"\imap^V", shift left = 3pt]
                \arrow[d,"\Phi"{swap}]
                &\bigl( \CVanEst^{\bullet}(V),\Double^V \bigr)
                \arrow[l,"\Pmap^V", shift left = 3pt]
                \arrow[loop,
                out = -30,
                in = 30,
                distance = 30pt,
                start anchor = {[yshift = -7pt]east},
                end anchor = {[yshift = 7pt]east},
                "\Hmap^V"{swap}
                ]
                \arrow[d,"\Phi"]
                \\
                \bigl( \CCE^\bullet(W), 0 \bigr)
                \arrow[r,"\imap^W", shift left = 3pt]
                &\bigl( \CVanEst^{\bullet}(W),\Double^W \bigr)
                \arrow[l,"\Pmap^W", shift left = 3pt]
                \arrow[loop,
                out = -30,
                in = 30,
                distance = 30pt,
                start anchor = {[yshift = -7pt]east},
                end anchor = {[yshift = 7pt]east},
                "\Hmap^W"{swap}
                ]
            \end{tikzcd}
        \end{equation}
        of homotopy retracts over $\vartheta$.
    \item The morphism $\Phi \colon V \to W$ induces a morphism
        \begin{equation}
            \begin{tikzcd}[row sep = large, column sep = large]
                \bigl( \CCa^\bullet(V), \delta_\Ca \bigr)
                \arrow[r,"\jmap^V", shift left = 3pt]
                \arrow[d,"\Phi"{swap}]
                &\bigl( \CVanEst^{\bullet}(V),\Double^V \bigr)
                \arrow[l,"\Qmap^V", shift left = 3pt]
                \arrow[loop,
                out = -30,
                in = 30,
                distance = 30pt,
                start anchor = {[yshift = -7pt]east},
                end anchor = {[yshift = 7pt]east},
                "\Kmap^V"{swap}
                ]
                \arrow[d,"\Phi"]
                \\
                \bigl( \CCa^\bullet(W), \delta_\Ca \bigr)
                \arrow[r,"\jmap^W", shift left = 3pt]
                &\bigl( \CVanEst^{\bullet}(W),\Double^W \bigr)
                \arrow[l,"\Qmap^W", shift left = 3pt]
                \arrow[loop,
                out = -30,
                in = 30,
                distance = 30pt,
                start anchor = {[yshift = -7pt]east},
                end anchor = {[yshift = 7pt]east},
                "\Kmap^W"{swap}
                ]
            \end{tikzcd}
        \end{equation}
        of homotopy retracts over $\vartheta$.
    \end{propositionlist}
\end{proposition}
\begin{proof}
    Since $\Phi$ commutes with $\delta$ and $\partial$, and $\Double$
    can be considered as a perturbation of $\delta$ by $\partial$ (or
    vice versa), we know from \autoref{cor:MorphismOfPerturbations}
    that $\Phi$ induces a morphism between the perturbed homotopy
    retracts. In fact, we have to use
    \autoref{cor:MorphismOfPerturbations} in a slightly modified
    way since the map $\Phi$ is only $\mathbb{Z}$-linear but not
    linear over $\ring{R}$ directly. However, it is clear that a
    merely additive map will also commute with the geometric series in
    the construction of the homotopies once it commutes with the
    individual summands.
\end{proof}

Composing these homotopy retracts yields the following compatibility
of $\Phi$ with the van Est maps.
\begin{proposition}
    \label{prop:GroupInvariantVanEst}%
    The morphism $\Phi \colon V \to W$ induces a morphism
    \begin{equation}
        \begin{tikzcd}[row sep = large, column sep = large]
            (\CCE^\bullet(V), 0)
            \arrow[r,"\VanEstInt", shift left = 3pt]
            \arrow[d,"\Phi"{swap}]
            &(\CCa^{\bullet}(V),\delta_\Ca^V)
            \arrow[l,"\VanEstDiff", shift left = 3pt]
            \arrow[loop,
            out = -30,
            in = 30,
            distance = 30pt,
            start anchor = {[yshift = -7pt]east},
            end anchor = {[yshift = 7pt]east},
            "\VanEstHomo^V"{swap}
            ]
            \arrow[d,"\Phi"]
            \\
            (\CCE^\bullet(W), 0)
            \arrow[r,"\VanEstInt", shift left = 3pt]
            &(\CCa^{\bullet}(W),\delta_\Ca^W)
            \arrow[l,"\VanEstDiff", shift left = 3pt]
            \arrow[loop,
            out = -30,
            in = 30,
            distance = 30pt,
            start anchor = {[yshift = -7pt]east},
            end anchor = {[yshift = 7pt]east},
            "\VanEstHomo^W"{swap}
            ]
        \end{tikzcd}
    \end{equation}
    of homotopy retracts over $\vartheta$.
\end{proposition}
\begin{proof}
    Since $\Phi$ commutes with all involved maps by
    \autoref{prop:FunctorialityCEandGrRetracts} it also commutes with
    $\Theta$ and the van Est maps.
\end{proof}

An important application is obtained for the following situation: Let
$\group{G}$ be a group together with an action
$\phi\colon \group{G} \ni g \mapsto \phi_g \in \Aut(\ring{R})$ on
$\ring{R}$ by ring automorphisms. Moreover, let
$\Phi\colon \group{G} \ni g \mapsto \Phi_g \in \Aut_{\mathbb{Z}}(V)$
be a group action of the same group on the module by additive but not
necessarily $\ring{R}$-linear automorphisms of $V$. We want the
compatibility
\begin{equation}
    \label{eq:PhigphigPhig}
    \Phi_g(\lambda v)
    =
    \phi_g(\lambda) \Phi_g(v)
\end{equation}
for all $\lambda \in \ring{R}$, all $v \in V$, and all
$g \in \group{G}$.  Then the van Est maps restrict to the subcomplexes
$\CCE^\bullet(V)^\group{G}$ and $\CCa^\bullet(V)^\group{G}$ of
cochains invariant under the action of $\group{G}$.  It is easy to see
that $\VanEstHomo$ is well-defined on $\CCa^\bullet(V)^\group{G}$ and
hence we obtain a deformation retract
\begin{equation}
    \label{diag:GinvVanEst}
    \begin{tikzcd}[column sep = large]
        \bigl( \CCE^\bullet(V)^\group{G}, 0 \bigr)
        \arrow[r,"\VanEstInt", shift left = 3pt]
        &\bigl( \CCa^{\bullet}(V)^\group{G},\delta_\Ca \bigr)
        \arrow[l,"\VanEstDiff", shift left = 3pt]
        \arrow[loop,
        out = -30,
        in = 30,
        distance = 30pt,
        start anchor = {[yshift = -7pt]east},
        end anchor = {[yshift = 7pt]east},
        "\VanEstHomo"{swap}
        ]
    \end{tikzcd}
    .
\end{equation}
In fact, the van Est maps between $\CCE^\bullet(V)$ and
$\CCa^\bullet(V)$ simply commute with the action of $\group{G}$,
i.e. are $\group{G}$-\emph{equivariant} maps. This is a stronger
statement which implies \eqref{diag:GinvVanEst}.

Note that the cohomology of the invariant Chevalley-Eilenberg
subcomplex is simply given by
$\CCE^\bullet(V)^\group{G} = (\Anti^\bullet V)^\group{G}$.  For the
coalgebra complex we will denote the cohomology of the invariant
subcomplex by $\HCainv{\group{G}}^\bullet(V)$.  Since every
automorphism $g \in \group{G}$ induces an automorphism of the cochain
complex $\CCa^\bullet(V)$, and thus in cohomology, we could also
consider the invariant cohomology $\HCa^\bullet(V)^\group{G}$.  These
cohomologies agree, since
\begin{equation}
    \label{eq:InvariantCohomologiesGeneral}
    \HCainv{\group{G}}^\bullet(V)
    \simeq
    (\Anti^\bullet V)^\group{G}
    \simeq
    (\HCa^\bullet(V))^\group{G},
\end{equation}
via the $\group{G}$-equivariant van Est integration and
differentiation maps.
\begin{example}[Particular scenarios]
    \label{example:ModuleIsDerR}%
    \
    \begin{examplelist}
    \item \label{item:VisDerR} A case of particular interest is
        obtained for $V = \Der(\ring{R})$, the left $\ring{R}$-module
        of derivations of the ring $\ring{R}$. In this case, every
        ring automorphism canonically acts on derivations by
        conjugation. Clearly, the compatibility
        \eqref{eq:PhigphigPhig} is satisfied here, where $\phi$ is the
        identity map on $\Aut(\ring{R})$ and $\Phi$ is the conjugation
        map.
    \item \label{item:RlinearActionOnV} The other extreme case is
        where we have $\ring{R}$-linear automorphisms of the module
        $V$. Here
        $\Phi\colon \Aut_{\ring{R}}(V) \longrightarrow
        \Aut_{\mathbb{Z}}(V)$ is the inclusion and
        $\phi\colon \Aut_{\ring{R}}(V) \longrightarrow \Aut(\ring{R})$
        is the trivial group morphism, i.e. \eqref{eq:PhigphigPhig}
        becomes simplified to $\Phi(\lambda v) = \lambda \Phi(v)$.
        Geometrically, this will correspond to fiberwise group actions
        on bundles.
    \item \label{item:ThingsForSubgroups} The geometrically
        interesting situations will arise not from these two maximal
        groups but for non-trivial subgroups $\group{G}$ in both two
        scenarios.
    \end{examplelist}
\end{example}

\subsection{Infinitesimal Functoriality}
\label{subsec:InfinitesimalFunctoriality}

Closely related to the previous functoriality is the following notion:
we will replace automorphisms by derivations everywhere and obtain
lifts to all of our complexes as well. Morally, this can be obtained
by extending the previous results to formal power series in an
auxiliary parameter $t$ and consider $\ring{R}\formal{t}$ etc. Then
automorphisms starting with the identity in zeroth order are
exponentials of derivations for which the following results can be
recovered from the above ones. However, it is perhaps easier to
directly formulate things infinitesimally.

We consider a derivation $\vartheta \in \Der(\ring{R})$ and a module
derivation $\Phi \in \End_{\mathbb{Z}}(V)$ over $\vartheta$, i.e. an
additive map satisfying the Leibniz rule
\begin{equation}
    \label{eq:LeibnizModuleDerivationOverDerivation}
    \Phi(\lambda v)
    =
    \vartheta(\lambda) v + \lambda \Phi(v)
\end{equation}
for $\lambda \in \ring{R}$ and $v \in V$. We can then extend $\Phi$ to
tensor powers of $V$ by a Leibniz rule with respect to the tensor
product, i.e. by requiring
\begin{equation}
    \label{eq:ExtendDtoTensorPowersV}
    \Phi(v_1 \tensor \cdots \tensor v_n)
    =
    \sum_{i=1}^n
    v_1 \tensor \cdots \tensor v_{i-1}
    \tensor \Phi(v_i) \tensor
    v_{i+1} \tensor \cdots \tensor v_n
\end{equation}
for $v_1, \ldots, v_n \in V$. Even though $\Phi$ is not
$\ring{R}$-linear unless $\vartheta = 0$, this gives a well-defined
extension of $D$ to $\Tensor^n V$. Moreover, this extension respects
the symmetry properties of tensors and hence maps (anti-)symmetric
tensors to (anti-)symmetric ones. It is then a simple verification on
factorizing tensors that this extension $\Phi$ commutes with all
structure maps $\partial$, $\delta$, $\imap$, $\pmap$, $\jmap$,
$\qmap$, $\hmap$, and $\kmap$ of the previous complexes.  From here we
can proceed as for the functoriality statements in
\autoref{sec:Functoriality} and arrive at the fact that also the
perturbed homotopies $\Hmap$ and $\Kmap$ commute with $\Phi$. Finally,
the van Est maps commute with $\Phi$ as well, leading to the following
statement:
\begin{proposition}
    \label{proposition:VanEstInfinitesimalFunctorial}%
    The module derivation $\Phi\colon V \to V$ over the derivation
    $\vartheta \in \Der(\ring{R})$ induces a morphism
    \begin{equation}
        \begin{tikzcd}[row sep = large, column sep = large]
            (\CCE^\bullet(V), 0)
            \arrow[r,"\VanEstInt", shift left = 3pt]
            \arrow[d,"\Phi"{swap}]
            &(\CCa^{\bullet}(V),\delta_\Ca)
            \arrow[l,"\VanEstDiff", shift left = 3pt]
            \arrow[loop,
            out = -30,
            in = 30,
            distance = 30pt,
            start anchor = {[yshift = -7pt]east},
            end anchor = {[yshift = 7pt]east},
            "\VanEstHomo"{swap}
            ]
            \arrow[d,"\Phi"]
            \\
            (\CCE^\bullet(V), 0)
            \arrow[r,"\VanEstInt", shift left = 3pt]
            &(\CCa^{\bullet}(V),\delta_\Ca)
            \arrow[l,"\VanEstDiff", shift left = 3pt]
            \arrow[loop,
            out = -30,
            in = 30,
            distance = 30pt,
            start anchor = {[yshift = -7pt]east},
            end anchor = {[yshift = 7pt]east},
            "\VanEstHomo"{swap}
            ]
        \end{tikzcd}
    \end{equation}
    of homotopy retracts over $\vartheta$.
\end{proposition}

Again, this results in an equivariance property of the van Est maps
for Lie algebra actions, being the infinitesimal analog of the
equivariance properties discussed in the previous section. More
precisely, let $\liealg{g}$ be a Lie algebra over $\mathbb{Z}$ and
consider a Lie algebra action
$\phi\colon \liealg{g} \ni \xi \mapsto \phi_\xi \in \Der(\ring{R})$
together with a Lie algebra action
$\Phi\colon \liealg{g} \ni \xi \mapsto \Phi_\xi \in
\End_{\mathbb{Z}}(V)$ satisfying the additional Leibniz rule
\begin{equation}
    \label{eq:LeibnizForLieAlgebraAction}
    \Phi_\xi (\lambda v)
    =
    \phi_\xi(\lambda) v + \lambda \Phi_\xi(v)
\end{equation}
for all $\lambda \in \ring{R}$, all $v \in V$ and all
$\xi \in \liealg{g}$. Then we get subcomplexes of
$\liealg{g}$-invariant elements, i.e. those being annihilated by
$\Phi_\xi$ for all $\xi \in \liealg{g}$. The van Est maps including
the homotopy restrict to invariant elements thanks to their
equivariance from
\autoref{proposition:VanEstInfinitesimalFunctorial}. This
results in a deformation retract
\begin{equation}
    \label{diag:GinvVanEstLieAlg}
    \begin{tikzcd}[column sep = large]
        \bigl( \CCE^\bullet(V)^\liealg{g}, 0 \bigr)
        \arrow[r,"\VanEstInt", shift left = 3pt]
        &\bigl( \CCa^{\bullet}(V)^\liealg{g},\delta_\Ca \bigr)
        \arrow[l,"\VanEstDiff", shift left = 3pt]
        \arrow[loop,
        out = -30,
        in = 30,
        distance = 30pt,
        start anchor = {[yshift = -7pt]east},
        end anchor = {[yshift = 7pt]east},
        "\VanEstHomo"{swap}
        ]
    \end{tikzcd}
    .
\end{equation}
As before, the equivariance of all involved maps is a stronger
statement which implies \eqref{diag:GinvVanEstLieAlg}.  As before the
cohomology of the invariant Chevalley-Eilenberg subcomplex is simply
given by $\CCE^\bullet(V)^\liealg{g} = (\Anti^\bullet V)^\liealg{g}$.
For the coalgebra complex we will denote the cohomology of the
invariant subcomplex by $\HCainv{\liealg{g}}^\bullet(V)$.  Finally, we
have the invariant cohomology $\HCa^\bullet(V)^\liealg{g}$ also in
this case.  Again the cohomologies agree, since
\begin{equation}
    \label{eq:gInvariantCohomologiesGeneral}
    \HCainv{\liealg{g}}^\bullet(V)
    \simeq
    (\Anti^\bullet V)^\liealg{g}
    \simeq
    (\HCa^\bullet(V))^\liealg{g},
\end{equation}
via the $\liealg{g}$-equivariant van Est integration and
differentiation maps.
\begin{example}[Particular scenarios: infinitesimal case]
    \label{example:ModuleIsDerRLieAlg}%
    \
    \begin{examplelist}
    \item \label{item:VisDerRInfinitesimal} Again we can consider
        $V = \Der(\ring{R})$, viewed as Lie algebra over $\mathbb{Z}$
        as usual. As action on itself we use the adjoint action.  The
        compatibility \eqref{eq:PhigphigPhig} clearly is satisfied.
    \item \label{item:RlinearActionOnVLieAlg} The other extreme case
        is the Lie algebra $\End_{\ring{R}}(V)$ of all
        $\ring{R}$-linear endomorphisms of the module $V$ acting on
        $V$ canonically. With the trivial action on $\ring{R}$ the
        Leibniz rule \eqref{eq:LeibnizModuleDerivationOverDerivation}
        becomes $\Phi(\lambda v) = \lambda \Phi(v)$. In particular,
        all our structure maps are $\End_{\ring{R}}(V)$-equivariant.
    \item \label{item:ThingsForLieSubalgebras} Again, it will be most
        interesting to consider non-trivial Lie subalgebras
        $\liealg{g}$ in these two scenarios. Then the cohomology
        \eqref{eq:gInvariantCohomologiesGeneral} becomes specific to
        $\liealg{g}$.
    \end{examplelist}
\end{example}

\section{The Van Est Complex with Coefficients}
\label{sec:VanEstComplexCoeff}

The van Est complex as introduced in \autoref{sec:VanEstDoubleComplex}
allowed us to compute the coalgebra cohomology $\CCa^\bullet(V)$ with
trivial coefficients.  Moreover, we have been able to compute the
coalgebra cohomology $\CCa^\bullet(V,\module{M})$ with coefficient
comodule $\module{M}$ in certain cases without the help of the van Est
double complex, see
\autoref{ex:CoalgebraCohomology}~\ref{item:CoalgebraCohomology_SymmetricAlgebra}
for $\module{M} = \Sym V$ and
\autoref{ex:CoalgebraCohomology}~\ref{item:CoalgebraCohomology_Submodule}
for $\module{M} = \Sym V \tensor \Sym V^\perp$.  Nevertheless, in
order to compute the coalgebra cohomology for other coefficients,
e.g. for $\module{M} = \Sym U$, see
\autoref{ex:Comodules}~\ref{item:SymUComoduleV}, we need to generalize
the van Est double complex to incorporate this coefficient comodule.

\subsection{Perturbing the Van Est Double Complex}
\label{subsec:PerturbungVanEstDoubleComplex}

To include coefficients assume we have given a left $\Sym V$-comodule
$\module{M}$.  We define the van Est double complex with coefficients
as
\begin{equation}
    \CVanEst^{\bullet,\bullet}(V,\module{M})
    \coloneqq
    \Tensor^\bullet\Sym V
    \tensor
    \Sym V
    \tensor
    \module{M}
    \tensor
    \Anti^\bullet V,
\end{equation}
and denote as before the total grading by
\begin{equation}
    \CVanEst^\bullet(V,\module{M})
    \coloneqq
    \bigoplus_{i=0}^\infty
    \bigoplus_{k + \ell = i}
    \CVanEst^{k,\ell}(V,\module{M}).
\end{equation}
By extending all maps to $\module{M}$ as the identity we obtain the
following two homotopy retracts:
\begin{equation}
    \begin{tikzcd}
        \CCa^\bullet(V) \tensor \module{M}
        \arrow[r,"\jmap",shift left = 3pt]
        & \CVanEst^{\bullet,\bullet}(V,\module{M})
        \arrow[l,"\Qmap", shift left = 3pt]
        \arrow[r,"\Pmap", shift left = 3pt]
        &\module{M} \tensor \CCE^\bullet(V)
        \arrow[l,"\imap",shift left = 3pt]
    \end{tikzcd}
\end{equation}
In order to include the module $\module{M}$ as coefficients for the
various complexes consider
$\Bmap \colon \CVanEst^\bullet(V,\module{M}) \to
\CVanEst^{\bullet+1}(V,\module{M})$ defined by
\begin{equation}
    \Bmap(X \tensor \phi \tensor m \tensor \xi)
    \coloneqq (-1)^{k} \cdot
    X
    \tensor
    \phi
    \tensor
    m_\sweedler{0}
    \tensor
    (\pr_V(m_\sweedler{-1}) \wedge \xi),
\end{equation}
where $X \tensor \phi \tensor m \tensor \xi \in \CVanEst^{k,\ell}(V,\module{M})$,
and set
\begin{equation}
    \Double_\module{M} \coloneqq \Double + \Bmap.
\end{equation}
\begin{lemma}
    \label{lem:BSquared}%
    \
    \begin{lemmalist}
    \item It holds $\Bmap^2 = 0$.
    \item It holds $\Double_\module{M}^2 = 0$.
    \end{lemmalist}
\end{lemma}
\begin{proof}
    Since $\module{M}$ is a $\Sym V$-comodule we have
    $\pr_V(m_\sweedler{-1}) \wedge \pr_V(m_\sweedler{-2}) = 0$,
    showing the first part.  The second part is a straightforward
    computation.
\end{proof}

From now on we consider $\CVanEst^\bullet(V,\module{M})$ to be
equipped with the differential $\Double_\module{M}$ and denote its
cohomology by $\HVanEst^\bullet(V,\module{M})$.  The bimodule
structure on $\CVanEst^\bullet(V)$ as introduced in
\autoref{prop:BimoduleCVanEst} can be extended to
$\CVanEst^\bullet(V,\module{M})$ by acting trivially on $\module{M}$.
\begin{lemma}
    The complex $(\CVanEst^\bullet(V,\module{M}),\Double_\module{M})$
    is a differential graded
    $(\CCa^\bullet(V),\CCE^\bullet(V))$-bimodule.
\end{lemma}

We now want to use $\Bmap$ as a perturbation for both directions.  We
start by perturbing the coalgebra complex side.
\begin{lemma}
    For every
    $X \tensor \phi \tensor m \tensor \xi \in \CVanEst^\bullet(V,
    \module{M})$ there exists $n \in \mathbb{N}_0$ such that
    $(\Bmap \Kmap)^n(X \tensor \phi \tensor m \tensor \xi) = 0$.
\end{lemma}
\begin{proof}
    For every $m \in \module{M}$ there exists $n \in \mathbb{N}_0$
    such that
    $m_\sweedler{-1} \tensor m_\sweedler{0}
    = \sum_{i=0}^{n-1} \pr_V(m_\sweedler{-i}) \vee
    \dots \vee \pr_V(m_\sweedler{-1}) \tensor m_\sweedler{0}$, see
    \autoref{lem:exponentialproj}.  Since $\Kmap$ is the identity on
    the tensor factor $\module{M}$ it follows that
    $(\Bmap\Kmap)^n = 0$.
\end{proof}

\begin{lemma}
    \label{lem:KBj}%
    \
    \begin{lemmalist}
    \item It holds $\Kmap \Bmap \jmap = \kmap \Bmap \jmap$.
    \item It holds
        \begin{equation}
            \label{eq:kBnj}
            (\kmap \Bmap)^n \jmap (X \tensor m)
            = \frac{1}{n!}
            X
            \tensor
            ( \pr_V(m_\sweedler{-n}) \vee \dots \vee \pr_V(m_\sweedler{-1}) )
            \tensor
            m_\sweedler{0}
            \tensor
            1.
        \end{equation}
    \end{lemmalist}
\end{lemma}
\begin{proof}
    The first part follows since $\Bmap\jmap(X \tensor m)$ is of
    antisymmetric degree $1$ and $\kmap$ lowers the antisymmetric
    degree by $1$.  For the second part note that for $n = 0$ the
    equality obviously holds.  Thus assume that \eqref{eq:kBnj} holds
    for some $n \in \mathbb{N}_0$, then for
    $X \tensor m \in \CCa^k(V,\module{M})$ it holds
    \begin{align*}
        (\kmap \Bmap)^{n+1}\jmap(X \tensor m)
        &=
        \frac{1}{n!} \kmap \Bmap \bigl(
        X
        \tensor
        ( \pr_V(m_\sweedler{-n}) \vee \dots \vee \pr_V(m_\sweedler{-1}) )
        \tensor
        m_\sweedler{0}
        \tensor
        1
        \bigr)
        \\
        &=
        \frac{(-1)^{k}}{n!}
        \kmap \bigl(
        X
        \tensor
        ( \pr_V(m_\sweedler{-(n+1)}) \vee \dots \vee \pr_V(m_\sweedler{-2}) )
        \tensor
        m_\sweedler{0}
        \tensor
        \pr_V(m_\sweedler{-1})
        \bigr)
        \\
        &=
        \frac{1}{(n+1)!}
        X
        \tensor
        ( \pr_V(m_\sweedler{-(n+1)}) \vee \dots \vee \pr_V(m_\sweedler{-1}) )
        \tensor
        m_\sweedler{0}
        \tensor
        1.
    \end{align*}
\end{proof}

\begin{proposition}
    Let $\module{M}$ be a $\Sym V$-comodule.  Then
    \begin{equation}
        \begin{tikzcd}[column sep = large]
            \bigl( \CCa^\bullet(V,\module{M}), \delta_\module{M} \bigr)
            \arrow[r,"\jmap_\module{M}", shift left = 3pt]
            &\bigl( \CVanEst^{\bullet}(V,\module{M}),\Double_\module{M} \bigr)
            \arrow[l,"\Qmap_\module{M}", shift left = 3pt]
            \arrow[loop,
            out = -30,
            in = 30,
            distance = 30pt,
            start anchor = {[yshift = -7pt]east},
            end anchor = {[yshift = 7pt]east},
            "\Kmap_\module{M}"{swap}
            ]
        \end{tikzcd}
    \end{equation}
    is a special deformation retract, with
    \begin{align}
        \label{eq:FormulajM}
        \jmap_\module{M}(X \tensor m)
        &=
        X \tensor s(m_\sweedler{-1}) \tensor m_\sweedler{0} \tensor 1,
        \\
        \Qmap_\module{M}
        &=
        \Qmap \sum_{n=0}^{\infty} (-1)^n (\Bmap \Kmap)^n,
        \\
        \Kmap_\module{M}
        &=
        \Kmap \sum_{n=0}^{\infty} (-1)^n (\Bmap \Kmap)^n,
    \end{align}
    with $s$ the antipode of $\Sym V$, see \eqref{eq:AntipodeSV}. Moreover, the map $\jmap_\module{M}$ is a map of $\CCa^\bullet(V)$ left modules.
\end{proposition}
\begin{proof}
    We perturb the special deformation retract
    \begin{equation*}
        \begin{tikzcd}[column sep = large]
            \bigl( \CCa^\bullet(V) \tensor \module{M}, \delta_\Ca \bigr)
            \arrow[r,"\jmap", shift left = 3pt]
            &\bigl( \CVanEst^{\bullet}(V) \tensor \module{M},\Double \bigr)
            \arrow[l,"\Qmap", shift left = 3pt]
            \arrow[loop,
            out = -30,
            in = 30,
            distance = 30pt,
            start anchor = {[yshift = -7pt]east},
            end anchor = {[yshift = 7pt]east},
            "\Kmap"{swap}
            ]
        \end{tikzcd}
    \end{equation*}
    by the perturbation $\Bmap$.  Then by
    \autoref{cor:SpecialDeformationRetract} this yields a special
    deformation retract.  The perturbations of $\Qmap$ and $\Kmap$
    follow directly from \autoref{prop:HomologicalPerturbation}.  It
    remains to show that the perturbations of $\jmap$ and $\delta_\Ca$
    are given as claimed.  The perturbation $\jmap_\module{M}$ of
    $\jmap$ is given by
    \begin{align*}
        \jmap_\module{M}(X \tensor m)
        &=
        \sum_{n=0}^{\infty}(-1)^n
        (\Kmap \Bmap)^n
        (X \tensor \Unit \tensor m \tensor 1)
        \\
        &=
        \sum_{n=0}^\infty \frac{(-1)^n}{n!}
        X
        \tensor
        (
        \pr_V(m_\sweedler{-n}) \vee \cdots \vee \pr_V(m_\sweedler{-1})
        )
        \tensor m_\sweedler{0} \tensor 1
        \\
        &=
        X \tensor s(m_\sweedler{-1}) \tensor m_\sweedler{0} \tensor 1,
    \end{align*}
    according to \eqref{eq:SummingUpInfinitesimalStuff}.
    Note that with the exact shape of $\jmap_\module{M}$ it is easy
    to see that it is a map of $\CCa^\bullet(V)$ left modules.
    Moreover,
    the perturbation of $\delta_\Ca$ is given by
    \begin{equation*}
        \Qmap (\id + \Bmap \Kmap)^{-1} \Bmap \jmap
        =
        \qmap \sum_{n=0}^{\infty} (-1)^n (\delta \kmap)^n
        \Bmap \jmap_\module{M}.
    \end{equation*}
    Thus on $X \tensor m \in \CCa^\bullet(V) \tensor \module{M}$ we
    obtain
    \begin{equation*}
        \label{eq:PerturbationDeltaCa}
        \tag{$*$}
        \Qmap (\id + \Bmap \Kmap)^{-1} \Bmap \jmap(X \tensor m)
        =
        (-1)^{k+1}
        \qmap \sum_{n=0}^{\infty} (-1)^n (\delta \kmap)^n
        \bigl(
        X
        \tensor
        m_\sweedler{-2}
        \tensor
        m_\sweedler{0}
        \tensor
        \pr_V(m_\sweedler{-1})
        \bigr).
    \end{equation*}
    Recall that $\qmap$ projects onto anti-symmetric degree $0$ and
    $\delta \kmap$ reduces the anti-symmetric degree by $1$.  Thus
    only the $n=1$ term in \eqref{eq:PerturbationDeltaCa} does not
    vanish, and hence we are left with
    \begin{align*}
        &(-1)^{k} \cdot \qmap \delta \kmap
        \bigl(
        X
        \tensor
        s(m_\sweedler{-2})
        \tensor
        m_\sweedler{0}
        \tensor
        \pr_V(m_\sweedler{-1})
        \bigr)
        \\
        &=
        (-1)^{k}
        \qmap \delta \kmap
        \biggl(
        X
        \tensor
        \sum_{n=0}^\infty
        \frac{(-1)^n}{n!}
        ( \pr_V(m_\sweedler{-(n+1)}) \vee \dots \vee \pr_V(m_\sweedler{-2}) )
        \tensor
        m_\sweedler{0}
        \tensor
        \pr_V(m_\sweedler{-1})
        \biggr)
        \\
        &=
        - \qmap \delta
        \biggl(
        X
        \tensor
        \sum_{n=0}^\infty
        \frac{(-1)^n}{(n+1)!}
        ( \pr_V(m_\sweedler{-(n+1)}) \vee \dots \vee \pr_V(m_\sweedler{-2}) \vee \pr_V(m_\sweedler{-1}) )
        \tensor
        m_\sweedler{0}
        \tensor
        1
        \biggr)
        \\
        &=
        \qmap \delta
        \bigl(
        X
        \tensor
        s(\pr_+(m_\sweedler{-1}))
        \tensor
        1
        \tensor
        m_\sweedler{0}
        \bigr)
        \\
        &=
        \bigl(X \tensor s(\pr_+(m_\sweedler{-1}))\bigr)
        \tensor
        m_\sweedler{0}.
    \end{align*}
    Hence the perturbed differential is exactly $\delta_\module{M}$.
\end{proof}
\begin{corollary}
    \label{cor:CgrandVEcohomologyWithCoeff}%
    Let $\module{M}$ be an $\Sym V$-comodule.  Then
    $\jmap_\module{M} \colon \HCa^\bullet(V,\module{M}) \to
    \HVanEst^\bullet(V,\module{M})$ is an isomorphism of differential
    graded left $\HCa^\bullet(V)$-modules with inverse
    $\Qmap_\module{M}$.
\end{corollary}

Let us now consider the Chevalley-Eilenberg side.
\begin{lemma}
    \label{lem:BProps}%
    \
    \begin{lemmalist}
    \item \label{item:hBBh}%
        It holds $\Bmap \hmap = \hmap \Bmap$.
    \item \label{item:IdPlusPartialhInVBi}%
        It holds
        $(\id + \partial \hmap)^{-1}\Bmap \imap = \Bmap \imap$.
    \item \label{item:HBiNull}%
        It holds $\Hmap\Bmap\imap = 0$.
    \item \label{item:IdPlusBHInvBi}%
        It holds $(\id + \Bmap \Hmap)^{-1} \Bmap \imap = \Bmap \imap$.
    \end{lemmalist}
\end{lemma}
\begin{proof}
    The first part is a direct computation.  Then
    $\Bmap \imap = \Bmap \imap + \partial \hmap \Bmap \imap = (\id +
    \partial \hmap )\Bmap \imap$ and thus
    $(\id + \partial \hmap)^{-1} \Bmap \imap = \Bmap \imap$.
    Moreover,
    \begin{equation*}
        \Hmap \Bmap \imap
        =
        \hmap (\id + \partial \hmap)^{-1} \Bmap \imap
        =
        \hmap \Bmap \imap
        =
        \Bmap \hmap \imap
        =
        0.
    \end{equation*}
    This implies the fourth part as well.
\end{proof}
\begin{lemma}
    It holds $(\Bmap \Hmap)^2 = 0$.
\end{lemma}
\begin{proof}
    Recall from \autoref{lem:BSquared} that $\Bmap^2 = 0$ and $\Bmap$
    anti-commutes with $\delta$.  Thus we get
    \begin{equation*}
        (\Bmap \Hmap)^2
        =
        \Bmap \hmap \sum_{n=0}^{\infty}(-1)^n (\delta \hmap)^n \Bmap \Hmap
        =
        \hmap \sum_{n=0}^{\infty} (\delta \hmap)^n \Bmap^2 \Hmap
        =
        0.
    \end{equation*}
\end{proof}
\begin{proposition}
    Let $\module{M}$ be a $\Sym V$-comodule.  Then
    \begin{equation}
        \begin{tikzcd}
            \bigl(\CCE^\bullet(V, \module{M}), \partial_\module{M}\bigr)
            \arrow[r,"\imap", shift left = 3pt]
            &\bigl(\CVanEst^{\bullet}(V, \module{M}),\Double_\module{M}\bigr)
            \arrow[l,"\Pmap_\module{M}", shift left = 3pt]
            \arrow[loop,
            out = -30,
            in = 30,
            distance = 30pt,
            start anchor = {[yshift = -7pt]east},
            end anchor = {[yshift = 7pt]east},
            "\Hmap"{swap}
            ]
        \end{tikzcd}
    \end{equation}
    is a homotopy retract, with
    \begin{align}
        \Pmap_\module{M}
        &=
        \Pmap - \Pmap \Bmap \Hmap.
    \end{align}
Moreover, the map $\imap$ is a morphism of $\CCE(V)$ right modules.
\end{proposition}
\begin{proof}
    We perturb the homotopy retract
    \begin{equation}
        \begin{tikzcd}
            \bigl( \module{M} \tensor \CCE^\bullet(V), 0 \bigr)
            \arrow[r,"\imap", shift left = 3pt]
            &\bigl(\module{M} \tensor \CVanEst^{\bullet}(V),\Double \bigr)
            \arrow[l,"\Pmap", shift left = 3pt]
            \arrow[loop,
            out = -30,
            in = 30,
            distance = 30pt,
            start anchor = {[yshift = -7pt]east},
            end anchor = {[yshift = 7pt]east},
            "\Hmap"{swap}
            ]
        \end{tikzcd}
    \end{equation}
    by the perturbation $\Bmap$.  Then by
    \autoref{prop:HomologicalPerturbation} we obtain a homotopy
    retract.  Since $\hmap$ commutes with $\Bmap$ by
    \autoref{lem:BProps} it is easy to see that
    $\Hmap \Bmap \Hmap = 0$, and hence the homotopy $\Hmap$ stays
    unperturbed.  Moreover, since $(\Bmap \Hmap)^2 = 0$ the
    perturbation of $\Pmap$ is given by
    $\Pmap_\module{M} = \Pmap \sum_{n=0}^{1}(-1)^n(\Bmap\Hmap)^n =
    \Pmap - \Pmap \Bmap \Hmap$.  Using
    \autoref{lem:BProps}~\ref{item:HBiNull} the perturbation of
    $\imap$ is given by
    $\sum_{n=0}^{\infty} (-1)^n (\Hmap\Bmap)^n \imap = \imap$, so it
    is also clear that $\imap$ is a right module morphism.  Moreover,
    the perturbation of the zero differential on
    $\module{M} \tensor \CCE^\bullet(V)$ is given by
    \begin{equation*}
        \Pmap (\id + \Bmap \Hmap)^{-1} \Bmap \imap
        =
        \Pmap \Bmap \imap
        =
        \pmap (\id + \partial \hmap)^{-1} \Bmap \imap
        =
        \pmap \Bmap \imap.
    \end{equation*}
    On elements $m \tensor \xi \in \module{M} \tensor \CCE^\bullet(V)$
    the perturbed differential is given by
    \begin{equation*}
        \partial_\module{M}(m \tensor \xi)
        =
        m_\sweedler{0} \tensor \pr_V(m_\sweedler{-1}) \wedge \xi
    \end{equation*}
    and hence coincides with $\partial_\module{M}$ as introduced in
    \autoref{sec:ChevalleyEilenbergCohomologyDualPointView}, see
    \eqref{eq:CEDifferentialInfinitesimal}.
\end{proof}
\begin{corollary}
    \label{cor:CCEandVEcohomologyWithCoeff}%
    Let $\module{M}$ be a $\Sym V$-comodule.  Then
    $\imap \colon \HCE^\bullet(V,\module{M}) \to
    \HVanEst^\bullet(V,\module{M})$ is an isomorphism of differential
    graded right $\HCE^\bullet(V)$-modules with inverse
    $\Pmap_\module{M}$.
\end{corollary}

Together we obtain the following situation:
\begin{equation}
    \begin{tikzcd}
        \CCa^\bullet(V, \module{M})
        \arrow[r,"\jmap_\module{M}",shift left = 3pt]
        &\CVanEst^{\bullet}(V,\module{M})
        \arrow[l,"\Qmap_\module{M}", shift left = 3pt]
        \arrow[d,"\Pmap_\module{M}", shift left = 3pt] \\
        { }
        &\CCE^\bullet(V, \module{M})
        \arrow[u,"\imap",shift left = 3pt]
    \end{tikzcd}
\end{equation}
Using the above maps we can in fact find an explicit direct
quasi-isomorphism between $\CCa^\bullet(V, \module{M})$ and
$\CCE^\bullet(V, \module{M})$.  For this we define the chain maps
\begin{equation}
    \VanEstDiff_\module{M}
    \coloneqq
    \Pmap_\module{M} \jmap_\module{M}
    \colon
    \CCa^\bullet(V, \module{M})
    \to
    \CCE^\bullet(V, \module{M})
\end{equation}
and
\begin{equation}
    \VanEstInt_\module{M}
    \coloneqq
    \Qmap_\module{M} \imap
    \colon
    \CCE^\bullet(V, \module{M})
    \to
    \CCa^\bullet(V,\module{M}),
\end{equation}
as well as the map
\begin{equation}
    \VanEstHomo_\module{M}
    \coloneqq
    \Qmap_\module{M} \Hmap \jmap_\module{M}
    \colon
    \CCa^\bullet(V, \module{M})
    \to
    \CCa^{\bullet-1}(V, \module{M}).
\end{equation}
\begin{theorem}[Van Est deformation retract with coefficients]
    \label{thm:VanEstDeformationRetractCoeff}%
    Let $V$ be a $\ring{R}$-module and let $\module{M}$ be a
    $\Sym V$-comodule.  Then
    \begin{equation}
        \label{diag:VanEstMapsDeformationRetractCoeff}
        \begin{tikzcd}[column sep = large]
            \bigl( \CCE^\bullet(V, \module{M}), \partial_\module{M} \bigr)
            \arrow[r,"\VanEstInt_\module{M}", shift left = 3pt]
            &\bigl( \CCa^{\bullet}(V,\module{M}),\delta_\module{M} \bigr)
            \arrow[l,"\VanEstDiff_\module{M}", shift left = 3pt]
            \arrow[loop,
            out = -30,
            in = 30,
            distance = 30pt,
            start anchor = {[yshift = -7pt]east},
            end anchor = {[yshift = 7pt]east},
            "\VanEstHomo_\module{M}"{swap}
            ]
        \end{tikzcd}
    \end{equation}
    is a deformation retract, i.e. we the following hold:
    \begin{theoremlist}
    \item We have
        $\VanEstDiff_\module{M} \circ \VanEstInt_\module{M} = \id$.
    \item We have
        $\delta_\module{M} \VanEstHomo_\module{M} +
        \VanEstHomo_\module{M} \delta_\module{M} = \id -
        \VanEstInt_\module{M} \circ \VanEstDiff_\module{M}$.
    \end{theoremlist}
\end{theorem}
\begin{proof}
    The proof is analogous to that of
    \autoref{thm:VanEstDeformationRetract}.  Consider the composition
    \begin{equation*}
        (\Pmap_\module{M},\imap,0)
        \circ (\jmap_\module{M},\Qmap_\module{M},\Kmap_\module{M})
        =
        (\Pmap_\module{M} \jmap_\module{M},
        \Qmap_\module{M} \imap,
        \Pmap_\module{M} \Kmap_\module{M} \imap)
        =
        (\VanEstDiff_\module{M},
        \VanEstInt_\module{M},
        \Pmap_\module{M} \Kmap_\module{M} \imap).
    \end{equation*}
    By \autoref{lem:hkpk} we know that $\Hmap \Kmap = 0 = \Pmap \Kmap$
    and thus
    \begin{equation*}
        \Pmap_\module{M} \Kmap_\module{M}
        =
        \Pmap
        \sum_{i=0}^{\infty}(-1)^i
        (\Bmap \Hmap)^i
        \Kmap
        \sum_{j=0}^{\infty} (-1)^j
        (\Bmap \Kmap)^j
        =
        0.
    \end{equation*}
    This shows the first part.  For the second part we consider the
    composition
    \begin{equation*}
        (\Qmap_\module{M},\jmap_\module{M},0) \circ
        (\imap,\Pmap_\module{M}, \Hmap_\module{M})
        =
        (\Qmap_\module{M} \imap,
        \Pmap_\module{M} \jmap_\module{M},
        \Qmap_\module{M} \Hmap \jmap)
        =
        (\VanEstInt_\module{M},
        \VanEstDiff_\module{M},
        \VanEstHomo_\module{M}).
    \end{equation*}
    Hence we have
    \begin{equation*}
        \delta_\module{M} \VanEstHomo_\module{M}
        + \VanEstHomo_\module{M} \delta_\module{M}
        =
        \id - \VanEstInt_\module{M} \circ \VanEstDiff_\module{M}.
    \end{equation*}
\end{proof}
\begin{remark}
        As already in the case of the deformation retract
        without coefficients, see \autoref{rem:nontrivialLiecoalg},
        we are very certain that one
        can use the deformation retract
        \eqref{diag:VanEstMapsDeformationRetractCoeff} and a suitable
        perturbation of the differential $\delta_\module{M}$ in order to
        obtain a deformation retract connecting the Chevalley-Eilenberg
        complex of a locally finite Lie coalgebra with values in a Lie comodule and the coalgebra complex of
        its universal enveloping coalgebra complex with values in the same
        comodule. This will be part of a future work.
\end{remark}
\begin{corollary}
    Let $V$ be a $\ring{R}$-module and let $\module{M}$ be a
    $\Sym V$-comodule. Then
    $\VanEstDiff_\module{M} \colon \HCa^\bullet(V,\module{M}) \to
    \HCE^\bullet(V,\module{M})$ is an isomorphism with inverse
    $\VanEstInt_\module{M}$.
\end{corollary}
\begin{proposition}
    \label{prop:ExplicitFormVE_M}%
    For factorizing
    $X \tensor m = X_1 \tensor \cdots \tensor X_k \tensor m \in
    \CCa^k(V, \module{M})$ it holds
    \begin{equation}
        \begin{split}
            \VanEstDiff_\module{M}(X \tensor m)
            &=
            m \tensor \VanEstDiff(X)
            \\
            &\quad+
            \counit(X_1) \cdot
            m_\sweedler{0}
            \tensor
            (
            \pr_V(m_\sweedler{-1})
            \wedge
            \VanEstDiff(X_2 \tensor \cdots \tensor X_{k})
            ).
        \end{split}
    \end{equation}
\end{proposition}
\begin{proof}
    For $X \tensor m \in \CCa^k(V,\module{M})$ we have
    \begin{align*}
        \VanEstDiff_\module{M}(X \tensor m)
        &=
        \Pmap_\module{M} \jmap_\module{M} (X \tensor m)
        \\
        &=
        \Pmap_\module{M}
        (X \tensor s(m_\sweedler{-1}) \tensor m_\sweedler{0} \tensor 1)
        \\
        &=
        \Pmap
        (X \tensor s(m_\sweedler{-1}) \tensor m_\sweedler{0} \tensor 1)
        -
        \Pmap\Bmap\Hmap
        (X \tensor s(m_\sweedler{-1}) \tensor m_\sweedler{0} \tensor 1).
    \end{align*}
    We compute the two summands separately: For the first term we get
    \begin{align*}
        \Pmap(X \tensor s(m_\sweedler{-1}) \tensor m_\sweedler{0} \tensor 1)
        &=
        (-1)^k \pmap (\partial \hmap)^k
        (X \tensor s(m_\sweedler{-1}) \tensor m_\sweedler{0} \tensor 1)
        \\
        &=
        \varepsilon(s(m_\sweedler{-1})) (-1)^k \cdot
        \pmap (\partial \hmap)^k
        (X \tensor \Unit \tensor m_\sweedler{0} \tensor 1)
        \\
        &=
        (-1)^k \pmap (\partial \hmap)^k
        (X \tensor \Unit \tensor m \tensor 1)
        \\
        &=
        p(1 \tensor (X_1)_\sweedler{0} \tensor m \tensor
        \pr_V((X_1)_\sweedler{1})
        \wedge
        \pr_V(X_2)
        \wedge \cdots \wedge
        \pr_V(X_k))
        \\
        &=
        m \tensor
        \pr_V(X_1)
        \wedge \cdots \wedge
        \pr_V(X_k)
        \\
        &=
        m \tensor \VanEstDiff(X),
    \end{align*}
    where we used \autoref{lem:partialh}.  For the second part we get
    \begin{align*}
        \Pmap \Bmap \Hmap(
        X
        \tensor
        s(m_\sweedler{-1})
        \tensor
        m_\sweedler{0}
        \tensor
        1)
        &=
        \pmap
        \sum_{n=0}^{\infty} (-1)^n
        (\partial \hmap)^n
        \Bmap \hmap
        \sum_{j=0}^{\infty} (-1)^j
        (\partial \hmap)^j
        (X \tensor s(m_\sweedler{-1}) \tensor m_\sweedler{0} \tensor 1)
        \\
        &=
        \pmap \Bmap \hmap
        \sum_{j=0}^{\infty} (-1)^j
        (\partial \hmap)^j
        (X \tensor s(m_\sweedler{-1}) \tensor m_\sweedler{0} \tensor 1)
        \\
        &=
        \pmap \Bmap \hmap (-1)^{k-1}
        (\partial \hmap)^{k-1}
        (X \tensor s(m_\sweedler{-1}) \tensor m_\sweedler{0} \tensor 1)
        \\
        &=
        \pmap \Bmap \hmap (
        X_1
        \tensor
        (X_2)_\sweedler{0}
        \tensor
        m
        \tensor
        \pr_V((X_2)_\sweedler{1})
        \wedge
        \pr_V(X_3)
        \wedge \cdots \wedge
        \pr_V(X_k)
        )
        \\
        &=
        - \pmap \Bmap (
        1
        \tensor
        X_1
        \tensor
        m
        \tensor
        \pr_V(X_2)
        \wedge \cdots \wedge
        \pr_V(X_k)
        )
        \\
        &=
        - \pmap (
        1
        \tensor
        X_1
        \tensor
        m_\sweedler{0}
        \tensor
        \pr_V(m_\sweedler{-1})
        \wedge
        \VanEstDiff(X_2 \wedge \cdots \wedge X_k)
        )
        \\
        &=
        - \counit(X_1) \cdot m_\sweedler{0}
        \tensor
        \pr_V(m_\sweedler{-1})
        \wedge
        \VanEstDiff(X_2 \wedge \dots \wedge X_k).
    \end{align*}
\end{proof}

\begin{remark}
Using \autoref{prop:ExplicitFormVE_M} and the deformation retract
without coefficients, see \autoref{thm:VanEstDeformationRetract},
one can show that
$\VanEstDiff_\module{M}$ is actually a morphism of left modules along
$\VanEstDiff\colon \HCa(V) \to \HCE(V)$,
where the left module
structure of $\HCE(V)$ on $\HCE(V,M)$ is given by the right module structure
adding the correct signs.
\end{remark}

\begin{example}[Computing coalgebra cohomology]
    \label{ex:VanEstWithCoeff}%
    \
    \begin{examplelist}
    \item In the case $\module{M} = \ring{R}$ equipped with its
        canonical trivial coaction we have
        \begin{equation}
            \CCa^\bullet(V,\ring{R}) = \CCa^\bullet(V),
            \quad
            \CVanEst^{\bullet}(V,\ring{R}) = \CVanEst^{\bullet}(V),
            \quad
            \CCE^\bullet(V,\ring{R}) = \CCE^\bullet(V).
        \end{equation}
        Similarly, all maps reduce to the case without coefficients as
        considered in \autoref{sec:VanEstDoubleComplex} and
        \autoref{sec:VanEstMaps}.
    \item \label{item:CoalgebraCohomologySubmoduleUV} Consider a
        complemented submodule $U \subseteq V$ such that
        $V = U \oplus U^\perp$, with the canonical $\Sym V$ coaction
        on $\Sym U$ and the trivial coaction on $\Sym U^\perp$ as in
        \autoref{ex:Comodules}~\ref{item:SymUComoduleV}.  By composing
        the van Est deformation retract
        \eqref{diag:VanEstMapsDeformationRetractCoeff} with the
        homotopy retract \eqref{diag:CCESUhomotopRetract} computing
        $\CCE^\bullet(V,\Sym U)$ we obtain a homotopy retract
        \begin{equation}
            \begin{tikzcd}[column sep = large]
                \Anti^\bullet U^\perp
                \arrow[rr,"\VanEstInt_{\Sym U} \circ \iota \tensor \id", shift left = 3pt]
                &
                &\CCa^\bullet(V,\Sym U)
                \arrow[ll,"\pi \tensor \id \circ \VanEstDiff_{\Sym U}", shift left = 3pt]
                \arrow[loop,
                out = -30,
                in = 30,
                distance = 30pt,
                start anchor = {[yshift = -7pt]east},
                end anchor = {[yshift = 7pt]east}
                ]
            \end{tikzcd}
        \end{equation}
        with homotopy given by
        $\VanEstHomo_{\Sym U} + \VanEstInt_{\Sym U} \circ
        \partial^{-1}_{\Sym U} \tensor \id \circ \VanEstDiff_{\Sym
          U}$.  Thus, in particular, we get
        \begin{equation}
            \HCa^\bullet(V,\Sym U) \simeq \Anti^\bullet U^\perp.
        \end{equation}
    \end{examplelist}
\end{example}

\subsection{Coalgebra Cohomology for the Reduced Symmetric Coalgebra}
\label{sec:ReducedSymCoalgebra}

We know that $\redSym^\bullet(V) = \bigoplus_{i=1}^\infty \Sym^i V$
is a sub-coalgebra of $\Sym^\bullet V$ without counit with respect to the reduced shuffle coproduct
$\redshcoprod$.
Moreover, every $\redSym^\bullet V$ comodule $\module{M}$
is a $\Sym V$-comodule in a canonical way.
Now let us consider
\begin{equation}
    \redCCa^\bullet(V,\module{M}) \coloneqq \Tensor^\bullet \redSym V \tensor \module{M}.
\end{equation}
By the definition of $\delta_\Ca$ and $\delta_\module{M}$, see \eqref{eq:dDefTheRealThing} and \eqref{eq:DifferentialDefOnMValued}, this is a subcomplex of
$\CCa^\bullet(V,\module{M})$ and we denote the inclusion by $\iota \colon \redCCa^\bullet(V,\module{M}) \to \CCa^\bullet(V,\module{M})$.
We will call $\redCCa^\bullet(V,\module{M})$ the \emph{reduced coalgebra complex} of $V$ with values in $\module{M}$.

To compute the reduced coalgebra cohomology one can now repeat our construction of the van Est deformation retract by introducing the \emph{reduced van Est double complex}
\begin{equation}
    \redCVanEst^{\bullet,\bullet}(V,\module{M})
    \coloneqq \Tensor^\bullet\redSym V \tensor \Sym V \tensor \module{M} \tensor \Anti V.
\end{equation}
By the concrete formulas \eqref{eq:VerticalVEDifferential} and \eqref{eq:HorizontalVEDifferential} we see that $\redCVanEst^{\bullet,\bullet}$
is in fact a double complex.
Moreover, the homotopies $\hmap$, $\kmap$ as well as the other maps relating $\redCVanEst^{\bullet,\bullet}(V,\module{M})$ with $\redCCa^\bullet(V,\module{M})$
and $\CCE^\bullet(V,\module{M})$ stay well-defined and still satisfy all required algebraic relations.
Additionally, the homotopy retract relating the reduced coalgebra complex to the reduced van Est double complex
becomes a special deformation retract.
Thus by repeating the perturbation arguments we obtain a special deformation retract
\begin{equation}
\begin{tikzcd}
    \CCE^\bullet(V, \module{M})
        \arrow[r,"", shift left = 3pt]
    & \redCCa^\bullet(V,\module{M})
        \arrow[l,"", shift left = 3pt]
        \arrow[loop,
        out = -30,
        in = 30,
        distance = 30pt,
        start anchor = {[yshift = -7pt]east},
        end anchor = {[yshift = 7pt]east},
        ""{swap}
        ]
\end{tikzcd}.
\end{equation}
This shows in particular $\redHCa^\bullet(V,\module{M})
\simeq \Anti^\bullet(V,\module{M})
\simeq \HCa^\bullet(V,\module{M})$,
and therefore the the reduced coalgebra cohomology agrees with the usual coalgebra cohomology.

\subsection{Functoriality with Coefficients}
\label{sec:FunctorialityWithCoeff}

Let $\ring{S}$ be another ring of scalars with a module $W$ over it.
Let $\Phi \colon V \to W$ be a module morphism along a unital ring
morphism $\vartheta \colon \ring{R} \to \ring{S}$.  Moreover, consider
a comodule morphism $\Xi \colon \module{M} \to \module{N}$ along
$\vartheta$ between an $\Sym V$-comodule $\module{M}$ and an
$\Sym W$-comodule $\module{N}$ along $\Phi \colon \Sym V \to \Sym W$,
i.e. we have
\begin{equation}
    \label{eq:ComoduleMorphismAlong}
    \Delta_{\module{N}} \circ \Xi
    =
    (\Phi \tensor \Xi) \circ \Delta_{\module{M}}
    \quad
    \textrm{and}
    \quad
    \Xi(\lambda m)
    =
    \vartheta(\lambda) \Xi(m)
\end{equation}
for all $\lambda \in \ring{R}$ and $m \in \module{M}$.  Note that the
first condition is consistent even though we use tensor products over
different rings thanks to the fact that both maps $\Phi$ and $\Xi$ are
along the same $\vartheta$. Here we extend $\Phi \colon V \to W$ as an
algebra homomorphism to the symmetric algebras along $\vartheta$ as in
\autoref{sec:Functoriality}.

Taking the tensor product with $\module{M}$ and $\module{N}$ of the
deformation retracts in \autoref{prop:CEDeformationRetract} and
\autoref{prop:coalgebraDeformationRetract} yields morphisms
\begin{equation}
    \label{eq:PhiTensorXiRetractCCEvE}
    \begin{tikzcd}[column sep = large, row sep = large]
        \bigl( \CCE^\bullet(V) \tensor \module{M}, 0 \bigr)
        \arrow[r,"\imap^V", shift left = 3pt]
        \arrow[d,"\Phi \tensor \Xi"{swap}]
        &\bigl( \CVanEst^{\bullet}(V) \tensor \module{M},\Double^V \bigr)
        \arrow[l,"\Pmap^V", shift left = 3pt]
        \arrow[loop,
        out = -30,
        in = 30,
        distance = 30pt,
        start anchor = {[yshift = -7pt]east},
        end anchor = {[yshift = 7pt]east},
        "\Hmap^V"{swap}
        ]
        \arrow[d,"\Phi \tensor \Xi"]
        \\
        \bigl( \CCE^\bullet(W) \tensor \module{N}, 0 \bigr)
        \arrow[r,"\imap^W", shift left = 3pt]
        &\bigl( \CVanEst^{\bullet}(W) \tensor \module{N},\Double^W \bigr)
        \arrow[l,"\Pmap^W", shift left = 3pt]
        \arrow[loop,
        out = -30,
        in = 30,
        distance = 30pt,
        start anchor = {[yshift = -7pt]east},
        end anchor = {[yshift = 7pt]east},
        "\Hmap^W"{swap}
        ]
    \end{tikzcd}
\end{equation}
and
\begin{equation}
    \label{eq:PhiTensorXiRetractCcavE}
    \begin{tikzcd}[column sep = large, row sep = large]
        \bigl( \CCa^\bullet(V) \tensor \module{M}, \delta_\Ca \bigr)
        \arrow[r,"\jmap^V", shift left = 3pt]
        \arrow[d,"\Phi \tensor \Xi"{swap}]
        &\bigl( \CVanEst^{\bullet}(V) \tensor \module{M},\Double^V \bigr)
        \arrow[l,"\Qmap^V", shift left = 3pt]
        \arrow[loop,
        out = -30,
        in = 30,
        distance = 30pt,
        start anchor = {[yshift = -7pt]east},
        end anchor = {[yshift = 7pt]east},
        "\Kmap^V"{swap}
        ]
        \arrow[d,"\Phi \tensor \Xi"]
        \\
        \bigl(\CCa^\bullet(W) \tensor \module{N}, \delta_\Ca \bigr)
        \arrow[r,"\jmap^W", shift left = 3pt]
        &\bigl( \CVanEst^{\bullet}(W) \tensor \module{N},\Double^W \bigr)
        \arrow[l,"\Qmap^W", shift left = 3pt]
        \arrow[loop,
        out = -30,
        in = 30,
        distance = 30pt,
        start anchor = {[yshift = -7pt]east},
        end anchor = {[yshift = 7pt]east},
        "\Kmap^W"{swap}
        ]
    \end{tikzcd}
\end{equation}
of homotopy retracts. As before, the map $\Phi \tensor \Xi$ is
well-defined since both are morphisms along the same $\vartheta$.
With this observation the verification of
\eqref{eq:PhiTensorXiRetractCCEvE} and
\eqref{eq:PhiTensorXiRetractCcavE} is clear since all involved maps
commute, as one can check on factorizing tensors using additivity
alone.
Composing these homotopy retracts yields the following compatibility
of $\Xi$ with the van Est maps:
\begin{proposition}
    \label{prop:VanEstfunctorialCoeffs}%
    Let $\Phi\colon V \to W$ be a morphism along
    $\vartheta\colon \ring{R} \to \ring{S}$ and let
    $\Xi \colon \module{M} \to \module{N}$ be a comodule morphism
    along $\Phi$.  Then
    \begin{equation}
        \begin{tikzcd}[column sep = large, row sep = large]
            \bigl( \CCE^\bullet(V,\module{M}), \partial_\module{M} \bigr)
            \arrow[r,"\VanEstInt_\module{M}", shift left = 3pt]
            \arrow[d,"\Phi \tensor \Xi"{swap}]
            &\bigl( \CCa^{\bullet}(V,\module{M}),\delta_\module{M}^V \bigr)
            \arrow[l,"\VanEstDiff_\module{M}", shift left = 3pt]
            \arrow[loop,
            out = -30,
            in = 30,
            distance = 30pt,
            start anchor = {[yshift = -7pt]east},
            end anchor = {[yshift = 7pt]east},"\VanEstHomo_\module{M}^V"{swap}
            ]
            \arrow[d,"\Phi \tensor \Xi"]
            \\
            \bigl( \CCE^\bullet(W,\module{N}), \partial_\module{N} \bigr)
            \arrow[r,"\VanEstInt_\module{N}", shift left = 3pt]
            &\bigl( \CCa^{\bullet}(W,\module{N}),\delta_\module{N}^W \bigr)
            \arrow[l,"\VanEstDiff_\module{N}", shift left = 3pt]
            \arrow[loop,
            out = -30,
            in = 30,
            distance = 30pt,
            start anchor = {[yshift = -7pt]east},
            end anchor = {[yshift = 7pt]east},"\VanEstHomo_\module{N}^W"{swap}
            ]
        \end{tikzcd}
    \end{equation}
    is a morphism of homotopy retracts.
\end{proposition}
\begin{proof}
    Since $\Phi \tensor \Xi$ commutes with all involved structure maps
    it also commutes with the homotopy and the van Est maps. Again,
    being additive is sufficient to show this.
\end{proof}

The important situation we are interested in is again the case of a
group action by some group $\group{G}$. As before, we need an action
$\phi\colon \group{G} \ni g \mapsto \phi_g \in \Aut(\ring{R})$ on
$\ring{R}$ by ring automorphisms and a corresponding action
$\Phi\colon \group{G} \ni g \mapsto \Phi_g \in \Aut_{\mathbb{Z}}(V)$
of the same group on the module by additive but not necessarily
$\ring{R}$-linear automorphisms of $V$ satisfying
\eqref{eq:PhigphigPhig} as before. In addition, we want a group action
$\Phi^{\module{M}}\colon G \ni g \mapsto \Phi^{\module{M}}_g
\in \End_{\mathbb{Z}}(\module{M})$ on the comodule $\module{M}$ by
comodule morphisms $\Phi^{\module{M}}_g$ along $\Phi_g$,
i.e. satisfying
\begin{equation}
    \label{eq:GactsComoduleM}
    \Delta_{\module{M}} \circ \Phi^{\module{M}}_g
    =
    (\Phi_g \tensor \Phi^{\module{M}}_g) \circ \Delta_{\module{M}}
    \quad
    \textrm{and}
    \quad
    \Phi^{\module{M}}_g (\lambda m) = \phi_g(\lambda) \Phi^{\module{M}}_g(m)
\end{equation}
for all $g \in \group{G}$ and $m \in \module{M}$. All the structure
maps including the van Est maps and the corresponding homotopies will
then be $\group{G}$-equivariant. Hence we still get a deformation
retract
\begin{equation}
    \label{diag:GinvVanEstCoeff}
    \begin{tikzcd}[column sep = large]
        \bigl( \CCE^\bullet(V, \module{M})^\group{G}, \partial_{\module{M}} \bigr)
        \arrow[r,"\VanEstInt_\module{M}", shift left = 3pt]
        &\bigl( \CCa^{\bullet}(V, \module{M})^\group{G},\delta_\Ca \bigr)
        \arrow[l,"\VanEstDiff_\module{M}", shift left = 3pt]
        \arrow[loop,
        out = -30,
        in = 30,
        distance = 30pt,
        start anchor = {[yshift = -7pt]east},
        end anchor = {[yshift = 7pt]east},
        "\VanEstHomo_\module{M}"{swap}
        ]
    \end{tikzcd}
    .
\end{equation}
Again, the invariant cohomologies are isomorphic via the
$\group{G}$-equivariant van Est integration and differentiation maps
since the homotopy $\VanEstHomo$ is $\group{G}$-equivariant, too.

\subsection{Infinitesimal Functoriality with Coefficients}
\label{subsec:InfinitesimalFunctorialityCoeff}

As for the functoriality we also are interested in the infinitesimal
counterpart for the case with coefficients. Hence let again
$\vartheta \in \Der(\ring{R})$ be a derivation of $\ring{R}$ with a
module derivation $\Phi \in \End_{\mathbb{Z}}(V)$ of $V$ along $d$ as
in \eqref{eq:LeibnizModuleDerivationOverDerivation}. In addition we
need now a module derivation
$\Phi_{\module{M}}\colon \module{M} \to \module{M}$ along $\vartheta$
which satisfies
\begin{equation}
    \label{eq:XiInfinitesimalCoMorph}
    \Delta \circ \Phi_{\module{M}}
    =
    \big(
    \Phi \tensor \mathord{\id}_{\module{M}}
    +
    \mathord{\id}_{\Sym V} \tensor \Phi_{\module{M}}
    \big)
    \circ \Delta.
\end{equation}
Again, we note that though neither $\Phi$ nor $\Phi_{\module{M}}$ are
$\ring{R}$-linear, this condition makes sense since we have the same
Leibniz rule for $\Phi$ and $\Phi_{\module{M}}$ along the same
derivation $\vartheta$. Here we have extended $\Phi$ to $\Sym V$ as in
\eqref{eq:ExtendDtoTensorPowersV}.

We extend $\Phi_{\module{M}}$ now to tensor products like
$\Sym V \tensor \module{M}$ by means of the Leibniz rule
\begin{equation}
    \label{eq:DMonTensorProducts}
    \Phi_{\module{M}}(v_1 \vee \cdots \vee v_n \tensor m)
    =
    \Phi(v_1 \vee \cdots \vee v_n) \tensor m
    +
    v_1 \vee \cdots \vee v_n \tensor \Phi_{\module{M}}(m),
\end{equation}
which yields again a well-defined map on the tensor product thanks to
the assumption that $\Phi$ and $\Phi_{\module{M}}$ are module
derivations along the same $\vartheta$. For other combinations of
tensor powers of $V$ with $\module{M}$ we use the analogous Leibniz
rules.

It is then a simple verification on factorizing tensors that this
$\Phi_{\module{M}}$ commutes with all other structure maps of our
complexes. As consequence we arrive at the analog of
\autoref{proposition:VanEstInfinitesimalFunctorial}:
\begin{proposition}
    \label{prop:VanEstfunctorialCoeffsInfinitesimal}%
    Let $\vartheta \in \Der(\ring{R})$ and let $\Phi\colon V \to V$ be
    a module derivation along $\vartheta$. Moreover, let
    $\Phi_{\module{M}}\colon \module{M} \to \module{M}$ be a module
    derivation along $\vartheta$ such that
    \eqref{eq:ComoduleMorphismAlong} is satisfied.  Then
    \begin{equation}
        \begin{tikzcd}[column sep = large, row sep = large]
            \bigl( \CCE^\bullet(V,\module{M}), \partial_\module{M} \bigr)
            \arrow[r,"\VanEstInt_\module{M}", shift left = 3pt]
            \arrow[d,"\Phi_{\module{M}}"{swap}]
            &\bigl( \CCa^{\bullet}(V,\module{M}),\delta_\module{M}^V \bigr)
            \arrow[l,"\VanEstDiff_\module{M}", shift left = 3pt]
            \arrow[loop,
            out = -30,
            in = 30,
            distance = 30pt,
            start anchor = {[yshift = -7pt]east},
            end anchor = {[yshift = 7pt]east},"\VanEstHomo_\module{M}^V"{swap}
            ]
            \arrow[d,"\Phi_{\module{M}}"]
            \\
            \bigl( \CCE^\bullet(W,\module{M}), \partial_\module{M} \bigr)
            \arrow[r,"\VanEstInt_\module{M}", shift left = 3pt]
            &\bigl( \CCa^{\bullet}(W,\module{M}),\delta_\module{M}^W \bigr)
            \arrow[l,"\VanEstDiff_\module{M}", shift left = 3pt]
            \arrow[loop,
            out = -30,
            in = 30,
            distance = 30pt,
            start anchor = {[yshift = -7pt]east},
            end anchor = {[yshift = 7pt]east},"\VanEstHomo_\module{M}^W"{swap}
            ]
        \end{tikzcd}
    \end{equation}
    is a morphism of homotopy retracts along $\vartheta$.
\end{proposition}
\begin{proof}
    Since $\Phi_{\module{M}}$ commutes with all structure maps of the
    complexes, it also commutes with the homotopy and the van Est
    maps.
\end{proof}

The analog of the group action is now again a Lie algebra action of
some Lie algebra $\liealg{g}$ over $\mathbb{Z}$. As in
\autoref{subsec:InfinitesimalFunctoriality} we require a Lie
algebra action $\phi\colon \liealg{g} \to \Der(\ring{R})$ together
with Lie algebra actions $\Phi\colon \liealg{g}
\to \End_{\mathbb{Z}}(V)$ and $\Phi^{\module{M}}\colon \liealg{g}
\to \End_{\mathbb{Z}}(\module{M})$ satisfying the Leibniz rule
\eqref{eq:LeibnizForLieAlgebraAction} and
\begin{equation}
    \label{eq:LeibnizRulesgActsOnVandOnM}
    \Phi^{\module{M}}_\xi(\lambda m)
    =
    \phi_\xi(\lambda) m + \lambda \Phi^{\module{M}}_\xi(m)
\end{equation}
for $\xi \in \liealg{g}$ and $m \in \module{M}$. In addition we need
the compatibility
\begin{equation}
    \label{eq:LieAlgCompatibleComoduleStructure}
    \Delta \circ \Phi^{\module{M}}_\xi
    =
    (\Phi_\xi \tensor \mathord{\id}
    +
    \mathord{\id} \tensor \Phi^{\module{M}}_\xi)
    \circ
    \Delta
\end{equation}
with the $\Sym V$-comodule structure of $\module{M}$, being the
infinitesimal counterpart of the compatibility
\eqref{eq:ComoduleMorphismAlong}. Then all the structure maps
including the van Est maps and the corresponding homotopies will be
$\liealg{g}$-equivariant. Hence we still get a deformation retract
\begin{equation}
    \label{diag:ginvVanEstCoeff}
    \begin{tikzcd}[column sep = large]
        \bigl( \CCE^\bullet(V, \module{M})^\liealg{g}, \partial_{\module{M}} \bigr)
        \arrow[r,"\VanEstInt_\module{M}", shift left = 3pt]
        &\bigl( \CCa^{\bullet}(V, \module{M})^\liealg{g},\delta_\Ca \bigr)
        \arrow[l,"\VanEstDiff_\module{M}", shift left = 3pt]
        \arrow[loop,
        out = -30,
        in = 30,
        distance = 30pt,
        start anchor = {[yshift = -7pt]east},
        end anchor = {[yshift = 7pt]east},
        "\VanEstHomo_\module{M}"{swap}
        ]
    \end{tikzcd}
    .
\end{equation}
Again, the invariant cohomologies are isomorphic via the
$\liealg{g}$-equivariant van Est integration and differentiation maps
since the homotopy $\VanEstHomo$ is $\liealg{g}$-equivariant, too.

\part{Hochschild Cohomology in Differential Geometry}
\label{part:HochschildCohomology}

In \autoref{part:CoalgebraCohomology} we discussed the relation of
coalgebra cohomology $\CCa^\bullet(V)$ to Chevalley-Eilenberg
cohomology $\CCE^\bullet(V)$ for a given $\ring{R}$-module $V$.  We
constructed the van Est deformation retract connecting
$\CCa^\bullet(V)$ and $\CCE^\bullet(V)$, which in particular yielded a
quasi-isomorphism $\VanEstDiff$.  In this second part we want to use
the van Est deformation retract to compute differential Hochschild
cohomologies of various geometric situations.  In fact, in all these
cases we will actually obtain again a deformation retract.  This is
achieved by using an adapted global symbol calculus to identify the
respective Hochschild complex with a specific coalgebra complex.

In particular, by choosing $\ring{R} = \Cinfty(M)$ and
$V = \Secinfty(TM)$, with $M$ a smooth manifold, we will obtain the
classical Hochschild-Kostant-Rosenberg Theorem for smooth manifolds
and differentiable cochains, see e.g. \cite{dewilde.lecomte:1995a,
  vey:1975a, gutt:1980b}.

\section{Differential Hochschild Complex}
\label{sec:HochschildComplex}

For the algebraic background we recall the well-known definition: let
$\algebra{A}$ be an associative algebra over an auxiliary ring
$\field{k}$ of scalars, typically with the extra assumption that
$\field{Q} \subseteq \field{k}$. Moreover, let $\module{M}$ be an
$(\algebra{A}, \algebra{A})$-bimodule. We always assume that modules
over $\algebra{A}$ have an underlying structure of a
$\field{k}$-module such that the $\algebra{A}$-module operations are
$\field{k}$-bilinear. Then the algebraic Hochschild complex of
$\algebra{A}$ with values in the bimodule $\module{M}$ is given by
\begin{equation}
    \label{eq:HCAlgebraically}
    \HC^\bullet(\algebra{A}, \module{M})
    \coloneqq
    \bigoplus_{k=0}^\infty \Hom_{\field{k}}
    \big(
    \algebra{A}^{\tensor k}, \module{M}
    \big),
\end{equation}
where all tensor products are taken over $\field{k}$. By convention,
$\algebra{A}^{\tensor 0} = \field{k}$ and hence
$\HC^0(\algebra{A}, \module{M}) = \module{M}$.  We will identify
linear maps in the tensor products with the corresponding multilinear
maps on the Cartesian products in the following.  The Hochschild
differential
$\delta\colon \HC^\bullet(\algebra{A}, \module{M}) \to
\HC^{\bullet+1}(\algebra{A}, \module{M})$ is then given by
\begin{equation}
    \label{eq:HochschildDifferentialAlgebraically}
    \begin{split}
        &(\delta \phi)(a_1, \ldots, a_{k+1})
        =
        a_1 \cdot \phi(a_2, \ldots, a_{k+1})
        \\
        &\quad
        +
        \sum_{i=1}^k
        (-1)^i \phi(a_1, \ldots, a_ia_{i+1}, \ldots, a_{k+1})
        +
        (-1)^{k+1} \phi(a_1, \ldots, a_k) \cdot a_{k+1},
    \end{split}
\end{equation}
where $a_1, \ldots, a_{k+1} \in \algebra{A}$ and we use the left and
right module structure on the values in $\module{M}$ in the first and
last term.

In the case $\module{M} = \algebra{A}$ the Hochschild complex
$\HC^\bullet(\algebra{A}) \coloneqq
\HC^\bullet(\algebra{A},\algebra{A})$ can be equipped with the
so-called \emph{cup product}
\begin{equation}
    \phi \cup \psi \coloneqq \mu \circ (\phi \tensor \psi),
\end{equation}
for $\phi,\psi \in \HC^\bullet(\algebra{A})$ and $\mu$ denoting the
multiplication of $\algebra{A}$.  With this product
$\HC^\bullet(\algebra{A})$ becomes a differential graded algebra,
i.e. it holds
\begin{equation}
    \delta(\phi \cup \psi)
    =
    \delta \phi \cup \psi
    +
    (-1)^n \phi \cup \delta \psi
\end{equation}
for $\phi \in \HC^n(\algebra{A})$ and
$\psi \in \HC^\bullet(\algebra{A})$.  Note that for an
$\algebra{A}$-bimodule $\module{M}$ there is also a pairing
\begin{equation}
    \cup_\module{M}
    \colon
    \HC^k (\algebra{A}) \tensor \HC^\ell (\algebra{A},\module{M})
    \to
    \HC^{k+\ell} (\algebra{A},\module{M})
\end{equation}
defined by
\begin{equation}
    (\phi\cup_\module{M} \psi)(a_1,\dots,a_{k+l})
    =
    \phi(a_1,\dots,a_k) \cdot \psi(a_{k+1},\dots, a_{k+\ell})
\end{equation}
which turns $\HC^\bullet(\algebra{A},\module{M})$ into a differential
graded $\HC^\bullet(\algebra{A})$-left module, i.e. for
$\phi\in \HC^k(\algebra{A})$, $\xi\in \HC^\ell$ and
$\psi\in \HC^m(\algebra{A},\module{M})$, we have
\begin{equation}
    \delta(\phi\cup_\module{M}\Psi)
    =
    (\delta\phi)\cup_\module{M}\Psi +(-1)^{k}\phi\cup_\module{M}(\delta \psi)
\end{equation}
and
\begin{equation}
    \phi\cup_{\module{M}}(\xi\cup_\module{M}\psi)
    =
    (\phi\cup\xi) \cup_\module{M} \psi.
\end{equation}

From now on let $M$ be a smooth manifold.  In differential geometry
the \emph{algebraic} Hochschild complex of $\algebra{A} = \Cinfty(M)$
is hardly of any interest. Instead, one requires the Hochschild
cochains to be multidifferential operators.

\subsection{Differential Operators and their Symbols}

In this and the following section we recall some well-known facts on
the global symbol calculus for (multi-)differential operators based on
the use of a covariant derivative. A systematic treatment is hard to
find in the literature, see \cite[Chap IV, §9]{palais:1965a}, we refer
to e.g. \cite{bordemann.neumaier.pflaum.waldmann:2003a} or
\cite[Appendix~A]{waldmann:2007a} for a formulation for differential
operators matching our needs rather closely.

Differential operators on sections of a vector bundle $E \to M$ with
values in sections of another vector bundle $F \to M$ can be defined
in many equivalent ways.  The perhaps most elegant and intrinsic
definition is by Grothendieck's inductive and algebraic
characterization \cite[Sect.~16.8]{grothendieck:1967a} that
$D \colon \Secinfty(E) \to \Secinfty(M)$ is of order zero if it is
$\Cinfty(M)$-linear, and of order $n \in \mathbb{N}$ if the commutator
$[D, f]$ with all left multiplications by functions $f \in \Cinfty(M)$
is of order $k-1$.  Then the differential operators $\Diffop^n(E;F)$
of order $n$ become a $\Cinfty(M)$-module from the left yielding the
filtered space of all differential operators
$\Diffop^\bullet(E;F) = \bigcup_{n=0}^\infty \Diffop^n(E;F)$ as
subspace of all linear maps between the sections.  It can then be
shown that a differential operator localizes to differential operators
$D_U\colon \Secinfty(E\at{U}) \to \Secinfty(F\at{U})$ for all open
$U \subseteq M$, i.e. $D$ becomes a sheaf morphism.  Moreover, in
local coordinates a differential operator is what one expects.
However, we are interested in a global symbol calculus.

In the following we will always consider $\Diffop(E;F)$ as a
$\Cinfty(M)$-bimodule with left action given by multiplication
$(f \acts D)(s) \coloneqq f \cdot D(s)$ and right action given by
$(D \racts f)(s) \coloneqq D(f \cdot s)$.

To study the symbol calculus fix a torsion-free covariant derivative
$\nabla$ on $M$.  Moreover, let $\nabla^E$ be a covariant derivative
for $E$.  On the symmetric $E$-valued $n$-forms
$\Secinfty(\Sym^n T^*M \tensor E)$ on $M$ one defines the symmetrized
covariant derivative
\begin{equation}
    \label{eq:SymDEvaluedDef}
    \SymD^E
    \colon
    \Secinfty(\Sym^n T^*M \tensor E)
    \to
    \Secinfty(\Sym^{n+1} T^*M \tensor E)
\end{equation}
by
\begin{equation}
    \label{eq:SymDDef}
    (\SymD^E \xi)(X_0, \ldots, X_n)
    =
    \sum_{i=0}^n
    \nabla^E_{X_i} \xi(X_0, \stackrel{i}{\ldots}, X_n)
    -
    \sum_{i \ne j}
    \xi(\nabla_{X_i} X_j,
    X_0, \stackrel{i}{\ldots} \, \stackrel{j}{\ldots}, X_n),
\end{equation}
where $X_0, \ldots, X_n \in \Secinfty(TM)$ are tangent vector fields
on $M$ and $\stackrel{i}{\ldots}$ means to omit the $i$-th
term. Indeed, a simple check shows that the right hand side is
$\Cinfty(M)$-linear in all arguments and hence a tensor field
$\SymD^E \xi \in \Secinfty(\Sym^{n+1} T^*M \tensor E)$ as claimed.  As
a particular but important case we can also omit the additional vector
bundle $E$ and obtain the symmetrized covariant derivative
\begin{equation}
    \label{eq:SymDWithoutE}
    \SymD
    \colon
    \Secinfty(\Sym^n T^*M)
    \to
    \Secinfty(\Sym^{n+1} T^*M),
\end{equation}
now only making use of $\nabla$.  The pre-factors in
\eqref{eq:SymDDef} are chosen in such a way that $\SymD^E$ becomes a
derivation of the symmetric tensor product. More precisely, we extend
$\SymD^E$ to a linear map
\begin{equation}
    \label{eq:SymDEOnCartProduct}
    \SymD^E
    \colon
    \prod_{n=0}^\infty \Secinfty(\Sym^n T^*M \tensor E)
    \to
    \prod_{n=0}^\infty \Secinfty(\Sym^n T^*M \tensor E).
\end{equation}
Then we have the Leibniz rule
\begin{equation}
    \label{eq:LeibnizRuleSymDE}
    \SymD^E(\xi \vee \eta)
    =
    \SymD \xi \vee \eta + \xi \vee \SymD^E \eta
\end{equation}
for $\xi \in \prod_{n=0}^\infty \Secinfty(\Sym^n T^*M)$ and
$\eta \in \prod_{n=0}^\infty \Secinfty(\Sym^n T^*M \tensor E)$.  On a
function $f \in \Cinfty(M)$ we have $\SymD f = \D f$, viewed as a
symmetric one-form.

In a next step we need symmetric multivector fields, i.e. sections
$X \in \Secinfty(\Sym^k TM)$. More generally, we consider
$X \tensor A \in \Secinfty(\Sym^k TM \tensor \Hom(E, F))$.  We have a
natural pairing
\begin{equation}
    \label{eq:SymmetricInsertion}
    \inss
    \colon
    \Secinfty(\Sym^n TM \tensor \Hom(E, F))
    \times
    \Secinfty(\Sym^m T^*M \tensor E)
    \to
    \Secinfty(\Sym^{m - n} T^*M \tensor F)
\end{equation}
defined by symmetrically inserting $X$ into the $\Sym^m T^*M$-part and
applying $A$ to the $E$-part. More precisely, the pre-factors are
fixed by specifying $\inss$ on factorizing sections
$X = X_1 \vee \cdots \vee X_n$ with
$X_1, \ldots, X_n \in \Secinfty(TM)$ and $A \in \Secinfty(\Hom(E, F))$
and $\xi \in \Secinfty(\Sym^m T^*M)$ and $s \in \Secinfty(E)$ by
\begin{equation}
    \label{eq:inssXalphas}
    \inss(X_1 \vee \cdots \vee X_n \tensor A) (\xi \tensor s)
    =
    \inss(X_1) \cdots \inss(X_n) \xi \tensor A(s),
\end{equation}
with the usual insertion derivation $\inss(X_i)\xi = \xi(X_i, \cdots)$
into the first argument. Of course, if $n > m$ then $\inss(X)$
vanishes on $\Secinfty(\Sym^m T^*M)$.  With this we can finally
introduce the global symbol calculus:
\begin{proposition}[Symbol calculus for differential operators]
    \label{proposition:SymbolCalculusDiffops}%
    Let $\pr_E\colon E \to M$ be a vector bundle over $M$ with
    covariant derivative $\nabla^E$.  Moreover, let
    $\pr_F\colon F \to M$ be another vector bundle, and let $\nabla$
    be a torsion-free covariant derivative on $M$.
    \begin{propositionlist}
    \item \label{item:OpXDef} For every fixed degree $n$ and every
        $X \tensor A \in \Secinfty(\Sym^n TM \tensor \Hom(E, F))$ the
        map
        \begin{equation}
            \label{eq:OpXDef}
            \Op(X \tensor A) \colon \Secinfty(E) \to \Secinfty(F)
        \end{equation}
        defined by
        \begin{equation}
            \label{eq:OpXOnSection}
            \Op(X \tensor A)s
            =
            \frac {1}{n!} \inss(X \tensor A) (\SymD^E)^n s
        \end{equation}
        is a differential operator of order $n \in \mathbb{N}_0$.
    \item \label{item:OpIsIso} The combined map for all degrees
        \begin{equation}
            \label{eq:OpSymToDiffop}
            \Op
            \colon
            \Secinfty(\Sym^\bullet TM \tensor \Hom(E, F))
            \to
            \Diffop^\bullet(E; F)
        \end{equation}
        is a left $\Cinfty(M)$-linear isomorphism of filtered modules,
        where on the left hand side we take the induced filtration
        from the grading.
    \end{propositionlist}
\end{proposition}
\begin{proof}
    The idea of the proof is that the \emph{leading symbol} of
    $\Op(X)$ reproduces the highest degree of $X$, a fact which can be
    inferred from the local description of $\Op(X)$. This results in
    an order-by-order construction of a preimage $X$ for a given
    differential operator, thus showing surjectivity. The injectivity
    is clear.
\end{proof}

For the scalar case $E = F = M \times \mathbb{R}$ we simply denote the
differential operators by $\Diffop^\bullet(M)$.  In this case the
above isomorphism is given by
\begin{equation}
    \Op
    \colon
    \Secinfty(\Sym^\bullet TM)
    \to
    \Diffop^\bullet(M).
\end{equation}

\begin{remark}[Sections of tensor products]
    \label{remark:SectionsTensorProducts}%
    To make our algebraic considerations available we have to use now
    the Serre-Swan theorem for vector bundles over manifolds: all the
    section spaces $\Secinfty(E)$ for any vector bundle $E$ are
    finitely generated projective modules over the functions. With
    this statement, sections of tensor products of bundles become
    tensor products of the sections of the individual bundles, i.e. we
    have
    \begin{equation}
        \label{eq:SerreSwanIsCool}
        \Secinfty(E \tensor F)
        =
        \Secinfty(E) \tensor[\Cinfty(M)] \Secinfty(F).
    \end{equation}
    Locally, this is obvious but we are interested in a global
    equality and therefore one needs the Serre-Swan theorem to see
    that the tensor products on the right indeed span the left hand
    side globally.  We will use this observation throughout the
    following without further mentioning.
\end{remark}

\subsection{Global Symbol Calculus for Multidifferential Operators}
\label{sec:GlobalSymbolCalculusMultiDiffops}

For multidifferential operators one proceeds essentially the same way.
Again, one has an intrinsic algebraic characterization of
multidifferential operators with a corresponding local description,
see e.g. \cite[App.~A]{waldmann:2007a} for a textbook.  Let us start
with the following simple situation: We consider multidifferential
operators of the form
\begin{equation}
    \label{eq:MultiDiffopScalar}
    D
    \colon
    \Cinfty(M) \times \cdots \times \Cinfty(M)
    \to
    \Cinfty(M),
\end{equation}
with $k$ arguments.  We have now a multi-order
$N = (n_1, \ldots, n_k) \in \mathbb{N}_0^k$ indicating the number of
differentiations in each argument.  We denote the multidifferential
operators on $M$ with $k$ arguments and multi-order $N$ by
$\Diffop^N(M)$, leading to a filtered $\Cinfty(M)$-submodule
\begin{equation}
    \label{eq:HCndiff}
    \HCdiff^k(M)
    \coloneqq
    \bigcup_{N \in \mathbb{N}_0^k}
    \Diffop^N(M)
\end{equation}
of the algebraic Hochschild complex $\HC^k(\Cinfty(M))$.  Note that we
suppress the symbol of smooth functions in our notation.  The needed
symbols will consist of a copy of the symmetric contravariant tensor
fields for each argument.  Thus we consider
\begin{equation}
    \label{eq:MultiSymbols}
    \Tensor^\bullet \Sym \Secinfty(TM)
    \coloneqq
    \bigoplus_{k=0}^\infty
    \bigoplus_{N \in \mathbb{N}_0^k}
    \Secinfty(\Sym^{n_1} TM \tensor \cdots \tensor \Sym^{n_k} TM).
\end{equation}
The pairing $\inss$ can then be extended to a pairing of
$\Tensor^k\Sym\Secinfty(TM)$ with $\Tensor^k\Sym \Secinfty(T^*M)$ by
specifying
\begin{equation}
    \label{eq:inssXonetoXn}
    \inss(X_1 \tensor \cdots \tensor X_k)
    =
    \inss(X_1) \tensor \cdots \tensor \inss(X_k)
\end{equation}
for
$X_1 \in \Sym^\bullet\Secinfty(TM), \dotsc, X_k \in
\Sym^\bullet\Secinfty(TM)$.
Putting things together we obtain the following symbol calculus:
\begin{proposition}[Symbol calculus for multidifferential operators I]
    \label{proposition:SymbolCalculusMultiDiffopsScalar}%
    Let $\nabla$ be a torsion-free covariant derivative on $M$.
    \begin{propositionlist}
    \item \label{item:OpXDefScalar} For every
        $X \in \Sym^{n_1} \Secinfty(TM) \tensor \cdots \tensor
        \Sym^{n_k} \Secinfty(TM) \subseteq \Tensor^k \Sym
        \Secinfty(TM)$ with $N = (n_1, \dotsc, n_k)$ the map
        \begin{equation}
            \label{eq:MultiOpXDef}
            \Op(X)
            \colon
            \Cinfty(M) \times \dots \times \Cinfty(M)
            \to
            \Cinfty(M)
        \end{equation}
        defined by
        \begin{equation}
            \label{eq:OpX}
            \Op(X)(f_1, \ldots, f_k)
            =
            \frac{1}{n_1! \cdots n_k!}
            \inss(X) \big(
            \SymD^{n_1} f_1 \tensor \cdots \tensor \SymD^{n_k} f_k
            \big),
        \end{equation}
        is a multidifferential operator of multi-order
        $N \in \mathbb{N}_0^k$.
    \item \label{item:MultiOpIsIso} The combined map
        \begin{equation}
            \label{eq:OpHCdiff}
            \Op
            \colon
            \Tensor^{k}\Sym\Secinfty(TM)
            \to
            \HCdiff^k(M)
        \end{equation}
        is an isomorphism of filtered $\Cinfty(M)$-modules, where on
        the left hand side we take the induced filtration from the
        grading.
    \end{propositionlist}
\end{proposition}

We can generalize the above situation.  For this consider a smooth map
$\phi \colon N \to M$ and let $E,F \to N$ be vector bundles over $N$.
Now we are interested in multidifferential operators of the form
\begin{equation}
    \label{eq:MultiDiffopVB}
    D
    \colon
    \Cinfty(M) \times \cdots \times \Cinfty(M)
    \to
    \Diffop(E; F),
\end{equation}
with $k$ arguments and values in the differential operators
$\Diffop(E; F)$ between vector bundles $E$ and $F$ over $N$, where we
consider $\Diffop(E; F)$ as a $\Cinfty(M)$-bimodule using the
pull-back $\phi^*\colon \Cinfty(M) \to \Cinfty(N)$.  We denote the
$\Diffop(E; F)$-valued multidifferential operators on $M$ with $k$
arguments and multi-order $N$ by $\Diffop^N(M; \Diffop(E; F))$,
leading to a filtered $\Cinfty(M)$-submodule
\begin{equation}
    \label{eq:HCndiff_VB}
    \HCdiff^k(M, \Diffop(E; F))
    \coloneqq
    \bigcup_{N \in \mathbb{N}_0^k}
    \Diffop^N(M; \Diffop(E; F))
\end{equation}
of the algebraic Hochschild complex $\HC^k(\Cinfty(M), \Diffop(E; F))$.

\begin{proposition}[Symbol calculus for multidifferential operators II]
    \label{prop:SymbolCalculusMultiDiffopsEF}%
    Let $M$ be a manifold and $\nabla$ a torsion-free covariant
    derivative on $M$.  Moreover, let $E,F \to M$ be vector bundles
    over $M$, with $\nabla^E$ a covariant derivative for $E$.
    \begin{propositionlist}
    \item \label{item:MultiOpXDef_VB} For every
        $X \in \Sym^{n_1} \Secinfty(TM)\tensor \cdots
        \tensor\Sym^{n_k} \Secinfty(TM) \subseteq \Tensor^k \Sym
        \Secinfty(TM)$, every $Y \in \Sym^n\Secinfty(TM) $ and all
        $A \in \Secinfty(\Hom(E,F))$ the map
        \begin{equation}
            \label{eq:MultiOpXDef_VB}
            \Op(X \tensor Y \tensor A )
            \colon
            \Cinfty(M) \times \cdots \times \Cinfty(M)
            \to
            \Diffop(E;F)
        \end{equation}
        defined by
        \begin{equation}
            \label{eq:OpXYN}
            \Op(X \tensor Y \tensor A)(f_1, \ldots, f_k)(s)
            =
            \frac{1}{n_1! \cdots n_k!}
            \inss(X) \big(
            \SymD^{n_1} f_1 \tensor \cdots \tensor \SymD^{n_k} f_k
            \big)
            \cdot
            \inss(Y \tensor A) (\SymD^E)^n s,
        \end{equation}
        is a multidifferential operator of multi-order
        $N = (n_1 \ldots, n_k) \in \mathbb{N}_0^k$ with values in
        $\Diffop^n(E; F)$.
    \item \label{item:OpIsIsodiffMan} The combined map
        \begin{equation}
            \label{eq:OpHCdiffnN}
            \Op
            \colon
            \Tensor^{k+1}\Sym\Secinfty(TM) \tensor \Secinfty(\Hom(E, F))
            \to
            \HCdiff^k(M, \Diffop(E; F))
        \end{equation}
        is a $\Cinfty(M)$-linear isomorphism of filtered modules,
        where on the left hand side we take the induced filtration
        from the grading.
    \end{propositionlist}
\end{proposition}

In the scalar case $E = F = M \times \mathbb{R}$ we have
\begin{equation}
    \Op
    \colon
    \Tensor^{k+1} \Sym \Secinfty(TM)
    \to
    \HCdiff^k(M,\Diffop(M)).
\end{equation}

As a last incarnation of the global symbol calculus, we consider a
smooth map $\phi \colon N \to M$ and differential operators
\begin{equation}
    \Cinfty(M)\times \cdots \times \Cinfty(M)
    \to
    \Diffop(N)
\end{equation}
as a left $\Cinfty(M)$-module using $\phi$. Moreover, we use the
obvious $\Cinfty(M)$-bimodule structure on $\Diffop(N)$ by pre- and
post-multiplications with the corresponding pull-back along $\phi$.
Not very surprisingly, these operators also have a symbol calculus.
\begin{proposition}[Symbol calculus for multidifferential operators III]
    \label{prop:SymbolCalculusMultiDiffopsdifferentMan}%
    Let $M$ and $N$ be manifolds with torsion-free covariant
    derivatives $\nabla^M$ and $\nabla^N$, respectively.
    Moreover, let $\phi \colon N \to M$ be a smooth map.
    \begin{propositionlist}
    \item \label{item:MultiOpXDef_MN} For every
        $X \in \Sym^{n_1} \Secinfty(TM) \tensor \cdots
        \tensor\Sym^{n_k} \Secinfty(TM)$ and
        $Y \in \Sym^n\Secinfty(TN)$ the map
        \begin{equation}
            \label{eq:MultiOpXDef_MN}
            \Op(X \tensor_{\Cinfty(M)} Y)
            \colon
            \Cinfty(M) \times \cdots \times \Cinfty(M)
            \to
            \Diffop(N)
        \end{equation}
        defined by
        \begin{equation}
            \label{eq:OpXYDiffopN}
            \Op(X \tensor[\Cinfty(M)] Y )(f_1, \ldots, f_k)(g)
            =
            \frac{1}{n_1! \cdots n_k!}
            \phi^*(\inss(X) \big(
            \SymD_M^{n_1} f_1 \tensor \cdots \tensor \SymD_M^{n_k} f_k
            \big))
            \inss(Y ) (\SymD_N)^n g,
        \end{equation}
        is a multidifferential operator of multi-order
        $(n_1, \ldots, n_k)\in \mathbb{N}_0^k$ with values in
        $\Diffop(N)$.
    \item \label{item:MultiOpIsIso_MN} The combined map
        \begin{equation}
            \label{eq:OpHCdiffnHomEF}
            \Op
            \colon
            \Tensor^\bullet\Sym\Secinfty(TM)
            \tensor[\Cinfty(M)]
            \Sym\Secinfty(TN)
            \to
            \HCdiff^\bullet(M, \Diffop(N))
        \end{equation}
        is a $\Cinfty(M)$-linear isomorphism of filtered modules,
        where on the left hand side we take the induced filtration
        from the grading.
    \end{propositionlist}
\end{proposition}
Note that
\begin{equation}
    \Tensor^{n}\Sym\Secinfty(TM) \tensor[\Cinfty(M)] \Sym\Secinfty(TN)
    \cong
    \Tensor^{n}\Sym\Secinfty(\phi^\# TM) \tensor \Sym\Secinfty(TN)
\end{equation}
as $\Cinfty(M)$-modules, where $\phi^\# TM \to N$ denotes the
pull-back bundle via $\phi$.  Moreover, using the map
\begin{equation}
    \phi_*\colon \Secinfty(TN)\ni
    Y \; \mapsto \;
    (p \; \mapsto \; (p,T_p\phi Y(p)))
    \in \Secinfty(\phi^\# TM)
\end{equation}
we can turn $\Sym\Secinfty(TN)$ into a
$\Sym\Secinfty(\phi^\# TM)$-comodule by
\begin{equation}
    \label{eq:leftcomodulechangeManifold}
    \Left (Y_1\vee\cdots\vee Y_k)
    =
    \sum_{i=0}^k \sum_{\sigma\in\Shuffle(i,k-i)}
    \phi_*Y_{\sigma(1)}\vee\cdots \vee\phi_*Y_{\sigma(i)}
    \tensor
    Y_{\sigma(i+1)}\vee\cdots\vee Y_{\sigma(k)}
    =
    \phi_*Y_\sweedler{-1}\tensor Y_\sweedler{0}.
\end{equation}
This means that we have a coalgebra differential on
$\Tensor^\bullet\Sym\Secinfty(\phi^\# TM) \tensor \Sym\Secinfty(TN)$
induced by the comodule structure.

The next and last step of this section is to compare the coalgebra
differentials with the Hochschild differentials of the corresponding
complexes.
\begin{theorem}[Differential Hochschild complex]
    \label{thm:SymbolCalculusIsoDifferential}
    Let $M$ and $N$ be smooth manifolds, let $\phi\colon N\to M$ be a
    smooth map, let $\nabla^M$ and $\nabla^N$ be covariant derivatives
    adapted to $\phi$, see \autoref{sec:AdaptedCovariantDerivatives},
    and let $E,F\to M$ be two vector bundles over $M$.
    \begin{theoremlist}
    \item \label{item:SymbolCalcIso_M} The map
        \begin{equation}
            \Op
            \colon
            \CCa^\bullet(\Secinfty(TM))
            \to
            \HCdiff^\bullet(M)
        \end{equation}
        from \autoref{proposition:SymbolCalculusMultiDiffopsScalar} is
        an isomorphism of differential graded algebras.
    \item \label{item:SymbolCalcIso_EF} The map
        \begin{equation}
            \Op
            \colon
            \CCa^\bullet(\Secinfty(TM),\Sym\Secinfty(TM)
            \tensor
            \Secinfty(\Hom(E, F)))
            \to
            \HCdiff^\bullet(M, \Diffop(E; F))
        \end{equation}
        from \autoref{prop:SymbolCalculusMultiDiffopsEF} is an
        isomorphism of differential graded modules along $\Op$ from
        \ref{item:SymbolCalcIso_M}, where the comodule structure on
        $\Sym\Secinfty(TM) \tensor \Secinfty(\Hom(E, F))$ is given by
        the formula from \ref{item:SymVTensorTrivialComoduleM} from
        \autoref{ex:Comodules}.
    \item \label{item:SymbolCalcIso_MN}
        The map
        \begin{equation}
            \Op
            \colon
            \CCa^\bullet(\Secinfty(\phi^\#TM),\Sym\Secinfty(TN))
            \to
            \HCdiff^\bullet(M, \Diffop(N))
        \end{equation}
        from \autoref{prop:SymbolCalculusMultiDiffopsdifferentMan} is
        an isomorphism of differential graded modules along $\Op$ from
        \ref{item:SymbolCalcIso_M}, where the comodule structure on
        $\Sym\Secinfty(TN)$ is given by
        \eqref{eq:leftcomodulechangeManifold}.
    \end{theoremlist}
    The differentials on the left hand sides are the coalgebra
    differentials from \autoref{prop:dgModuleStructureCCaVM}.
\end{theorem}
\begin{proof}
    Let us start with the key observation. Let
    $X_1\vee\cdots\vee X_n \in \Sym^n\Secinfty(TM)$ and let
    $f,g\in \Cinfty(M)$, then
    \begin{align*}
        &\Op(X)(fg)
        \\
        &\quad=
        \frac{1}{n!}\inss(X)(\SymD^nfg)
        \\
        &\quad=
        \frac{1}{n!}
        \inss(X)\bigg(
        \sum_{k=0}^n\binom{n}{k} \SymD^k f\vee \SymD^{n-k}g
        \bigg)
        \\
        &\quad=
        \sum_{k=0}^n\sum_{\sigma\in\Shuffle(k,n-k)}
        \frac{1}{(n-k)!k!}
        \inss( X_{\sigma(1)}\vee\cdots\vee X_{\sigma(k)})
        (\SymD^k f)
        \inss( X_{\sigma(k+1)}\vee\cdots\vee X_{\sigma(n)})
        (\SymD^{n-k}g)
        \\
        &\quad=
        \Op(\shcoprod X)(f,g).
    \end{align*}
    This means in particular that for the Hochschild differential we
    get
    \begin{align*}
        \delta\Op(X)(f,g)
        &=
        f\Op(X)(g) - \Op(X)(fg) + \Op(X)(f)g
        \\
        &=
        \Op(-\redshcoprod X)(f,g)
        \\
        &=
        \Op(\delta_\Ca(X)) (f,g).
    \end{align*}
    Note that it is obvious that $\Op(X\tensor Y)=\Op(X)\cup \Op(Y)$
    and \ref{item:SymbolCalcIso_M} is proven, since in both cases the
    differential is extended as a derivation of the products.  The
    proofs of \ref{item:SymbolCalcIso_EF} and
    \ref{item:SymbolCalcIso_MN} are very similar, so we only prove
    \ref{item:SymbolCalcIso_MN} for convenience.  Let
    $\alpha \in \Sym^n\Secinfty(T^*M)$, then we have
    \begin{equation*}
        \SymD_N \phi^*\alpha
        =
        \phi^* \SymD_M \alpha,
    \end{equation*}
    see \autoref{sec:AdaptedCovariantDerivatives}.  This means in
    particular that $\SymD_N^n \phi^*f = \phi^* \SymD_M^nf$ for all
    $f \in \Cinfty(M)$.  For $Y \in \Sym^n\Secinfty(TN)$, we get
    \begin{equation*}
        \Op(Y)(\phi^*f)
        =
        \frac{1}{n!}\inss(Y)(\phi^* \SymD_M^n f)
        =
        \frac{1}{n!}\inss(\phi_*Y)(\SymD_M^n f),
    \end{equation*}
    where we interpret $\phi_* Y \in \Secinfty(\phi^\# TM)$ via the
    tangent map of $\phi$.  Let now $Y \in \Sym^n\Secinfty(TN)$, let
    $f \in \Cinfty(M)$ and let $g \in \Cinfty(N)$, then
    \begin{equation*}
        \Op(Y)(\phi^*fg)
        =
        \Op(Y_\sweedler{1})(\phi^*f) \Op(Y_\sweedler{2})(g)
        =
        \Op(\phi_*Y_\sweedler{1} \tensor Y_\sweedler{2})(f)(g).
    \end{equation*}
    Applying the Hochschild differential to $\Op(Y)$ gives now exactly
    the correct formula for the coalgebra differential with values in
    a comodule. Now one can easily check that for
    $X \in \Tensor^\bullet\Sym\Secinfty(TM)$ and
    $Y \in \Sym\Secinfty(TN)$ one gets
    \begin{equation*}
        \Op(X \tensor[\Cinfty(M)] Y)
        =
        \Op(X) \cup_{\Diffop(N)} \Op(Y)
    \end{equation*}
    and the claim is proven, since we checked that the differentials
    coincide on generators.
\end{proof}
\begin{remark}
    One can also combine all the notions of geometrically relevant
    Hochschild cohomologies to get differential operators with values
    in differential operators between two vector bundles over a
    different base.  We decided not to include it in detail, even
    though the proofs look exactly the same, in case one has adapted
    covariant derivatives. If needed then the following techniques to
    compute Hochschild cohomologies can easily be adapted also to
    these more general situations.
\end{remark}

\section{The HKR Theorem in Differential Geometry}
\label{sec:HKRDifferentialGeometry}

For any manifold $M$ and torsion-free covariant derivative $\nabla$ on
it we know from
\autoref{thm:SymbolCalculusIsoDifferential}~\ref{item:SymbolCalcIso_M}
that $\Op \colon \CCa^\bullet(\Secinfty(TM)) \to \HCdiff^\bullet(M)$
is an isomorphism of complexes.
Moreover, using the van Est maps from \autoref{sec:VanEstMaps} we can define
\begin{align}
    \hkr \coloneqq \Op \circ \VanEstInt
    &\colon
    \Anti^\bullet \Secinfty(TM)
    \to
    \HCdiff^\bullet(M)
    \\
    \shortintertext{and}
    \hkrInv_\nabla \coloneqq \VanEstDiff \circ (\Op)^{-1}
    &\colon
    \HCdiff^\bullet(M)
    \to
    \Anti^\bullet \Secinfty(TM).
\end{align}
Then \autoref{thm:VanEstDeformationRetract} yields directly the
classical Hochschild-Kostant-Rosenberg (HKR) Theorem together with a global homotopy:
\begin{theorem}[HKR Theorem]
    \label{thm:classicalHKR}%
    Let $M$ be a manifold and let $\nabla$ be a torsion-free covariant
    derivative on $M$.  Then
    \begin{equation}
        \label{eq:HKRDeformationRetract}
        \begin{tikzcd}[column sep = large]
            \Anti^\bullet \Secinfty(TM)
            \arrow[r,"\hkr", shift left = 3pt]
            &\bigl( \HCdiff^{\bullet}(M),\delta \bigr)
            \arrow[l,"\hkrInv_\nabla", shift left = 3pt]
            \arrow[loop,
            out = -30,
            in = 30,
            distance = 30pt,
            start anchor = {[yshift = -7pt]east},
            end anchor = {[yshift = 7pt]east},
            "\hkrHomo"{swap}
            ]
        \end{tikzcd}
    \end{equation}
    with
    \begin{equation}
        \hkrHomo \coloneqq \mathord{\Op} \circ \VanEstHomo \circ (\Op)^{-1}
    \end{equation}
    is a deformation retract, i.e. the following holds:
    \begin{theoremlist}
    \item We have $\hkrInv_\nabla \circ \hkr = \id$.
    \item We have
        $\delta \hkrHomo + \hkrHomo \delta = \id - \hkr \circ
        \hkrInv_\nabla$.
    \end{theoremlist}
\end{theorem}
Indeed, \autoref{cor:CaCohomologyComputed} shows that in cohomology
the HKR-maps are compatible with the multiplications:
\begin{corollary}
    Let $M$ be manifold and let $\nabla$ be a torsion-free covariant
    derivative on $M$.  In cohomology
    $\hkr \colon \Anti^\bullet \Secinfty(TM) \to \HHdiff^\bullet(M)$
    is an isomorphism of graded algebras, with respect to $\wedge$ and
    $\cup$, with inverse $\hkrInv_\nabla$.
\end{corollary}

By \autoref{prop:VanEstMaps} we know that $\hkr$ is explicitly given
by
\begin{equation}
    \begin{split}
        \hkr(\xi_1 \wedge \dots \wedge \xi_\ell)
        &=
        \Op\bigg(
        \frac{1}{\ell!}\sum_{\sigma\in \Perm_\ell}\sign(\sigma) \cdot
        \xi_{\sigma(1)}\tensor\cdots\tensor \xi_{\sigma(\ell)}
        \bigg)
        \\
        &=
        \frac{1}{\ell!}\sum_{\sigma\in \Perm_\ell}\sign(\sigma) \cdot
        \Lie_{\xi_{\sigma(1)}}\cup\cdots\cup \Lie_{\xi_{\sigma(\ell)}},
    \end{split}
\end{equation}
where we used that $\Op$ is an algebra morphism turning tensor
products into cup-products and that for a vector field $X$ we get
$\Op(X) = \Lie_X$.  Note that this is exactly the classical definition
of the Hochschild-Kostant-Rosenberg map, which is \emph{independent}
of the chosen covariant derivative.
\begin{remark}[Dependence on $\nabla$]
    \label{remark:WhatIsNew}%
    The quasi-inverse $\hkrInv_\nabla$ and the homotopy $\hkrHomo$
    depend heavily on the chosen covariant derivative. While the
    actual cohomology is well-known the construction of the homotopy
    $\hkrHomo$ and the quasi-inverse $\hkrInv_\nabla$ as \emph{global}
    maps is new. They are completely canonical once a covariant
    derivative $\nabla$ is chosen for the symbol calculus. Moreover,
    by construction they can be split into a differential part
    determined by $\Op$ and a \emph{pointwise} part given by
    the van Est maps $\VanEstInt$ and $\VanEstDiff$ as well as the
    homotopy $\VanEstHomo$. Indeed, since $\VanEstInt$, $\VanEstDiff$
    and $\VanEstHomo$ are $\Cinfty(M)$-linear, they provide tensorial
    operators acting on the symbols.
\end{remark}
\begin{remark}[Differential operators vanishing on constants]
    Instead of considering the differential Hochschild complex we
    could have restricted ourselves to (multi-)differential operators
    vanishing on constants from the beginning.  In this case the
    symbols would be given by the tensor algebra
    $\Tensor^\bullet\redSym\Secinfty(TM)$ of the reduced symmetric
    algebra of $\Secinfty(TM)$.  By \autoref{sec:ReducedSymCoalgebra}
    we would still obtain a deformation retract computing the
    Hochschild cohomology with cochains vanishing on constants to be
    the multivector fields $\Anti^\bullet \Secinfty(TM)$.
\end{remark}
\begin{remark}
    Instead of interpreting $\Tensor^\bullet\Sym\Secinfty(TM)$
    as the coalgebra complex of $\Secinfty(M)$
    we could also interpret it as polynomial functions on
    $T^*M \oplus \dots \oplus T^*M$.
    In other words, this is the complex of polynomial (or homogeneous)
    Lie groupoid cochains for $T^*M \to M$ considered as a VB groupoid.
    The \cite{cabrera.drummond:2017a} yields the HKR isomorphism in cohomology.
    But it should be noted that \cite{cabrera.drummond:2017a} does not provide global homotopies.
\end{remark}

\subsection{Hochschild Cohomologies Associated to Vector Bundles}
\label{subsec:HochschildAssociatedVectorBundles}

As a first scenario with coefficients we consider two vector bundles
$\pr_E\colon E \to M$ and $\pr_F\colon F \to M$ over $M$. The sections
$\Secinfty(E)$ and $\Secinfty(F)$ are then $\Cinfty(M)$-modules, in
fact finitely generated projective modules by the Serre-Swan
theorem.

The first variant is a vector-valued version of the classical HKR
theorem based on the (symmetric) bimodule $\Secinfty(E)$. We consider
$E$-valued multidifferential operators
$\HCdiff^\bullet(M; E) = \HCdiff^\bullet(M) \tensor \Secinfty(E)$.
\begin{theorem}[Vector-Valued HKR Theorem]
    \label{thm:HKR_VectorValued}%
    Let $\pr_E\colon E \to M$ be a vector bundle over $M$. For a
    torsion-free covariant derivative $\nabla$ one obtains a
    deformation retract
    \begin{equation}
        \label{eq:VectorValuedHKR}
        \begin{tikzcd}[column sep = large]
            \Anti^\bullet \Secinfty(TM) \tensor \Secinfty(E)
            \arrow[r,"\hkr_E", shift left = 3pt]
            &\bigl( \HCdiff^{\bullet}(M, E),\delta \bigr)
            \arrow[l,"\hkrInv_E", shift left = 3pt]
            \arrow[loop,
            out = -30,
            in = 30,
            distance = 30pt,
            start anchor = {[yshift = -7pt]east},
            end anchor = {[yshift = 7pt]east},
            "\hkrHomo_E"{swap}
            ]
        \end{tikzcd}
        ,
    \end{equation}
    where $\hkr_E \coloneqq \hkr \tensor \id_{\Secinfty(E)}$,
    $\hkrInv_E \coloneqq \hkrInv_\nabla \tensor \id_{\Secinfty(E)}$ and
    $\hkrHomo_E \coloneqq \hkrHomo \tensor \id_{\Secinfty(E)}$.
    In particular, it holds
    \begin{equation}
        \HHdiff^\bullet(M,E) \simeq \Anti^\bullet \Secinfty(TM) \tensor \Secinfty(E)^.
    \end{equation}
\end{theorem}
\begin{proof}
    Indeed, this is the situation of trivial coefficients from
    \autoref{ex:CECohomology}, \ref{item:TrivialCE}.
\end{proof}

Slightly more interesting is the case where we consider differential
operators $\Diffop(E; F)$ mapping sections of $E$ to sections of $F$
as values, i.e. $\HCdiff^\bullet(M; \Diffop(E; F))$. Here we get the
following result:
\begin{theorem}[HKR Theorem with $\Diffop(E; F)$ as values]
    \label{thm:HKR_DiffopEFValues}%
    Let $\pr_E\colon E \to M$ and $\pr_F\colon F \to M$ be two vector
    bundles over $M$. Moreover, let $\nabla$ be a torsion-free
    covariant derivative on $M$ and let $\nabla^E$ be a covariant
    derivative for $E$. The global symbol calculus yields a
    deformation retract
    \begin{equation}
        \label{eq:HKRDiffopEtoF}
        \begin{tikzcd}[column sep = 3cm]
            \Secinfty(\Hom(E, F))
            \arrow[r,"\hkr_{\Diffop(E; F)}", shift left = 3pt]
            &\bigl( \HCdiff^{\bullet}(M, \Diffop(E; F)),\delta \bigr)
            \arrow[l,"\hkrInv_{\Diffop(E; F)}", shift left = 3pt]
            \arrow[loop,
            out = -30,
            in = 30,
            distance = 30pt,
            start anchor = {[yshift = -7pt]east},
            end anchor = {[yshift = 7pt]east},
            "\hkrHomo_{\Diffop(E; F)}"{swap}
            ]
        \end{tikzcd}
        ,
    \end{equation}
    where the left hand side is concentrated in degree zero. Here
    \begin{equation}
        \label{eq:HKRDiffopEF}
        \hkr_{\Diffop(E; F)}
        \coloneqq
        \Op \circ \VanEstInt
        \quad
        \textrm{and}
        \quad
        \hkrInv_{\Diffop(E; F)}
        \coloneqq
        \VanEstDiff \circ (\Op)^{-1}
    \end{equation}
    as well as
    $\hkrHomo_{\Diffop(E; F)} \coloneqq \mathord{\Op} \circ
    \VanEstHomo \circ (\Op)^{-1}$ with $\Op$ from
    \autoref{prop:SymbolCalculusMultiDiffopsEF}.
    In particular, we get
    \begin{equation}
        \HHdiff^{\bullet}(M, \Diffop(E; F))
        \simeq \Secinfty(\Hom(E, F))
    \end{equation}
    concentrated in degree $0$.
\end{theorem}
In particular, the higher Hochschild cohomologies with values in
$\Diffop(E; F)$ vanish. This result has important applications in the
deformation quantization of vector bundles, see
\cite{bursztyn.waldmann:2000b}, leading to a good understanding of
Morita theory of star products
\cite{bursztyn.dolgushev.waldmann:2012a, bursztyn.waldmann:2002a,
  bursztyn.waldmann:2004a}. In fact, the above statement can
alternatively be obtained from the fact that sections of vector
bundles form a finitely generated projective module over $\Cinfty(M)$.

\subsection{Hochschild Cohomologies Associated to a Submanifold}
\label{sec:Tangential}

We will use our results now to compute two versions of differential
Hochschild cohomologies associated to a closed embedded submanifold
$C \subseteq M$ with inclusion denoted by $\iota$.  We call
$\vanishing_C \coloneqq \ker \iota^* \subseteq \Cinfty(M)$ the
vanishing ideal of $C$, i.e. the set of all functions on $M$ vanishing
on $C$.  By \autoref{sec:AdaptedCovDer_Submanifolds} we can always
find a covariant derivative $\nabla$ on $M$ adapted to the submanifold
$C$, i.e. such that $\nabla_X Y \in \Secinfty_C(TM)$ for all
$X, Y \in \Secinfty_C(TM)$.  Here we denote by
\begin{equation}
    \Secinfty_C(TM)
    \coloneqq
    \left\{
        X \in \Secinfty(TM)
        \; \big| \;
        X\at{C} \in \Secinfty(TC)
    \right\}
\end{equation}
the vector fields on $M$ which are tangent to the submanifold
$C$. Phrased differently, $X \in \Secinfty_C(TM)$ iff there is a
vector field $X' \in \Secinfty(TC)$ such that $X'$ is $\iota$-related
to $X$, i.e. $X' \sim_\iota X$.

\subsubsection{Tangential Hochschild Cohomology}
\label{subsubsec:TangentialHochschildCohomology}%

A multidifferential operator $D \in \Diffop^k(M)$ is called
\emph{tangential} if $D(f_1, \dotsc, f_k) \in \vanishing_C$ whenever
there exists $i \in \{1, \dotsc, k\}$ such that
$f_i \in \vanishing_C$.  The set of all tangential multidifferential
operators is clearly a $\Cinfty(M)$-submodule of $\Diffop^\bullet(M)$,
which is preserved by the Hochschild differential $\delta$ and the cup
product $\cup$.  We will denote this subcomplex by
$\HCdiff^\bullet(M)_C \subseteq \HCdiff^\bullet (M)$. Note that this
is slightly different from the situation considered in
\cite{bordemann.et.al:2005a:pre, hurle:2018a:pre} where the cochains
$D$ are required to vanish if in the last argument a function of
$\vanishing_C$ is inserted.
\begin{proposition}
    \label{proposition:TangentialSymbolCalculus}%
    Let $C \subseteq M$ be a closed embedded submanifold.  Moreover,
    let $\nabla$ be a covariant derivative on $M$ adapted to the
    submanifold $C$.  Then $\Op$ yields an isomorphism
    \begin{equation}
        \label{eq:TangentialSymbolCalculus}
        \Op
        \colon
        \CCa^\bullet\big(\Secinfty_C(TM)\big)
        \to
        \HCdiff^\bullet(M)_C
    \end{equation}
    of differential graded algebras.
\end{proposition}
\begin{proof}
    Since $\nabla$ is adapted to $C$, i.e. there exists a covariant
    derivative $\nabla^C$ on $C$, such that the pair $\nabla$ and
    $\nabla^C$ are adapted with respect to $\iota$, and we have
    $\iota^*\circ \SymD = \SymD_C \circ \iota^*$.  Let now
    $X \in \Sym^k\Secinfty_C(TM)$ together with
    $X' \in \Sym^n\Secinfty(TC)$ such that $X' \sim_\iota X$
    and let
    $f \in \ker\iota^*\subseteq \Cinfty(M)$, then
    \begin{equation*}
        \iota^*\Op(X)(f)
        =
        \frac{1}{n!}\iota^*\inss(X)(\SymD^n f)
        =
        \frac{1}{n!}\inss(X')\iota^*(\SymD^n f)
        =
        \frac{1}{n!}\inss(X')(\SymD_C^n \iota^*f)
        =
        0.
    \end{equation*}
    Since $\Op$ is a morphism of differential graded algebras, we get
    that $\Op$ restricts to a map
    $\Op \colon \Tensor^\bullet (\Sym \Secinfty_C(TM)) \to
    \HCdiff^\bullet(M)_C$.  On the other hand, if a multidifferential
    operator $D$ is tangential, it is easy to see that its principal
    symbol is in $\Tensor^\bullet \Sym\Secinfty_C(TM)$ and the claim
    is proven.
\end{proof}
\begin{theorem}[Tangential HKR Theorem]
    \label{thm:HKR_Submanifold}%
    Let $C \subseteq M$ be a closed embedded submanifold and let
    $\nabla$ be a torsion-free covariant derivative on $M$ adapted to
    $C$, cf. \autoref{sec:AdaptedCovDer_Submanifolds}.  Then
    \begin{equation}
        \label{eq:TangentialDeformationRetract}
        \begin{tikzcd}[column sep = large]
            \Anti^\bullet \Secinfty_C(TM)
            \arrow[r,"\hkr", shift left = 3pt]
            &\bigl( \HCdiff^{\bullet}(M)_C,\delta \bigr)
            \arrow[l,"\hkrInv_\nabla", shift left = 3pt]
            \arrow[loop,
            out = -30,
            in = 30,
            distance = 30pt,
            start anchor = {[yshift = -7pt]east},
            end anchor = {[yshift = 7pt]east},
            "\hkrHomo"{swap}
            ]
        \end{tikzcd}
    \end{equation}
    is a deformation retract.  In particular, on cohomology we obtain an isomorphism
    \begin{equation*}
        \Anti^\bullet \Secinfty_C(TM) \to
    \HHdiff^\bullet(M)_C
    \end{equation*}
    of graded algebras.
\end{theorem}
\begin{proof}
    Applying \autoref{cor:CaCohomologyComputed} to the
    $\Cinfty(M)$-module $V = \Secinfty_C(TM)$ yields the result.
\end{proof}

The tangential Hochschild cohomology is a special case of constraint
Hochschild cohomology as introduced and discussed in
\cite{dippell.esposito.waldmann:2022a}, see also \cite{dippell:2023a}.
In particular, the above computes the Hochschild cohomology associated
to the constraint algebra
$\algebra{A} = (\Cinfty(M),\Cinfty(M),\vanishing_C)$.

\subsubsection{Coefficients in $\Diffop(C)$}
\label{subsubsec:CoefficientsDiffopC}

There is another Hochschild cohomology associated to a submanifold,
namely $\HCdiff(M,\Diffop(C))$, i.e. multidifferential operators on
$M$ with values in differential operators on $C$. These differential
operators become a $\Cinfty(M)$-bimodule via the pull-back $\iota^*$
with the inclusion map, i.e. the restriction of functions from $M$ to
$C$.  In \autoref{thm:SymbolCalculusIsoDifferential}, we already
discussed the isomorphism
\begin{align}
    \Op\colon \CCa\bigl(\Secinfty(TM),\Sym\Secinfty(TC)\bigr)
    \to
    \HCdiff\bigl(M,\Diffop(C)\bigr)
\end{align}
and thus we can directly use the techniques from
\autoref{part:CoalgebraCohomology}.

\begin{theorem}[HKR Theorem with values in $\Diffop(C)$]
    \label{thm:HKR_ValuesDiffopC}%
    Let $C$ be an embedded submanifold of $M$, such that there exist
    adapted covariant derivatives.  Moreover, let $TC^\perp \to C$ be
    a complementary vector bundle to $TC$ inside $TM\at{C}$, i.e.
    $TM\at{C} = TC \oplus TC^\perp$.  Then there exists a deformation
    retract
    \begin{equation}
        \begin{tikzcd}
            \Anti^\bullet \Secinfty\bigl(TC^\perp\bigr)
            \arrow[r, shift left = 3pt]
            &\bigl( \HCdiff^{\bullet}(M,\Diffop(C)),\delta \bigr)
            \arrow[l, shift left = 3pt]
            \arrow[loop,
            out = -30,
            in = 30,
            distance = 30pt,
            start anchor = {[yshift = -7pt]east},
            end anchor = {[yshift = 7pt]east}
            ]
        \end{tikzcd}
        ,
    \end{equation}
    depending on the covariant derivatives and the choice of the
    complementary vector bundle.
    In particular, we obtain an isomorphism
    \begin{equation*}
        \HCdiff^{\bullet}(M,\Diffop(C)) \simeq \Anti^\bullet \Secinfty\bigl(TC^\perp\bigr)
    \end{equation*}
    of left modules along
    $\hkr \colon \Anti^\bullet\Secinfty(TM)
    \to \HHdiff^\bullet(M)$
\end{theorem}
\begin{proof}
    From \autoref{thm:SymbolCalculusIsoDifferential} we know that
    $\HCdiff(M,\Diffop(C))$ is isomorphic to the cochain complex
    $\CCa(\Secinfty(\iota^\#TM), \Sym\Secinfty(TC))$.  Using
    \autoref{ex:VanEstWithCoeff}~\ref{item:CoalgebraCohomologySubmoduleUV}
    for $U = \Secinfty(TC)$ and $U^\perp = \Secinfty(TC^\perp)$ we can
    therefore conclude that there is a deformation retract of the
    above form. Note that both vector bundles over $C$ become
    $\Cinfty(M)$-modules via the restriction $\iota^*$.
\end{proof}

\subsection{Hochschild Cohomologies Associated to a Surjective Submersion}
\label{sec:Surjective Submersions}

We compute two different Hochschild cohomologies associated to a
surjective submersion $\pr \colon P \to M$. Throughout this section we
assume that $\pr\colon P\to M$ has connected fibers.  We denote the
vertical subbundle by $F \coloneqq \ker T\pr$.  To establish a good
symbol calculus we fix an always existing adapted covariant derivative
$\nabla^P$ on $P$, see \autoref{sec:AdaptedCovDer_Submersions}, and a
horizontal subbundle $F^\perp$ such that $TP = F \oplus F^\perp$.  The
horizontal lift of vector fields $X \in \Secinfty(TM)$ will be denoted
by $X^\hor \in \Secinfty(F^\perp)^F$, where $\Secinfty(F^\perp)^F$
denotes the horizontal sections which are covariantly constant along
the fibers.  With this we obtain an isomorphism
\begin{equation}
    \label{eq:HCIsoToTensorSymFperpF}
    \HCdiff^\bullet(M)
    \simeq
    \Tensor^\bullet \Sym \Secinfty(F^\perp)^F,
\end{equation}
of $\Cinfty(M) \simeq \Cinfty(P)^F$-modules.  Note that
$\Tensor^\bullet\Sym \Secinfty(F^\perp)$ is (globally)
$\Cinfty(P)$-spanned by $\Cinfty(M)$-submodule
$\Tensor^\bullet\Sym\Secinfty(F^\perp)^F$.  This is again a
consequence of the Serre-Swan Theorem applied to $TM$.

\subsubsection{Coefficients in $\Diffop(P)$}
\label{subsubsec:CoefficientsDiffopP}

We consider the (scalar) differential operators $\Diffop^\ell(P)$ on
$P$ as a left $\Cinfty(M)$-module via the pull-back $\pr^*$ and
multiplication from the left.  Then we are interested in the
multidifferential operators on $M$ with values in $\Diffop^\ell(P)$
for all $\ell \in \mathbb{N}_0$, i.e. we consider
\begin{equation}
    \label{eq:HCdiffDiffopEllP}
    \begin{split}
        \HCdiff^n(M, \Diffop^\bullet(P))
        &=
        \bigcup_{\ell = 0}^\infty
        \HCdiff^n(M, \Diffop^\ell(P))
        \\
        &=
        \HCdiff^n\Big(
        M,
        \bigcup\nolimits_{\ell = 0}^\infty\Diffop^\ell(P)
        \Big),
    \end{split}
\end{equation}
see in particular
\cite[Lemma~4.4]{bordemann.neumaier.waldmann.weiss:2010a} for the
second equality.  We combine now the symbol calculi on $M$ and on $P$
as follows.
\begin{proposition}
    \label{proposition:SymbolCalcSurSub}%
    Let $\pr \colon P \to M$ be a surjective submersion with vertical
    bundle $F$.  Moreover, let $\nabla^P$ be a covariant derivative on
    $P$ adapted to a given torsion-free covariant derivative $\nabla$
    on $M$ and choose a horizontal bundle $F^\perp$.  Then
    \begin{equation}
        \label{eq:OpForDiffopP}
        \Op
        \colon
        \CCa^\bullet(\Secinfty(F^\perp), \Sym^\bullet \Secinfty(TP))
        \to
        \HCdiff^\bullet(M, \Diffop^\bullet(P))
    \end{equation}
    defined on factorizing sections by
    \begin{equation}
        \label{eq:OpValuesInDiffopPDef}
        \Op(X \tensor Y)(f_1, \ldots, f_k; g)
        =
        \iota^*\big(
        \inss(X) \big(
        \E^{\SymD_P} \pr^*f_1
        \tensor \cdots \tensor
        \E^{\SymD_P} \pr^*f_k
        \big)
        \tensor
        \inss(Y) \E^{\SymD_P} g
        \big)
    \end{equation}
    where $f_1, \ldots, f_k \in \Cinfty(M)$, $g \in \Cinfty(P)$,
    $X \in \Tensor^k\Sym\Secinfty(F^\perp)$ and
    $Y \in \Sym^\bullet \Secinfty(TP)$ is an isomorphism
    of differential graded modules along
    $\Op \colon \CCa^\bullet(\Secinfty(TM)) \to \HCdiff^\bullet(M)$.
\end{proposition}
\begin{proof}
    Without restriction, $X$ is actually a horizontal lift from $M$
    since the $\Cinfty(P)$-linear combinations can be obtained by
    making use of general sections $Y \in \Secinfty(\Sym^\ell TP)$.
    \begin{equation*}
        \label{eq:InssHorizontalLift}
        \inss(X^\hor)
        \Big(
        \E^{\SymD_P} \pr^*f_1
        \tensor \cdots \tensor
        \E^{\SymD_P} \pr^*f_k
        \Big)
        =
        \pr^*\Big(
        \inss(X)
        \big(
        \E^{\SymD} f_1
        \tensor \cdots \tensor
        \E^{\SymD} f_k
        \big)
        \Big)
    \end{equation*}
    is just the pull-back of the scalar multidifferential operator
    $\Op(X)$ applied to the functions $f_1, \ldots, f_k$. Again, this
    calculus exhausts all elements of $\HCdiff^k(M, \Diffop(P))$,
    i.e. \eqref{eq:OpForDiffopP} becomes an isomorphism of left
    $\Cinfty(M)$-modules.
\end{proof}
\begin{theorem}[HKR Theorem for surjective submersions]
    \label{thm:HKR_surjSubmersions}%
    Let $\pr \colon P \to M$ be a surjective submersion and let
    $\nabla^P$ be a torsion-free covariant derivative on $P$ adapted
    to a given torsion-free covariant derivative $\nabla$ on $M$,
    cf. \autoref{sec:AdaptedCovDer_Submersions}.  Then there is a
    deformation retract
    \begin{equation}
        \label{eq:SurjectiveSubmersionRetract}
        \begin{tikzcd}[column sep = large]
            \Diffop_\ver(P)
            \arrow[r,"\hkr", shift left = 3pt]
            &\bigl( \HCdiff^{\bullet}(M,\Diffop(P)),\delta \bigr)
            \arrow[l,"\hkrInv", shift left = 3pt]
            \arrow[loop,
            out = -30,
            in = 30,
            distance = 30pt,
            start anchor = {[yshift = -7pt]east},
            end anchor = {[yshift = 7pt]east},
            "\hkrHomo"{swap}
            ]
        \end{tikzcd}
    \end{equation}
    depending on the choice of the covariant derivatives.  In
    particular, we have
    \begin{equation}
        \HHdiff^\bullet\big(
        M,\Diffop(P)
        \big)
        \simeq
        \Diffop_\ver(P),
    \end{equation}
    concentrated in degree $0$, where
    $\Diffop_\ver(P)$ denotes the differential operators which only
    derive in fiber direction.
\end{theorem}
\begin{proof}
    Let us choose a horizontal bundle $F^\perp$, i.e. a splitting
    $TP = F^\perp \oplus F$.  Then
    $\Sym \Secinfty(TP) = \Sym \big(\Secinfty(F^\perp) \oplus
    \Secinfty(F) \big) \simeq \Sym\Secinfty(F^\perp) \tensor
    \Sym\Secinfty(F)$, with $\Sym\Secinfty(F^\perp)$-coaction on
    $\Sym \Secinfty(TP)$ given by the canonical coaction on
    $\Sym\Secinfty(F^\perp)$ and trivial coaction on
    $\Sym\Secinfty(F)$.  Using
    \autoref{ex:CoalgebraCohomology}~\ref{item:CoalgebraCohomology_Submodule}
    applied to $W = \Secinfty(TP)$ and $V^\perp = \Secinfty(F)$ and
    the fact that $\Sym\Secinfty(TF) \simeq \Diffop_\ver(P)$ we
    immediately obtain the above deformation retract.
\end{proof}

The computation of the cohomology was obtained already in
\cite{bordemann.neumaier.waldmann.weiss:2010a}. Here we have the
explicit and global homotopies as additional feature.

\subsubsection{Projectable Hochschild Cohomology}
\label{subsubsec:ProjectableHochschildCohomology}%

Another Hochschild complex associated to a surjective submersion
$\pr \colon P \to M$ with vertical bundle $F$ is given by
\begin{equation}
    \HCdiff^k(P)^F
    \coloneqq
    \left\{
        D \in \HCdiff^k(P)
        \; \Big| \;
        D(f_1, \dotsc, f_k) \in \Cinfty(P)^F
        \textrm{ for all }
        f_1, \dotsc, f_k \in \Cinfty(P)^F
    \right\},
\end{equation}
where $\Cinfty(P)^F$ denotes the functions constant along the fibers
of the submersion $\pr$.  We call these differential operators
\emph{projectable}, since to every $D \in \Diffop^k(P)^F$ we can
assign a differential operator $\check{D} \in \Diffop^k(M)$ by
defining $\check{D}(g_1, \dotsc, g_k)$ to be the uniquely defined
function $g \in \Cinfty(M)$ such that
$D(\pr^*g_1, \dotsc, \pr^*g_k) = \pr^*g$.  Moreover, the Hochschild
differential $\delta$ clearly preserves $\HCdiff^\bullet(P)^F$,
turning it into a subcomplex of $\HCdiff^\bullet(P)$.  To compute the
projectable Hochschild cohomology note that the map
$D \mapsto \check{D}$ is in fact a surjective morphism of cochain
complexes onto $\Diffop^\bullet(M)$.  If we consider its kernel
\begin{equation}
    \HCdiff^k(P)_0
    \coloneqq
    \left\{
        D \in \HCdiff^k(P)
        \; \Big| \;
        D(f_1, \dotsc, f_k) = 0
        \textrm{ for all }
        f_1, \dotsc, f_k \in \Cinfty(P)^F
    \right\},
\end{equation}
then we obtain an isomorphism
\begin{equation}
    \label{eq:SplittingHCPF}
    \HCdiff^\bullet(P)^F
    \simeq
    \HCdiff^\bullet(P)_0
    \oplus
    \HCdiff^\bullet(M)
\end{equation}
of complexes.  Thus it remains to compute the cohomology of
$\HCdiff^\bullet(P)_0$.

For this we use again a covariant derivative $\nabla^P$ on $P$ adapted
to $\pr$ and the choice of a horizontal bundle $F^\perp$.  Then using
\eqref{eq:SplittingHCPF} we obtain an isomorphism
\begin{equation}
    \label{eq:SplittingHCP}
    \HCdiff^\bullet(P)
    \simeq
    \HCdiff^\bullet(P)_0
    \oplus
    \Cinfty(P) \cdot \HCdiff^\bullet(M),
\end{equation}
here $\Cinfty(P) \cdot \HCdiff^\bullet(M)$ denotes the
$\Cinfty(P)$-module generated by $\HCdiff^\bullet(M)$.  We already
discussed that taking symbols of $\Cinfty(P) \cdot \HCdiff^\bullet(M)$
yields an isomorphism
$\Cinfty(P) \cdot \HCdiff^\bullet(M) \simeq \Tensor^\bullet\Sym
\Secinfty(F^\perp)$, and thus an associated homotopy retract
\begin{equation}
    \label{diag:DefRetractPgeneratedHCM}
    \begin{tikzcd}[column sep = large]
        \Anti^\bullet \Secinfty(TF^\perp)
        \arrow[r,"\hkr", shift left = 3pt]
        &\bigl(\Cinfty(P) \cdot \HCdiff^{\bullet}(M),\delta \bigr)
        \arrow[l,"\hkrInv_\nabla", shift left = 3pt]
        \arrow[loop,
        out = -30,
        in = 30,
        distance = 30pt,
        start anchor = {[yshift = -7pt]east},
        end anchor = {[yshift = 7pt]east},
        "\hkrHomo"{swap}
        ]
    \end{tikzcd}.
\end{equation}
With this we are finally able to compute the projectable Hochschild
cohomology:
\begin{theorem}[Projectable HKR]
    \label{thm:HKR_Projectable}%
    Let $\pr \colon P \to M$ be a surjective submersion and let
    $\nabla^P$ be a torsion-free covariant derivative on $P$ adapted
    to a given torsion-free covariant derivative $\nabla$ on $M$,
    cf. \autoref{sec:AdaptedCovDer_Submersions}.  Moreover, let
    $F^\perp$ be a horizontal bundle for $\pr$.
    \begin{theoremlist}
    \item There exists a deformation retract
        \begin{equation}
            \begin{tikzcd}
                \Anti^{\bullet-1} \Secinfty(TP)
                \wedge
                \Secinfty(F)
                \arrow[r, shift left = 3pt]
                &\bigl( \HCdiff^{\bullet}(P)_0,\delta \bigr)
                \arrow[l, shift left = 3pt]
                \arrow[loop,
                out = -30,
                in = 30,
                distance = 30pt,
                start anchor = {[yshift = -7pt]east},
                end anchor = {[yshift = 7pt]east}
                ]
            \end{tikzcd}
            .
        \end{equation}
        In particular, we have an isomorphism
        \begin{equation}
            \HHdiff^\bullet(P)_0
            \simeq
            \Anti^{\bullet-1} \Secinfty(TP)
            \wedge
            \Secinfty(F)
        \end{equation}
        as graded $\Cinfty(P)$-algebras.
    \item There exists a deformation retract
        \begin{equation}
            \begin{tikzcd}
                \Anti^{\bullet-1} \Secinfty(TP)
                \wedge
                \Secinfty(F)
                \oplus
                \Anti^\bullet\Secinfty(F^\perp)^F
                \arrow[r, shift left = 3pt]
                &\bigl( \HCdiff^{\bullet}(P)^F,\delta \bigr)
                \arrow[l, shift left = 3pt]
                \arrow[loop,
                out = -30,
                in = 30,
                distance = 30pt,
                start anchor = {[yshift = -7pt]east},
                end anchor = {[yshift = 7pt]east}
                ]
            \end{tikzcd}
            .
        \end{equation}
        In particular, we have an isomorphism
        \begin{equation}
            \HCdiff^\bullet(P)^F
            \simeq
            \Anti^{\bullet-1} \Secinfty(TP)
            \wedge
            \Secinfty(F)
            \oplus
            \Anti^\bullet\Secinfty(F^\perp)^F.
        \end{equation}
        of graded $\Cinfty(P)^F$-algebras.
    \end{theoremlist}
\end{theorem}
\begin{proof}
    We can obtain the first part as a quotient of the HKR deformation
    retract for $\HCdiff^\bullet(P)$ by
    \eqref{diag:DefRetractPgeneratedHCM} using
    \eqref{eq:SplittingHCP}.  The second part then directly follows by
    taking the direct sum of deformation retracts, see
    \autoref{prop:SumTensorHomotopies} according to
    \eqref{eq:SplittingHCPF}.
\end{proof}

This is another instance of the coisotropic Hochschild cohomology used
in the study of Poisson structures which are compatible with
coisotropic reduction, see \cite{dippell.esposito.waldmann:2022a,
  dippell:2023a}.  In particular, the above computes the Hochschild
cohomology associated to the constraint algebra
$\algebra{A} = (\Cinfty(P),\Cinfty(P)^F,0)$.

\subsection{Hochschild Cohomology Associated to a Foliation}
\label{subsec:HochschildCohomologyAssociatedFoliation}%

Given a regular foliation on $M$ with distribution
$F \subseteq TM$ one can consider multidifferential operators which
derive in the direction of the foliation, i.e. $D\in \HCdiff^k(M)$
shall fulfil
\begin{equation}
    D(f_1, \ldots, f_{i-1}, g, f_{i}, \ldots, f_{k-1})
    -
    gD(f_1, \ldots, f_{i-1}, 1, f_{i}, \ldots, f_{k-1})
    =
    0
\end{equation}
for all $f_1, \ldots, f_{k-1} \in \Cinfty(M)$, and for all
$g \in \Cinfty(M)^F$ and all $i\in\{1, \ldots, k\}$. Here
$\Cinfty(M)^F \subseteq \Cinfty(M)$ denotes those functions $g$ which
are constant in direction of the distribution, i.e. $\Lie_X g = 0$ for
all $X \in \Secinfty(F)$.  We denote these multidifferential operators
by $\HCdiff^\bullet(M)_F$ and one can easily show that it is a subcomplex
of $\HCdiff^\bullet(M)$.

We call a covariant derivative $\nabla$ on $M$ which restricts to $\Secinfty(F)$
\emph{adapted to $F$}.
Note that such covariant derivatives always exist by choosing a complement of $F$
and projecting appropriately.
\begin{theorem}[Foliation HKR]
    \label{thm:HKR_Foliation}
    Let $F \subseteq TM$ be an integrable regular smooth distribution,
    and let $\nabla$ be a torsion-free covariant derivative adapted to $F$.
    Then there exists a deformation retract
    \begin{equation}
        \begin{tikzcd}
            \Anti^\bullet \Secinfty(F)
            \arrow[r, shift left = 3pt]
            &\bigl( \HCdiff^{\bullet}(M)_F,\delta \bigr)
            \arrow[l, shift left = 3pt]
            \arrow[loop,
            out = -30,
            in = 30,
            distance = 30pt,
            start anchor = {[yshift = -7pt]east},
            end anchor = {[yshift = 7pt]east}
            ]
        \end{tikzcd}
        .
    \end{equation}
    In particular, we have an isomorphism
    \begin{equation}
        \HHdiff^{\bullet}(M)_F \simeq \Anti^\bullet \Secinfty(F)
    \end{equation}
    of graded $\Cinfty(M)$-algebras.
\end{theorem}
\begin{proof}
    Let us choose a covariant derivative $\nabla$ adapted to $F$,
    then we can define the map
    \begin{equation*}
        \Op \colon \CCa^\bullet(\Secinfty(F)) \to \HCdiff^\bullet(M)_F,
    \end{equation*}
    which turns out to be an isomorphism. Using
    \autoref{thm:VanEstDeformationRetract}, one gets the desired
    deformation retract.
\end{proof}

\subsection{Invariant HKR Theorems}
\label{sec:InvariantHKR}

We have already seen in the previous parts that our construction of
the deformation retract respects geometric structures such as
submanifolds and surjective submersions as long as the covariant
derivative respects them as well.  In this subsection, we exploit this
observation further and discuss the case of a Lie group (resp. Lie
algebra) acting on a manifold.

Recall that a covariant derivative $\nabla$ is
\emph{$\group{G}$-invariant} with respect to a smooth action
$\Phi\colon \group{G} \times M \to M$ of a Lie group $\group{G}$ if
\begin{equation}
    \label{eq:InvariantConnection}
    \Phi_g^*(\nabla_X Y)
    =
    \nabla_{\Phi_g^*X}\Phi_g^*Y
\end{equation}
for all $X, Y \in \Secinfty(TM)$ and all $g\in \group{G}$.  In case of
a Lie algebra action $\phi\colon \liealg{g} \to \Secinfty(TM)$ the
covariant derivative $\nabla$ is called \emph{$\liealg{g}$-invariant},
if
\begin{equation}
    \label{eq:InfinitesimalInvariantConnection}
    [\phi(\xi),\nabla_XY]
    =
    \nabla_{[\phi(\xi),X]}Y + \nabla_X[\phi(\xi),Y]
\end{equation}
for all $X, Y \in \Secinfty(TM)$ and all $\xi \in \liealg{g}$.  Note
that if a covariant derivative is $\group{G}$-invariant, then it is
also $\mathrm{Lie}(\group{G})=\liealg{g}$-invariant with respect to
the infinitesimal action of $\group{G}$. The converse is true if
$\group{G}$ is connected. Nevertheless, a Lie algebra action does not
need to integrate to a Lie group action.
\begin{remark}
    In general one cannot guarantee the existence of a
    $\group{G}$-invariant covariant derivative, but we have the
    following implications:
    \begin{equation}
        \label{eq:ProperActionStuff}
        \group{G} \textrm{ compact }
        \implies
        \group{G} \textrm{ acts properly }
        \implies
        \textrm{ Existence of a }
        \group{G}\textrm{-invariant covariant derivative}.
    \end{equation}
    This is usually shown by a standard averaging procedure for proper
    actions. Note that none of the above implications is actually an
    equivalence: the first one is clear and for the second one the
    linear action of $\operator{GL}_n$ on $\mathbb{R}^n$ admits an
    invariant covariant derivative (the canonical flat one), but the
    action is not proper.
\end{remark}

If $D \in \Diffop^k(M)$ is a multidifferential operator we can act on
$D$ with diffeomorphisms as well as with vector fields as follows. One
defines $\Phi^*D \in \Diffop^k(M)$ by
\begin{equation}
    \label{eq:PhiPullbackDiffop}
    (\Phi^*D)(f_1, \ldots, f_k)
    =
    \Phi^*(D(\Phi_* f_1, \ldots, \Phi_* f_k))
\end{equation}
where $\Phi^*$ is the usual pull-back of functions and
$\Phi_* = (\Phi^{-1})^*$. It requires only a little check that
$\Phi^*D$ is indeed again a multidifferential operator of the same
order as $D$. Similarly, one defines the Lie derivative
$\Lie_X D \in \Diffop^k(M)$ of $D$ with respect to a vector field
$X \in \Secinfty(TM)$ by
\begin{equation}
    \label{eq:LieXDDef}
    (\Lie_X D)(f_1, \ldots, f_k)
    =
    \Lie_X (D(f_1, \ldots, f_k))
    -
    \sum_{r=1}^k
    D(f_1, \ldots, \Lie_X f_r, \ldots, f_k).
\end{equation}
Again, $\Lie_X D$ is indeed a multidifferential operator of the same
order as $D$. In fact, $\Lie_X D$ is the Gerstenhaber bracket of the
derivation $\Lie_X$ with $D$.

In particular, in both cases we can speak of an \emph{invariant}
multidifferential operator. We denote these invariant elements of the
differential Hochschild complex by $\HCdiff(M)^\group{G}$ in the case
of a group action and by $\HCdiff(M)^\liealg{g}$ in the case of a Lie
algebra action, respectively.
\begin{proposition}
    \label{prop:invarianSymbCalc}%
    Let $\Phi\colon \group{G} \times M \to M$ be a Lie group action
    (resp.  $\phi\colon \liealg{g} \to \Secinfty(TM)$ a Lie algebra
    action) and let $\nabla$ be a torsion-free $\group{G}$-invariant
    (resp. $\liealg{g}$-invariant) covariant derivative.
    \begin{propositionlist}
    \item \label{item:OpIsEquivariant} The global symbol calculus is
        equivariant, i.e. we have
        \begin{equation}
            \label{eq:OpGequivariant}
            \mathord{\Op} \circ \Phi_g^*
            =
            \Phi_g^* \circ \Op
        \end{equation}
        for all $g \in \group{G}$ in the case of a group action. In
        the case of a Lie algebra action we have the infinitesimal
        equivariance
        \begin{equation}
            \label{eq:InfinitesimalEquivarianceOp}
            \mathord{\Op} \circ \Lie_{\phi(\xi)}
            =
            \Lie_{\phi(\xi)} \circ \Op
        \end{equation}
        for all $\xi \in \liealg{g}$.
    \item \label{item:InvariantSymbols} The equivariant symbol
        calculus restricts to an isomorphism
        \begin{equation}
            \label{eq:EquivSymbolCalInvariantSymbols}
            \Op\colon
            \bigl(\Tensor^\bullet\Secinfty(\Sym^\bullet TM)\bigr)^\group{G}
            \to
            \HCdiff(M)^\group{G}
        \end{equation}
        of differential graded algebras between $\group{G}$-invariant symbols and
        $\group{G}$-invariant multidifferential operators in the case
        of a group action. In the case of a Lie algebra action we get
        the isomorphism
        \begin{equation}
            \label{eq:LiealgInVSymbols}
            \Op\colon
            \bigl(\Tensor^\bullet \Secinfty(\Sym^\bullet TM)\bigr)^\liealg{g}
            \to
            \HCdiff(M)^\liealg{g}
        \end{equation}
        of differential graded algebras.
    \end{propositionlist}
\end{proposition}
\begin{proof}
    Note that the symmetrized covariant derivative from
    Equation~\eqref{eq:SymDWithoutE} fulfils
    \begin{equation*}
        \Phi_g^* \circ \SymD
        =
        \SymD \circ \Phi^*_g
    \end{equation*}
    and likewise for a Lie algebra action.  With this,
    \eqref{eq:OpGequivariant} and
    \eqref{eq:InfinitesimalEquivarianceOp}, respectively, are direct
    consequences. Since $\Op$ is a bijection between symbols and
    multidifferential operators, the second part follows from the
    first at once.
\end{proof}
\begin{theorem}[Invariant HKR]
    \label{thm:HKR_inv}%
    Let $\Phi \colon \group{G} \times M \to M$ be a Lie group action
    (resp.  $\phi\colon \liealg{g}\to \Secinfty(TM)$ a Lie algebra
    action).  Suppose there exists a torsion-free $\group{G}$-invariant
    (resp. $\liealg{g}$-invariant) connection $\nabla$.
    \begin{theoremlist}
    \item \label{item:EquivariantHomotopies} All structure maps in the
        deformation retract \eqref{eq:HKRDeformationRetract} from
        \autoref{thm:classicalHKR} are equivariant. In particular, the
        homotopy $\hkrHomo$ is equivariant.
    \item \label{item:InvariantDeformationRetract} The deformation
        retract \eqref{eq:HKRDeformationRetract} restricts to a
        deformation retract
        \begin{equation}
            \begin{tikzcd}[column sep = large]
                \bigl( \Anti^\bullet \Secinfty(TM) \bigr)^\group{G}
                \arrow[r,"\hkr", shift left = 3pt]
                &\bigl( \HCdiff^{\bullet}(M)^\group{G},\delta \bigr)
                \arrow[l,"\hkrInv_\nabla", shift left = 3pt]
                \arrow[loop,
                out = -30,
                in = 30,
                distance = 30pt,
                start anchor = {[yshift = -7pt]east},
                end anchor = {[yshift = 7pt]east},
                "\hkrHomo"{swap}
                ]
            \end{tikzcd}
        \end{equation}
        or
        \begin{equation}
            \begin{tikzcd}[column sep = large]
                \bigl( \Anti^\bullet \Secinfty(TM) \bigr)^\liealg{g}
                \arrow[r,"\hkr", shift left = 3pt]
                &\bigl( \HCdiff^{\bullet}(M)^\liealg{g},\delta \bigr)
                \arrow[l,"\hkrInv_\nabla", shift left = 3pt]
                \arrow[loop,
                out = -30,
                in = 30,
                distance = 30pt,
                start anchor = {[yshift = -7pt]east},
                end anchor = {[yshift = 7pt]east},
                "\hkrHomo"{swap}
                ]
            \end{tikzcd}
            ,
        \end{equation}
        respectively.
        In particular, we obtain isomorphisms
        \begin{equation}
            \HCdiff^{\bullet}(M)^\group{G}
            \simeq \bigl( \Anti^\bullet \Secinfty(TM) \bigr)^\group{G}
        \end{equation}
        and
        \begin{equation}
            \HCdiff^{\bullet}(M)^\liealg{g}
            \simeq \bigl( \Anti^\bullet \Secinfty(TM) \bigr)^\liealg{g}
        \end{equation}
        of differential graded algebras.
    \end{theoremlist}
\end{theorem}
\begin{proof}
    From \autoref{prop:invarianSymbCalc} we know already that the
    symbol calculus $\Op$ is equivariant. Since in general all
    structure maps of the algebraic side are equivariant by
    \eqref{diag:GinvVanEst} and \eqref{diag:GinvVanEstLieAlg},
    respectively, the equivariance follows from the invariance of the
    covariant derivative. With this observation the second statement
    is clear.
\end{proof}

Note that \autoref{thm:HKR_inv} implies in particular that the
cohomology of the invariant Hochschild complex, the invariant
Hochschild cohomology and invariant multivector fields are all
isomorphic. This was so far only known for proper Lie group actions
with a proof based on averaging, see e.g. \cite{miaskiwskyi:2021a}.

We continue by using the observation of the equivariance of the global
symbol calculus and apply it to various invariant versions of HKR-like
theorems, which we discussed before.
\begin{theorem}[Invariant tangential HKR Theorem]
    \label{theorem:HKR_InvariantTangential}%
    Let $\iota\colon C\hookrightarrow M$ be a submanifold, let
    $\Phi\colon \group{G}\times M\to M$ (resp.
    $\phi\colon \liealg{g}\to \Secinfty(TM)$) be a Lie group
    (resp. Lie algebra) action on $M$ which restricts to $C$ and let
    $\nabla$ be a torsion-free $\group{G}$-invariant (resp. $\liealg{g}$-invariant)
    covariant derivative which is adapted to $C$.
    \begin{theoremlist}
    \item \label{item:TangentialStuffEquivariant} All structure maps
        in the deformation retract
        \eqref{eq:TangentialDeformationRetract} from
        \autoref{thm:HKR_Submanifold} are equivariant. In particular,
        the homotopy $\hkrHomo$ is equivariant.
    \item \label{item:InvariantTangentialDefRetract} The deformation
        retract \eqref{eq:TangentialDeformationRetract} restricts to a
        deformation retract
        \begin{equation}
            \label{eq:GroupInvariantTangentialRetract}
            \begin{tikzcd}[column sep = large]
                \bigl( \Anti^\bullet \Secinfty_C(TM) \bigr)^\group{G}
                \arrow[r,"\hkr", shift left = 3pt]
                &\bigl( \HCdiff^{\bullet}(M)_C^\group{G},\delta \bigr)
                \arrow[l,"\hkrInv_\nabla", shift left = 3pt]
                \arrow[loop,
                out = -30,
                in = 30,
                distance = 30pt,
                start anchor = {[yshift = -7pt]east},
                end anchor = {[yshift = 7pt]east},
                "\hkrHomo"{swap}
                ]
            \end{tikzcd}
        \end{equation}
        or
        \begin{equation}
            \label{eq:LieAlgebraInvariantTangentialRetract}
            \begin{tikzcd}[column sep = large]
                \bigl( \Anti^\bullet \Secinfty_C(TM) \bigr)^\liealg{g}
                \arrow[r,"\hkr", shift left = 3pt]
                &\bigl( \HCdiff^{\bullet}(M)_C^\liealg{g},\delta \bigr)
                \arrow[l,"\hkrInv_\nabla", shift left = 3pt]
                \arrow[loop,
                out = -30,
                in = 30,
                distance = 30pt,
                start anchor = {[yshift = -7pt]east},
                end anchor = {[yshift = 7pt]east},
                "\hkrHomo"{swap}
                ]
            \end{tikzcd}
        \end{equation}
        respectively.
        In particular, we obtain isomorphisms
        \begin{equation}
            \HCdiff^{\bullet}(M)_C^\group{G}
            \simeq \bigl( \Anti^\bullet \Secinfty_C(TM) \bigr)^\group{G}
        \end{equation}
        and
        \begin{equation}
            \HCdiff^{\bullet}(M)_C^\liealg{g}
            \simeq \bigl( \Anti^\bullet \Secinfty_C(TM) \bigr)^\liealg{g}
        \end{equation}
        of graded algebras.
    \end{theoremlist}
\end{theorem}

Another interesting case is to consider a principal bundle
$\pi\colon P\to M$ with structure group $\group{G}$. Note that the
construction of the adapted covariant derivative constructed in
\autoref{sec:AdaptedCovDer_Submersions} using a
$\group{G}$-invariant splitting of $TP$ results in a
$\group{G}$-invariant (adapted) covariant derivative. Using the symbol
calculus with respect to it, gives the following statement:
\begin{theorem}[Invariant HKR for principal bundle]
    \label{thm:HKR_PrincipalBundle}%
    Let $\pi\colon P\to M$ be a principal bundle with structure group
    $\group{G}$. Choose a torsion-free covariant derivative $\nabla$
    for $M$ and a principal connection for $P$ and use the
    corresponding $\group{G}$-invariant adapted covariant derivative
    $\nabla^P$ for $P$.
    \begin{theoremlist}
    \item \label{item:PrincipalBundleEquivariance} All structure maps
        in the deformation retract
        \eqref{eq:SurjectiveSubmersionRetract} from
        \autoref{thm:HKR_surjSubmersions} are
        $\group{G}$-equivariant. In particular, $\hkrHomo$ is
        $\group{G}$-equivariant.
    \item \label{item:InvariantHKRPrincipalBundle} The deformation
        retract \eqref{eq:SurjectiveSubmersionRetract} restricts to a
        deformation retract
        \begin{equation}
            \begin{tikzcd}[column sep = large]
                \Diffop_\ver(P)^\group{G}
                \arrow[r,"\hkr", shift left = 3pt]
                &\bigl( \HCdiff^{\bullet}(M,\Diffop(P))^\group{G},\delta \bigr)
                \arrow[l,"\hkrInv_\nabla", shift left = 3pt]
                \arrow[loop,
                out = -30,
                in = 30,
                distance = 30pt,
                start anchor = {[yshift = -7pt]east},
                end anchor = {[yshift = 7pt]east},
                "\hkrHomo"{swap}
                ]
            \end{tikzcd}
            .
        \end{equation}
    \item \label{item:HKRPrincipal} In particular, we have
        \begin{equation}
            \HHdiff^\bullet\big(M,\Diffop(P)\big)^\group{G}
            \simeq
            \Diffop_\ver(P)^\group{G}
        \end{equation}
        as graded $\Cinfty(M)$-modules concentrated in degree $0$.
    \end{theoremlist}
\end{theorem}
\begin{proof}
    Note that the Lie group action on $\Cinfty(M)$ is trivial, so it
    only acts on the values of the Hochschild cochains. Since the
    construction of the adapted covariant derivative on $P$ yields an
    invariant covariant derivative automatically, see
    \autoref{sec:AdaptedCovDer_Submersions}, the symbol calculus is
    $\group{G}$-equivariant. With the general equivariance of the
    algebraic part, the first part follows. Then the second and third
    statement is clear.
\end{proof}

The computation of this Hochschild cohomology was obtained in
\cite[Theorem~5.2]{bordemann.neumaier.waldmann.weiss:2010a} with a
more ad-hoc argument and without an equivariant global homotopy
$\hkrHomo$.

\begin{remark}[More invariant HKR theorems]
    \label{remark:InvariantHKRMachine}%
    It should be clear that various combinations of the previous HKR
    theorems have their invariant counterparts as soon as one has an
    invariant global symbol calculus. This is guaranteed as soon as
    one has invariant covariant derivatives, adapted to the needs of
    the previous geometric situations. It should be noted that
    invariant covariant derivatives are typically not too difficult to
    obtain, a proper group action is sufficient, but by far not
    necessary.
\end{remark}

\bookmarksetupnext{level=-1}
\appendix
\begin{appendices}

\section{Adapted Covariant Derivatives}
\label{sec:AdaptedCovariantDerivatives}

Two covariant derivatives $\nabla^M$ and $\nabla^N$ on the manifolds
$M$ and $N$ are called adapted with respect to a smooth map
$\phi\colon N\to M$, if whenever we have two pairs of $\phi$-related
vector fields $X\sim_\phi X'$ and $Y\sim_\phi Y'$ then also
\begin{equation}
    \nabla^N_X Y \sim_\phi \nabla^M_{X'}Y'.
\end{equation}
Note that if one additionally considers a $1$-form
$\alpha\in \Omega^1(M)$, then one gets
\begin{equation}
\begin{split}
    \phi^*(\nabla^M_{X'}\alpha)(Y)
    &=
    \phi^*[\nabla^M_{X'}\alpha(Y')]
    =
    X\phi^*(\alpha(Y'))-\phi^*(\alpha(\nabla^M_{X'}Y'))
    \\
    &=
    X\phi^*\alpha(Y)- \phi^*\alpha(\nabla^N_X Y)
    =
    \nabla^N_X\phi^*\alpha(Y).
\end{split}
\end{equation}
This means that $\phi^*\nabla^M_{X'}\alpha$ and
$\nabla^N_X\phi^*\alpha$ coincide on vector fields that are related to
a vector field on $M$. Using adapted charts of $\phi$ around regular
points, one can show that this is actually enough to show that they
coincide as differential forms.

Adapted covariant derivatives do not exist for every smooth map, but
there are examples of maps, where one can guarantee their existence
and the next two subsections provide a construction for two classes of
examples which are of interest for this note.

\subsection{Surjective Submersions}
\label{sec:AdaptedCovDer_Submersions}

Let $\pr\colon P \to M$ be a surjective submersion. As throughout the
paper, we fix a torsion-free covariant derivative $\nabla$ on $M$. The
aim is now to construct a nice covariant derivative on $P$.

Let $\Ver(P) = \ker T\pr \subseteq TP$ denote the vertical
subbundle. The map $\pr$ being submersive shows that $\Ver(P)$ is
indeed a smooth subbundle which in addition is involutive. In fact,
the connected components of the fibers $\pr^{-1}(\{p\}) \subseteq P$
with $p \in M$ provide the integral submanifolds of $\Ver(P)$.

We choose now a connection $\mathcal{P} \in \Secinfty(\End(TP))$ on
$P$, i.e. an idempotent $\mathcal{P}^2 = \mathcal{P}$ with
$\image \mathcal{P} = \Ver(P)$ pointwise. This gives a horizontal
subbundle $\Hor(P) = \ker \mathcal{P}$ such that
$TP = \Ver(P) \oplus \Hor(P)$. Moreover, we have an induced horizontal
lift of vector fields $X \in \Secinfty(TM)$ to horizontal vector
fields $X^\hor \in \Secinfty(\Hor(P)) \subseteq \Secinfty(TP)$
uniquely characterized by $T\pr \circ X^\hor = X \circ \pr$,
i.e. $X^\hor \sim_{\pr} X$.

Since $\Ver(P)$ is involutive, we find a partially defined covariant
derivative
\begin{equation}
    \label{eq:PartialCovDefForVer}
    \nabla^{\Ver}\colon
    \Secinfty(\Ver(P)) \times \Secinfty(\Ver(P))
    \to
    \Secinfty(\Ver(P))
\end{equation}
satisfying the Leibniz rule and which is torsion-free. Indeed,
starting with any torsion-free covariant derivative $\tilde{\nabla}$
on $P$ the definition
$\nabla^{\Ver}_V W = \mathcal{P} \tilde{\nabla}_V W$ for
$V, W \in \Secinfty(\Ver(P))$ will do the job.

Next we make use of the decomposition $TP = \Ver(P) \oplus \Hor(P)$ to
extend $\nabla^{\Ver}$ back to a covariant derivative for all tangent
vector fields on $P$. Since every tangent vector field is a sum of a
vertical and a horizontal one and since the horizontal vector fields
are (locally) $\Cinfty(P)$-spanned by the horizontal lifts, we need to
specify the covariant derivative $\nabla^P$ only on the combinations
$\nabla^P_{X^\hor} Y^\hor$, $\nabla^P_{X^\hor} V$,
$\nabla^P_V X^\hor$, and of course on $\nabla^P_V W$ where
$V, W \in \Secinfty(\Ver(P))$ and $X, Y \in \Secinfty(TM)$. We require
\begin{gather}
    \label{eq:nablaPVW}
    \nabla^P_V W
    =
    \nabla^{\Ver}_V W,
    \\
    \label{eq:nablaPXhorV}
    \nabla^P_{X^\hor} V
    =
    [X^\hor, V],
    \\
    \label{eq:nablaPVXhor}
    \nabla^P_V X^\hor
    =
    0,
    \\
    \label{eq:nablaPXhorYhor}
    \nabla^P_{X^\hor} Y^\hor
    =
    \big(\nabla_X Y\big)^\hor
    +
    \frac{1}{2}
    \big(
    [X^\hor, Y^\hor]
    -
    [X, Y]^\hor
    \big),
\end{gather}
and extend this by the necessary Leibniz rules to all horizontal
vector fields. A quick check shows that this is possible and indeed
yields a torsion-free covariant derivative $\nabla^P$. Since
$[X^\hor, Y^\hor] - [X, Y]^\hor$ is vertical, we have the
compatibility
\begin{equation}
    \label{eq:TprnablaPnablapr}
    T\pr \big(\nabla^P_{X^\hor} Y^\hor\big)
    =
    (\nabla_X Y) \circ \pr
\end{equation}
for all $X, Y \in \Secinfty(TM)$. A torsion-free covariant derivative
$\nabla^P$ on $P$ with $\nabla^P_Z V \in \Secinfty(\Ver(P))$ for all
vertical $V \in \Secinfty(\Ver(P))$ and all $Z \in \Secinfty(TP)$ is
then called \emph{adapted} to the covariant derivative $\nabla$ on $M$
and the connection $\mathcal{P}$ if \eqref{eq:TprnablaPnablapr}
holds. The above construction shows that one always has adapted
covariant derivatives for a given $\nabla$.

A particular case of interest is a $\group{G}$-principal fiber bundle
$\pr\colon P \to M$. Here we can choose a $\group{G}$-invariant
horizontal complement $\Hor(P)$, i.e. a principal connection. Then
$TP = \Ver(P) \oplus \Hor(P)$ splits into two $\group{G}$-equivariant
subbundles, since the vertical subbundle is $\group{G}$-equivariant
for the principal $\group{G}$-action on $P$ anyway. Moreover, in this
case $\Ver(P) \cong P \times \liealg{g}$ is a trivial bundle with
trivialization given by the fundamental vector fields. We can now use
as partially defined covariant derivative
\eqref{eq:PartialCovDefForVer} the $\frac{1}{2}$-commutator
connection. This is uniquely determined by
\begin{equation}
    \label{eq:HalfCommutatorConnection}
    \nabla^{\Ver}_{\xi_P} \eta_P
    =
    \frac{1}{2} [\xi, \eta]_P
\end{equation}
for $\xi, \eta \in \liealg{g}$ where $\xi_P \in \Secinfty(\Ver(P))$
denotes the fundamental vector field corresponding to
$\xi \in \liealg{g}$. Then $\nabla^{\Ver}$ is a $\group{G}$-invariant
covariant derivative which ultimately yields a $\group{G}$-invariant
covariant derivative $\nabla^P$ by the above construction.

\subsection{Submanifolds}
\label{sec:AdaptedCovDer_Submanifolds}

Let $\iota\colon C\hookrightarrow M$ be an injective immersion.
Recall that a vector field $X\in \Secinfty(TM)$ is called tangent to
$C$, if there exists a vector field $\tilde{X}\in \Secinfty(TC)$ such
that for all functions $f\in \Cinfty(C)$, we have
\begin{equation*}
    \iota^*X(f)
    =
    \tilde X (\iota^*f),
\end{equation*}
i.e. $\tilde{X} \sim_\iota X$.  Note that tangential vector fields
form a Lie subalgebra of all vector fields. We denote these vector
fields by $\Secinfty_C(TM)$.  The aim is now to construct an adapted
covariant derivative $\nabla$ on $M$, i.e.
\begin{equation*}
    \nabla_X Y \in \Secinfty_C(TM)
\end{equation*}
for all $X,Y\in \Secinfty_C(TM)$.  Note that this is exactly
equivalent to the definition of having a pair of covariant derivatives
such that it preserves related vector fields, since we can induce a
covariant derivative on $C$ using tangential vector fields.

Even though adapted covariant derivatives may exist for certain
injective immersions, they may fail to exist in general.
Nevertheless, we assume now that $\iota\colon C \hookrightarrow M$ is
an embedding and choose a complementary vector bundle
$p \colon TC^\perp \to C$ of $TC$ inside $TM\at{C}$, i.e.
$TM\at{C} = TC \oplus TC^\perp$.  Moreover, choose a tubular
neighbourhood
\begin{equation}
    \psi \colon TC^\perp \to U \subseteq_{\mathrm{open}} M.
\end{equation}
Finally, we choose a torsion- free covariant derivative $\nabla^C$ on
$C$ and a splitting
\begin{equation*}
    T(TC^\perp) \cong \Hor \oplus \ker Tp
\end{equation*}
with $\Hor\at{C}=T\iota(TC)$.

Note that for a vector field $X\in \Secinfty(TC)$ the horizontal lift
$X^\hor \in\Secinfty(T(TC^\perp))$ of it fulfils
\begin{equation*}
    \iota^* X^\hor(f)
    =
    X(\iota^*f).
\end{equation*}
and thus $X^\hor$ is tangential to the $C$.

By \autoref{sec:AdaptedCovDer_Submersions} we can find a covariant
derivative $\nabla$ adapted to the surjective submersion
$p \colon TC^\perp \to C$.  By the concrete formula of this covariant
derivative and using $\Hor\at{C}=T\iota(TC)$ on can see that it is
adapted to $C \hookrightarrow TC^\perp$.
\begin{lemma}
    \label{lemma:AdaptedCovariantDerivativeSubmanifold}%
    Let $C \hookrightarrow M$ be a closed embedded submanifold, then
    there exists a covariant derivative adapted to $C$.
\end{lemma}
\begin{proof}
    Let us choose a tubular neighbourhood and construct the covariant
    derivative $\nabla$ for a surjective submersion as in
    \autoref{sec:AdaptedCovDer_Submersions} from a covariant derivative
    $\nabla^C$ on $C$ and a connection on $p\colon TC^\perp \to C$
    described above.  Note that we get for horizontal lifts of vector
    fields $X,Y\in \Secinfty(TC)$ and for every function
    $f\in \Cinfty(C)$
    \begin{align*}
        \iota^*\nabla_{X^\hor}Y^\hor (f)
        &=
        \iota^*((\nabla^C_X Y)^\hor f +
        \frac{1}{2}([X^\hor,Y^\hor]-[X,Y]^\hor)f)
        \\
        &=
        \nabla^C_X Y (\iota^*f).
    \end{align*}
    If arbitrary vector fields $X,Y\in \Secinfty(T(TC^\perp))$ are
    tangential to $C$ with corresponding vector fields
    $\tilde{X},\tilde{Y}\in \Secinfty(TC)$, then $X-\tilde{X}^\hor$
    and $Y-\tilde{Y}^\hor$ are vanishing on $C$ and therefore
    \begin{equation*}
        \iota^*(\nabla_X Y (f))
        =
        \iota^* (\nabla_{\tilde{X}^\hor} Y (f))
        =
        \iota^* (\nabla_{\tilde{X}^\hor} \tilde{Y}^\hor (f)),
    \end{equation*}
    where the first equality follows from the fact that
    $X-\tilde{X}^\hor$ vanishes along $C$ and the covariant derivative
    is a function linear in the first argument, and the second equality
    follows since $Y-\tilde{Y}^\hor$ vanishes along $C$ and
    $\tilde{X}^\hor$ is tangential to $C$ by the Leibniz rule in the
    second argument of the covariant derivative.  This shows that the
    connection is adapted to $C$ inside the tubular neighbourhood.  Since
    $C$ is closed in $M$, we can find a function $f\in \Cinfty(M)$,
    such that $f\at{C}=1$ and $f\at{M\setminus U}=0$.  Now we choose
    an arbitrary covariant derivative $\tilde{\nabla}$ on $M$ and
    define
    \begin{equation*}
        \nabla
        =
        \tilde{\nabla}+ f(\nabla^U-\tilde{\nabla}),
    \end{equation*}
    where we interpret
    $\nabla^U-\tilde{\nabla}\in \Secinfty(T^*U\tensor\End(TU))$ as
    tensor field, extended to $M$ by multiplying it by $f$. It is now
    easy to show that this covariant derivative does the job.
\end{proof}

\section{Homotopy Retracts and their Perturbations}
\label{sec:HomologicalAlgebra}

We collect some well-known results about homotopy retracts and the
homological perturbation lemma.  In the following we suppose that
$\ring{R}$ is a commutative ring.  In the main part we mostly use the
special case of $\ring{R}$ being some algebra of smooth functions on
a manifold.

\subsection{Homotopy Retracts}
\label{sec:HomotopyRetracts}

We introduce homotopy retracts and their morphisms as well as some
first constructions.  Note that what we call homotopy retracts
sometimes carries different names in the literature, e.g. homotopy
equivalence data (HE data) in \cite{crainic:2004a:pre}.
\begin{definition}[Homotopy retract]
    \label{def:HomotopyRetract}%
    Let $C^\bullet = \bigoplus_{n=0}^\infty C^n$ and
    $D^\bullet = \bigoplus_{n=0}^\infty D^n$ be cochain complexes of
    $\ring{R}$-modules with differentials
    $\D_C \colon C^{\bullet} \to C^{\bullet+1}$ and
    $\D_D \colon D^{\bullet} \to D^{\bullet+1}$, respectively.  A
    triple $(\imap,\pmap,\hmap)$ consisting of two cochain morphisms
    $\imap \colon C^\bullet \to D^\bullet$ and
    $\pmap \colon D^\bullet \to C^\bullet$ and a map
    $\hmap \colon D^\bullet \to D^{\bullet-1}$ fulfilling
    \begin{equation}
        \hmap \D_D + \D_D \hmap
        =
        \id - \imap \pmap
    \end{equation}
    is called \emph{homotopy retract}, see
    \autoref{fig:homotopyRetract}.
\end{definition}
\begin{figure}
    \centering
    \includegraphics{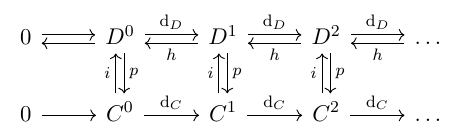}
    \caption{Homotopy retract $(C^\bullet,\D_C)$ and $(D^\bullet,\D_D)$.}
    \label{fig:homotopyRetract}
\end{figure}

We will also depict a homotopy retract as follows:
\begin{equation}
    \begin{tikzcd}
        (C^\bullet,\D_C)
        \arrow[r,"\imap", shift left = 3pt]
        &(D^\bullet,\D_D)
        \arrow[l,"\pmap", shift left = 3pt]
        \arrow[loop,
        out = -30,
        in = 30,
        distance = 30pt,
        start anchor = {[yshift = -7pt]east},
        end anchor = {[yshift = 7pt]east},
        "\hmap"{swap}
        ]
    \end{tikzcd}
\end{equation}
An immediate computation then gives the following composition
property:
\begin{proposition}[Composition of homotopy retracts]
    \label{prop:CompHomotopyRetracts}%
    Let
    \begin{equation}
        \begin{tikzcd}
            (C^\bullet,\D_C)
            \arrow[r,"\imap_1", shift left = 3pt]
            &(D^\bullet,\D_D)
            \arrow[l,"\pmap_1", shift left = 3pt]
            \arrow[loop,
            out = -30,
            in = 30,
            distance = 30pt,
            start anchor = {[yshift = -7pt]east},
            end anchor = {[yshift = 7pt]east},
            "\hmap_1"{swap}]
        \end{tikzcd}
    \end{equation}
    and
    \begin{equation}
        \begin{tikzcd}
            (D^\bullet,\D_D)
            \arrow[r,"\imap_2", shift left = 3pt]
            &(E^\bullet,\D_E)
            \arrow[l,"\pmap_2", shift left = 3pt]
            \arrow[loop,
            out = -30,
            in = 30,
            distance = 30pt,
            start anchor = {[yshift = -7pt]east},
            end anchor = {[yshift = 7pt]east},
            "\hmap_2"{swap}]
        \end{tikzcd}
    \end{equation}
    be homotopy retracts.  Then
    \begin{equation}
        \begin{tikzcd}
            (C^\bullet,\D_C)
            \arrow[r,"\imap", shift left = 3pt]
            &(E^\bullet,\D_E)
            \arrow[l,"\pmap", shift left = 3pt]
            \arrow[loop,
            out = -30,
            in = 30,
            distance = 30pt,
            start anchor = {[yshift = -7pt]east},
            end anchor = {[yshift = 7pt]east},
            "\hmap"{swap}]
        \end{tikzcd}
    \end{equation}
    with
    \begin{equation}
        \imap = \imap_2  \imap_1,
        \qquad
        \pmap = \pmap_1  \pmap_2,
        \qquad
        \hmap = \hmap_2 + \imap_2  \hmap_1  \pmap_2
    \end{equation}
    is a homotopy retract.
\end{proposition}
\begin{definition}[Composition of homotopy retracts]
    \label{definition:CompositionRetracts}%
    Given homotopy retracts $(\imap_1,\pmap_1,\hmap_1)$ and
    $(\imap_2,\pmap_2,\hmap_2)$ as in
    \autoref{prop:CompHomotopyRetracts} we denote by
    \begin{equation}
        (\imap_2,\pmap_2,\hmap_2) \circ (\imap_1,\pmap_1,\hmap_1)
        \coloneqq (
        \imap_2 \imap_1,\,
        \pmap_1 \pmap_2,\,
        \hmap_2 + \imap_2 \hmap_1 \pmap_2
        )
    \end{equation}
    their composition.
\end{definition}
It is easy to see that the composition of homotopy retracts is in fact
associative.

\begin{definition}[Morphism of homotopy retracts]
    \label{def:MorphismHomotopyRetract}%
    A morphism from a homotopy retract
    \begin{equation}
        \begin{tikzcd}
            ( C^\bullet,\D_C )
            \arrow[r,"\imap", shift left = 3pt]
            &( D^\bullet,\D_D )
            \arrow[l,"\pmap", shift left = 3pt]
            \arrow[loop,
            out = -30,
            in = 30,
            distance = 30pt,
            start anchor = {[yshift = -7pt]east},
            end anchor = {[yshift = 7pt]east},
            "\hmap"{swap}
            ]
        \end{tikzcd}
    \end{equation}
    to a homotopy retract
    \begin{equation}
        \begin{tikzcd}
            ( C'^\bullet,\D_{C'} )
            \arrow[r,"\imap'", shift left = 3pt]
            &(D'^\bullet,\D_{D'} )
            \arrow[l,"\pmap'", shift left = 3pt]
            \arrow[loop,
            out = -30,
            in = 30,
            distance = 30pt,
            start anchor = {[yshift = -7pt]east},
            end anchor = {[yshift = 7pt]east},
            "\hmap'"{swap}
            ]
        \end{tikzcd}
    \end{equation}
    is given by a morphism $\Phi \colon D^\bullet \to D'^\bullet$ of
    complexes such that $\Phi \hmap = \hmap'\Phi$.
\end{definition}

Note that setting $\Psi \coloneqq \pmap'\Phi\imap$ the two squares in
\begin{equation}
    \begin{tikzcd}
        C^\bullet
        \arrow[r,"\imap", shift left = 3pt]
        \arrow[d,"\Psi"{swap}]
        & D^\bullet
        \arrow[l,"\pmap", shift left = 3pt]
        \arrow[d,"\Phi"]\\
        C'^\bullet
        \arrow[r,"\imap'", shift left = 3pt]
        & D'^\bullet
        \arrow[l,"\pmap'", shift left = 3pt] \\
    \end{tikzcd}
\end{equation}
commute, i.e. we have $\imap'\Psi = \Phi \imap$ and
$\Psi \pmap = \pmap'\Phi$.  Moreover, the above notion of morphism of
homotopy retracts can easily generalized to morphisms
$\Phi \colon D^\bullet \to D'^\bullet$ along a ring morphism
$\vartheta \colon \ring{R} \to \ring{R}'$.
\begin{proposition}
    \label{prop:SumTensorHomotopies}%
    Let $C^\bullet$, $D^\bullet$, $E^\bullet$ and $F^\bullet$ be
    complexes together with homotopy retracts $(\imap, \pmap, \hmap)$
    and $(\jmap, \qmap, \kmap)$ between $C^\bullet$ and $D^\bullet$
    and between $E^\bullet$ and $F^\bullet$, respectively.
    \begin{propositionlist}
    \item \label{item:SumHomotopyRetracts} Then the direct sum
        \begin{equation}
            \begin{tikzcd}[column sep = large]
                \bigl( (C \oplus E)^\bullet,\D_{C} + \D_E \bigr)
                \arrow[r,"\imap + \jmap", shift left = 3pt]
                &\bigl( (D \oplus F)^\bullet,\D_{D} + \D_{F} \bigr)
                \arrow[l,"\pmap + \qmap", shift left = 3pt]
                \arrow[loop,
                out = -30,
                in = 30,
                distance = 30pt,
                start anchor = {[yshift = -7pt]east},
                end anchor = {[yshift = 7pt]east},
                "\hmap + \kmap"{swap}
                ]
            \end{tikzcd}
        \end{equation}
        is a homotopy retract.
    \item \label{item:TensorHomotopyRetracts} Then the tensor product
        \begin{equation}
            \begin{tikzcd}[column sep = large]
                \bigl( (C \tensor E)^\bullet,\D_{C \tensor E} \bigr)
                \arrow[r,"\imap \tensor \jmap", shift left = 3pt]
                &\bigl( (D \tensor F)^\bullet,\D_{D \tensor F}\bigr)
                \arrow[l,"\pmap \tensor \qmap", shift left = 3pt]
                \arrow[loop,
                out = -30,
                in = 30,
                distance = 30pt,
                start anchor = {[yshift = -7pt]east},
                end anchor = {[yshift = 7pt]east},
                "\hmap_{D \tensor F}"{swap}
                ]
            \end{tikzcd}
        \end{equation}
        with
        $\hmap_{D \tensor F}(x \tensor y) \coloneqq \hmap(x) \tensor y
        + (-1)^k \pmap \imap (x) \tensor \kmap (y)$ for all
        $x \in D^k$ and $y \in F^\bullet$, is a homotopy retract.
    \end{propositionlist}
\end{proposition}

\subsection{Homological Perturbation}
\label{sec:HomologicalPerturbation}

We recall the classical homological perturbation lemma.  More details
and proofs can be found e.g. in \cite{crainic:2004a:pre}.  Note that
we have chosen different sign conventions here.
\begin{definition}[Perturbation]
    \label{def:Perturbation}%
    Consider a homotopy retract
    \begin{equation}
        \begin{tikzcd}
            (C^\bullet,\D_C)
            \arrow[r,"\imap", shift left = 3pt]
            &(D^\bullet,\D_D)
            \arrow[l,"\pmap", shift left = 3pt]
            \arrow[loop,
            out = -30,
            in = 30,
            distance = 30pt,
            start anchor = {[yshift = -7pt]east},
            end anchor = {[yshift = 7pt]east},
            "\hmap"{swap}
            ]
        \end{tikzcd}
        .
    \end{equation}
    \begin{definitionlist}
    \item \label{item:Perturbation} A $\ring{R}$-linear map
        $\bmap \colon D^\bullet \to D^{\bullet+1}$ such that
        $(\D_D + \bmap)^2 = 0$ is called a \emph{perturbation} of
        $(\imap,\pmap,\hmap)$.
    \item \label{item:SmallPerturbation} A perturbation $\bmap$ is
        called \emph{small} if $\id + \bmap \hmap$ is invertible.
    \item \label{item:LocallyNilpotentPerturbation} A perturbation
        $\bmap$ is called \emph{locally nilpotent} if for every
        $a \in D^\bullet$ there exists $n \in \mathbb{N}_0$ such that
        $(\bmap \hmap)^n(a) = 0$.
    \end{definitionlist}
\end{definition}

Every locally nilpotent perturbation $b$ is small with
\begin{equation}
    \label{eq:LocallyNilpotentInverse}
    (\id + \bmap\hmap)^{-1}
    =
    \sum_{n=0}^{\infty} (-1)^n (\bmap\hmap)^n.
\end{equation}
The proof of the following perturbation lemma can be found in
\cite[\nopp 2.4]{crainic:2004a:pre}.
\begin{proposition}[Homological perturbation]
    \label{prop:HomologicalPerturbation}%
    Let a homotopy retract
    \begin{equation}
        \begin{tikzcd}
            (C^\bullet,\D_C)
            \arrow[r,"\imap", shift left = 3pt]
            &(D^\bullet,\D_D)
            \arrow[l,"\pmap", shift left = 3pt]
            \arrow[loop,
            out = -30,
            in = 30,
            distance = 30pt,
            start anchor = {[yshift = -7pt]east},
            end anchor = {[yshift = 7pt]east},
            "\hmap"{swap}
            ]
        \end{tikzcd}
    \end{equation}
    together with a small perturbation $\bmap$ be given.  Then
    \begin{equation}
        \begin{tikzcd}
            (C^\bullet,\Double_C)
            \arrow[r,"\Imap", shift left = 3pt]
            &(D^\bullet,\D_D + \bmap)
            \arrow[l,"\Pmap", shift left = 3pt]
            \arrow[loop,
            out = -30,
            in = 30,
            distance = 30pt,
            start anchor = {[yshift = -7pt]east},
            end anchor = {[yshift = 7pt]east},
            "\Hmap"{swap}
            ]
        \end{tikzcd}
    \end{equation}
    is a homotopy retract with
    \begin{alignat}{2}
        \Double_C &\coloneqq
        \D_C
        + \pmap(\id + \bmap\hmap)^{-1}\bmap \imap,
        \qquad\qquad
        &\Imap &\coloneqq
        (\id + \hmap \bmap)^{-1}\imap,
        \\
        \Hmap &\coloneqq
        (\id + \hmap\bmap)^{-1}\hmap,
        &\Pmap &\coloneqq
        \pmap(\id + \bmap\hmap)^{-1}.
    \end{alignat}
\end{proposition}
\begin{corollary}[Morphism of perturbed retracts]
    \label{cor:MorphismOfPerturbations}%
    Let $\Phi \colon D^\bullet \to D'^\bullet$ be a morphism of
    homotopy retracts $(\imap,\pmap,\hmap)$ and
    $(\imap',\pmap',\hmap')$.  Moreover, let $\bmap$ be a small
    perturbation of $(\imap,\pmap,\hmap)$ and let $\bmap'$ be small
    perturbation of $(\imap',\pmap',\hmap')$.  If $\Phi$ satisfies
    $\bmap'\Phi = \Phi \bmap$, then $\Phi$ is a morphism between the
    perturbed homotopy retracts, too.
\end{corollary}
This corollary still holds if we consider morphisms of homotopy
retracts over different rings of scalars along a ring homomorphism
between the scalars.  Often homotopy retracts satisfy additional
properties.  For us the following notions will be important.

\begin{definition}[Special deformation retract]
    \label{def:SpecialDeformationRetract}%
    Let
    \begin{equation}
        \begin{tikzcd}
            (C^\bullet,\D_C)
            \arrow[r,"\imap", shift left = 3pt]
            &(D^\bullet,\D_D)
            \arrow[l,"\pmap", shift left = 3pt]
            \arrow[loop,
            out = -30,
            in = 30,
            distance = 30pt,
            start anchor = {[yshift = -7pt]east},
            end anchor = {[yshift = 7pt]east},
            "\hmap"{swap}]
        \end{tikzcd}
    \end{equation}
    be a homotopy retract.
    \begin{definitionlist}
    \item \label{item:DeformationRetract} If additionally
        \begin{equation}
            \label{eq:DeformationRetract}
            \pmap\imap = \id
        \end{equation}
        holds, the above homotopy retract is called \emph{deformation
          retract}.
    \item \label{item:SpecialDeformationRetract} A deformation retract
        such that additionally
        \begin{equation}
            \hmap^2 = 0,
            \qquad
            \pmap \hmap = 0,
            \quad
            \textrm{ and }
            \quad
            \hmap \imap = 0
        \end{equation}
        hold is called \emph{special deformation retract}.
    \end{definitionlist}
\end{definition}

Note that \eqref{eq:DeformationRetract} is equivalent to
\begin{equation}
    \begin{tikzcd}
        (D^\bullet,\D_D)
        \arrow[r,"\pmap", shift left = 3pt]
        &(C^\bullet,\D_C)
        \arrow[l,"\imap", shift left = 3pt]
        \arrow[loop,
        out = -30,
        in = 30,
        distance = 30pt,
        start anchor = {[yshift = -7pt]east},
        end anchor = {[yshift = 7pt]east},
        "0"{swap}
        ]
    \end{tikzcd}
\end{equation}
being a homotopy retract.
Perturbations of deformation retracts will in general not be deformation retracts.
Nevertheless, special deformation retracts are preserved under perturbation
\cite[\nopp 3.1]{crainic:2004a:pre}:

\begin{corollary}[Perturbation of special deformation retract]
    \label{cor:SpecialDeformationRetract}%
    Let a special deformation retract
    \begin{equation}
        \begin{tikzcd}
            (C^\bullet,\D_C)
            \arrow[r,"\imap", shift left = 3pt]
            &(D^\bullet,\D_D)
            \arrow[l,"\pmap", shift left = 3pt]
            \arrow[loop,
            out = -30,
            in = 30,
            distance = 30pt,
            start anchor = {[yshift = -7pt]east},
            end anchor = {[yshift = 7pt]east},
            "\hmap"{swap}
            ]
        \end{tikzcd}
    \end{equation}
    together with a locally nilpotent perturbation $\bmap$ be given.
    Then
    \begin{equation}
        \begin{tikzcd}
            (C^\bullet,\Double_C)
            \arrow[r,"\Imap", shift left = 3pt]
            &(D^\bullet,\D_D + \bmap)
            \arrow[l,"\Pmap", shift left = 3pt]
            \arrow[loop,
            out = -30,
            in = 30,
            distance = 30pt,
            start anchor = {[yshift = -7pt]east},
            end anchor = {[yshift = 7pt]east},
            "\Hmap"{swap}
            ]
        \end{tikzcd}
    \end{equation}
    as in \autoref{prop:HomologicalPerturbation} is a special
    deformation retract.
\end{corollary}

Let us finally discuss special cases of the homological perturbation
lemma that will be used in the main part.
\begin{corollary}[Perturbation lemma I]
    \label{cor:PerturbationLemmaI}%
    Let $(D^\bullet,\D)$ be a cochain complex and consider a
    $\ring{R}$-module $C$ as a complex concentrated in degree zero.
    Moreover, let
    \begin{equation}
        \begin{tikzcd}
            C
            \arrow[r,"\imap", shift left = 3pt]
            &(D^\bullet,\D)
            \arrow[l,"\pmap", shift left = 3pt]
            \arrow[loop,
            out = -30,
            in = 30,
            distance = 30pt,
            start anchor = {[yshift = -7pt]east},
            end anchor = {[yshift = 7pt]east},
            "\hmap"{swap}
            ]
        \end{tikzcd}
    \end{equation}
    be a special deformation retract.  Here we consider $\imap$ and
    $\pmap$ to be $0$ outside of their domains.  If $\bmap$ is a
    locally nilpotent perturbation, then
    \begin{equation}
        \begin{tikzcd}
            C
            \arrow[r,"\Imap", shift left = 3pt]
            &(D^\bullet,\D + \bmap)
            \arrow[l,"\pmap", shift left = 3pt]
            \arrow[loop,
            out = -30,
            in = 30,
            distance = 30pt,
            start anchor = {[yshift = -7pt]east},
            end anchor = {[yshift = 7pt]east},
            "\Hmap"{swap}
            ]
        \end{tikzcd}
    \end{equation}
    is a special deformation retract with
    \begin{equation}
        \Imap \coloneqq (\id + \hmap \bmap)^{-1}\imap,
        \quad
        \textrm{and}
        \quad
        \Hmap \coloneqq (\id + \hmap \bmap)^{-1}\hmap.
    \end{equation}
\end{corollary}

In the following we denote by $D^{\bullet,\bullet}$ a double complex
with $\delta \colon D^{k,\bullet} \to D^{k+1,\bullet}$ the vertical
differential and
$\partial \colon D^{\bullet,\ell} \to D^{\bullet,\ell}$ the horizontal
differential.  Moreover, we denote by $\D = \delta + \partial$ the
total differential on the total complex
\begin{equation}
    D^m \coloneqq \bigoplus_{k + \ell = m} D^{k,\ell}.
\end{equation}
\begin{corollary}[Perturbation lemma II]
    \label{cor:PerturbationLemmaII}%
    Let $(D^{\bullet,\bullet},\delta,\partial)$ be a double complex.
    \begin{corollarylist}
    \item Assume that for each row there is a deformation retract
        \begin{equation}
            \begin{tikzcd}
                C^k
                \arrow[r,"\jmap", shift left = 3pt]
                &(D^{k,\bullet},\partial)
                \arrow[l,"\qmap", shift left = 3pt]
                \arrow[loop,
                out = -30,
                in = 30,
                distance = 30pt,
                start anchor = {[yshift = -7pt]east},
                end anchor = {[yshift = 7pt]east},
                "\kmap"{swap}
                ]
            \end{tikzcd}
        \end{equation}
        given.  Then
        \begin{equation}
            \begin{tikzcd}
                (C^\bullet,\delta_C)
                \arrow[r,"\jmap", shift left = 3pt]
                &(D^\bullet,\partial + \delta)
                \arrow[l,"\Qmap", shift left = 3pt]
                \arrow[loop,
                out = -30,
                in = 30,
                distance = 30pt,
                start anchor = {[yshift = -7pt]east},
                end anchor = {[yshift = 7pt]east},
                "\Kmap"{swap}
                ]
            \end{tikzcd}
        \end{equation}
        is a deformation retract with
        \begin{equation}
            \delta_C = \qmap \delta \jmap,
            \qquad
            \Qmap = \qmap(\id + \delta \kmap)^{-1},
            \qquad
            \textrm{and}
            \qquad
            \Kmap = \kmap(\id + \delta \kmap)^{-1}.
        \end{equation}
    \item Assume that for each column there is a deformation retract
        \begin{equation}
            \begin{tikzcd}
                C^\ell
                \arrow[r,"\imap", shift left = 3pt]
                &(D^{\bullet,\ell},\delta)
                \arrow[l,"\pmap", shift left = 3pt]
                \arrow[loop,
                out = -30,
                in = 30,
                distance = 30pt,
                start anchor = {[yshift = -7pt]east},
                end anchor = {[yshift = 7pt]east},
                "\hmap"{swap}
                ]
            \end{tikzcd}
        \end{equation}
        given.  Then
        \begin{equation}
            \begin{tikzcd}
                (C^\bullet,\partial_C)
                \arrow[r,"\imap", shift left = 3pt]
                &(D^\bullet,\delta + \partial)
                \arrow[l,"\Pmap", shift left = 3pt]
                \arrow[loop,
                out = -30,
                in = 30,
                distance = 30pt,
                start anchor = {[yshift = -7pt]east},
                end anchor = {[yshift = 7pt]east},
                "\Hmap"{swap}
                ]
            \end{tikzcd}
        \end{equation}
        is a deformation retract with
        \begin{equation}
            \partial_C = \pmap \partial \imap,
            \qquad
            \Pmap = \pmap(\id + \partial \hmap)^{-1},
            \qquad
            \textrm{and}
            \qquad
            \Hmap = \hmap(\id + \partial \hmap)^{-1}.
        \end{equation}
    \end{corollarylist}
\end{corollary}
In both cases, if we start with a special deformation retract we
obtain a special deformation retract.

\end{appendices}

\bookmarksetup{startatroot}
{
  \footnotesize
  \printbibliography[heading=bibintoc]
}

\ifdraft{\clearpage}{}
\ifdraft{\phantomsection}{}
\ifdraft{\addcontentsline{toc}{section}{List of Corrections}}{}
\ifdraft{\listoffixmes}{}

\end{document}